\documentclass[12pt,reqno, a4paper]{amsart}

\usepackage{amsmath,amsthm}
\usepackage{amssymb}
\usepackage{mathtools}
\numberwithin{equation}{section}

\usepackage{latexsym}
\usepackage{hyperref}

\usepackage{cite}
\usepackage{enumitem}

\usepackage[margin=2.5cm]{geometry}
\usepackage{eucal}

\usepackage{tikz-cd}
\usepackage{tikz-3dplot}
\usepackage{pgfplots}
\pgfplotsset{compat=1.18}
\usetikzlibrary{positioning}

\usepackage{todonotes}

\mathcode`\:=\string"603A

\theoremstyle{plain}
\newtheorem{Theorem}{Theorem}[section]
\newtheorem{Proposition}[Theorem]{Proposition}
\newtheorem{Lemma}[Theorem]{Lemma}
\newtheorem{Corollary}[Theorem]{Corollary}

\theoremstyle{definition}
\newtheorem{Remark}[Theorem]{Remark}
\newtheorem{Definition}[Theorem]{Definition}
\newtheorem{Example}[Theorem]{Example}
\newtheorem{Terminology}[Theorem]{Terminology}

\newtheorem{Warning}[Theorem]{Warning}

\theoremstyle{remark}

\DeclareMathOperator{\Der}{Der}

\DeclareMathOperator{\Hom}{Hom}
\DeclareMathOperator{\End}{End}
\DeclareMathOperator{\Aut}{Aut}

\DeclareMathOperator*{\Lim}{lim}
\DeclareMathOperator*{\Colim}{colim}

\DeclareMathOperator{\Shuff}{Sh}

\newcommand{\bbR}{\mathbb{R}}

\newcommand{\bbZ}{\mathbb{Z}}
\newcommand{\bbA}{\mathbb{A}}

\newcommand{\calA}{\mathcal{A}}

\newcommand{\calC}{\mathcal{C}}

\newcommand{\calO}{\mathcal{O}}
\newcommand{\calT}{\mathcal{T}}
\newcommand{\calTcone}{\mathcal{T}^\triangleright}
\newcommand{\calR}{\mathcal{R}}
\newcommand{\calX}{\mathcal{X}}

\newcommand{\id}{\mathrm{id}}
\newcommand{\op}{\mathrm{op}}
\newcommand{\Id}{\mathrm{1}}
\newcommand{\pr}{\mathrm{pr}}

\newcommand{\Lie}{\mathcal{L}}

\newcommand{\Val}{\eta}
\newcommand{\Rmult}{\hat{m}}
\newcommand{\Radd}{\hat{+}}
\newcommand{\Runit}{\hat{1}}
\newcommand{\Rzero}{\hat{0}}

\newcommand{\Aadd}{\hat{+}}
\newcommand{\Azero}{\hat{0}}
\newcommand{\Aminus}{\hat{-}}
\newcommand{\Amod}{\hat{\kappa}}

\newcommand{\Tpart}[1]{T_{(#1)}}

\newcommand{\Swp}{\mathrm{swp}}

\newcommand{\Set}{{\mathcal{S}\mathrm{et}}}

\newcommand{\Eucl}{{\mathcal{E}\mathrm{ucl}}}

\newcommand{\Dflg}{{\mathcal{D}\mathrm{flg}}}
\newcommand{\Mod}{{\mathcal{M}\mathrm{od}}}
\newcommand{\Aff}{{\mathcal{A}\mathrm{ff}}}

\newcommand{\CAlg}{{\mathrm{C}\mathcal{A}\mathrm{lg}}}

\DeclareMathOperator{\Sym}{Sym}
\DeclareMathOperator{\Spec}{Spec}

\renewcommand{\phi}{\varphi}
\renewcommand{\epsilon}{\varepsilon}

\begin{document}

\title[Cartan calculus of cubical forms in tangent categories]{Cartan calculus of cubical forms\\ in tangent categories}

\author[C.~Blohmann]{Christian Blohmann}
\address{Max-Planck-Institut f\"ur Mathematik, Vivatsgasse 7, 53111 Bonn, Germany}
\email{blohmann@mpim-bonn.mpg.de}

\subjclass[2020]{18F40 (58A10, 58A03)}


\keywords{Tangent category, Cartan calculus, differential cubical forms}

\begin{abstract}
We construct the Cartan calculus of differential cubical forms on an object $X$ of a cartesian tangent category with a scalar multiplication by a commutative ring object $R$. We prove that the differential cubical forms are naturally equipped with the structure of a commutative differential graded algebra. Then we show that, if $R$ has no $2$-torsion, this CDGA together with the Lie algebra of vector fields and the inner derivatives constitutes a Cartan calculus, which lies between the initial Cartan calculus of algebraic K\"ahler forms and the terminal one of Lie-Rinehart forms. We give a number of examples, including affine schemes and elastic diffeological spaces, and show that we recover the established notions of differential forms in all cases.
\end{abstract}

\maketitle

\setcounter{tocdepth}{1}
\tableofcontents

\section{Introduction}

The Lie algebra of vector fields and the ring of differential forms on a smooth manifold together with the de Rham differential, the inner derivative, and the Lie derivative form the Cartan calculus, which is a fundamental tool of differential geometry. For generalized geometric structures, such as schemes, diffeological spaces, differential graded manifolds, differentiable stacks, etc., the question arises if and how a Cartan calculus can be constructed in a natural way. Tangent categories, introduced by Rosick\'y~\cite{Rosicky:1984} and further developed by Cockett and Cruttwell~\cite{CockettCruttwell:2014}, are a suitable framework to address this question.

The structure of a tangent category consists of a tangent functor together with the natural transformations that are needed to define a Lie algebra of vector fields \cite{Rosicky:1984, CockettCruttwell:2015}. For the Cartan calculus we also need a suitable notion of scalar multiplication by a ring object $R$ playing the role of the real numbers. In \cite{AintablianBlohmann:2026} such a notion was introduced, and it was proved that the tangent bundle has the structure of an abstract Lie algebroid \cite[Def.~7.1]{AintablianBlohmann:2025}. As a consequence, the Lie algebra of vector fields is a Lie-Rinehart algebra \cite{Rinehart:1963} over the ring of $R$-valued functions, so we obtain the Cartan calculus on Lie-Rinehart forms \cite[Thm.~6.13]{AintablianBlohmann:2026}. This is where the story could end. However, in important cases the Cartan calculus of Lie-Rinehart forms is not satisfactory.

The first case is that of affine schemes. The Lie algebra of vector fields on $\Spec(A)$ is the Lie algebra $\Der(A)$ of derivations of $A$. The usual notion of $1$-forms on $\Spec(A)$ is that of K\"ahler $1$-forms $\Omega^1_{A/k}$. The Lie-Rinehart $1$-forms, however, are given by the $A$-linear bidual,
\begin{equation*}
  \Omega^1_\mathrm{LR}\bigl( \Spec(A) \bigr)
  \cong
  \Hom_A\bigl( \Hom_A(\Omega^1_{A/k}, A), A \bigr)
  \,,
\end{equation*}
which is not genearlly isomorphic to $\Omega^1_{A/k}$. A sufficient condition for the Lie-Rinehart forms to be isomorphic to the K\"ahler forms is that $A$ is smooth over $k$ \cite[\S 7.7]{AintablianBlohmann:2026}.

The second case is that of elastic diffeological spaces \cite{Blohmann:2024a}, which have applications in Lagrangian Field Theory \cite{Blohmann:2024b}, infinite-dimensional Lie theory, geometric deformation theory, and other fields. On diffeological spaces there is a natural notion of de Rham forms \cite[\S 6.28]{IglesiasZemmour:2013} given by the left Kan extension of the de Rham functor on smooth manifolds. However, they are generally different from Lie-Rinehart forms (Section~\ref{sec:Elastic}).

All cases in which the Lie-Rinehart forms yield the standard notion of forms satisfy some finiteness condition. On a smooth manifold $M$, the sheaf of vector fields is locally free and finitely generated over the sheaf of smooth real-valued functions. This implies that the $C^\infty(M)$-linear maps $\calX(M) \to C^\infty(M)$ are in natural bijection with the fiberwise linear maps $TM \to \bbR$. Similarly, for an affine scheme $\Spec(A)$ over $k$, the condition for $A$ to be smooth ensures that $\Omega^1_{A/k}$ is finitely generated projective. It follows that $\Omega^1_{A/k}$ is reflexive, and hence isomorphic to the Lie-Rinehart $1$-forms. Without any such finiteness condition, the notion of Lie-Rinehart forms will not be the right one.

We can try to define de Rham $n$-forms on an object $X$ in a tangent category with scalar $R$-multiplication as antisymmetric and multilinear morphisms $\omega: TX \times_X \cdots \times_X TX \to R$ on the $n$-fold fiber product of the tangent bundle. However, it turns out that there is no intrinsic way to define a differential on such forms (Section~\ref{sec:DeRhamForms}). The solution we present here is to consider \textbf{cubical forms} instead \cite[\S I.14.2]{Kock:2006}. They are defined as antisymmetric and multilinear morphisms $\omega: T^n X \to R$ on the $n$-fold iterated tangent bundle. The differential of $\omega$ is defined as the antisymmetrization of the tangent morphism $T\omega: T^{n+1} X \to TR$ followed by the projection onto the fiber of $TR \cong R \times R$ (Definition~\ref{def:WedgeDiffInnder}).

On smooth manifolds, de Rham and cubical forms are isomorphic. In general tangent categories the cubical forms have to satisfy an additional condition, which essentially states that the derivative of $\omega$ in the fiber direction is $\omega$ itself \eqref{eq:CubDiffDef}. (In categories in which the derivative of a morphism is defined as its linear approximation, such as categories of manifolds of any kind, this extra condition is not needed, since the linear approximation of a linear map is the map itself.) We call cubical forms with this property \textbf{differential cubical forms}. They correspond to the singular forms in \cite{CruttwellLucyshynWright:2018}, albeit with a ring object as target.

The main result of this paper is Theorem~\ref{thm:CartanCube}. It states that in a cartesian tangent category with scalar $R$-multiplication (Definition~\ref{def:RScalar}), where $R$ has no $2$-torsion, the Lie algebra of vector fields and the differential cubical forms (Definition~\ref{def:CubeForms}), together with the natural wedge product, differential, and inner derivative (Definition~\ref{def:WedgeDiffInnder}), constitute a Cartan calculus (Definition~\ref{def:CartanCalc}).

\subsection{Related work}
\label{sec:RelatedWork}

During the last decade, substantial effort was invested in generalizing notions of differential geometry to tangent categories, largely by Cockett, Cruttwell, Lemay, and their collaborators: differential bundles as a replacement for vector bundles~\cite{CockettCruttwell:2018, MacAdam:2021, Ching:2024, CruttwellLemay:2023}, connections~\cite{CockettCruttwellConnections:2017}, Lie algebroids in the guise of involution algebroids~\cite{BurkeMacAdam:2019}, and flows of vector fields and ordinary differential equations~\cite{CockettCruttwellLemay:2021}. Structural results of Leung~\cite{Leung:2017} and Garner~\cite{Garner:2018} identify tangent structures with actions of a monoidal category of Weil algebras and show that every tangent category admits an embedding into a representable one. The supply of examples has been enlarged in several directions, among them algebras over an operad~\cite{Lanfranchi:2023b, IkonicoffLanfranchiLemay:2023}, affine schemes~\cite{LanfranchiLemay:2025}, coalgebras of differential categories~\cite{CockettLemayLucyshynWright:2020}, ind-categories~\cite{Vooys:2023}, and $\infty$-categories, where Bauer, Burke, and Ching~\cite{BauerBurkeChing:2021} exhibit Goodwillie calculus as a tangent structure.

The basic axioms of a tangent category do not provide an object of scalars. The tangent bundle is assumed to be a bundle of commutative monoids (or abelian groups), and the role of vector spaces is played by the differential objects of \cite[Def.~4.8]{CockettCruttwell:2014} and \cite[Def.~3.1]{CockettCruttwell:2018}, that is, by commutative monoids $E$ with a trivialization $TE \cong E \times E$ subject to compatibility conditions; by Proposition~\ref{prop:TangStabDiffObj} these are exactly the tangent-stable modules used here. MacAdam, Gallagher, and Lucyshyn-Wright~\cite{MacAdamGallagherLucyshynWright:2019} have proposed to axiomatize scalars by a differential object $R$ together with a point $\Runit: * \to R$ that is universal among linear maps. Such a scalar unit is automatically a commutative rig, and the differential objects are then precisely the Kock-Lawvere $R$-modules. The scalar $R$-multiplication of Definition~\ref{def:RScalar} is a structure of a different kind: it is not required to be universal, but it is required to act on the tangent bundle of every object, compatibly with the vertical lift $\lambda$ and the symmetric structure $\tau$. Since $R$ is a commutative ring object, $\calO_X = \Hom(X,R)$ is a commutative ring on which the vector fields act by derivations (Proposition~\ref{prop:VecActFunc}), which is a basic requirement of a Cartan calculus.

The work most closely related to ours is the paper of Cruttwell and Lucyshyn-Wright on sector forms~\cite{CruttwellLucyshynWright:2018}. The ring object $R$ of the scalar multiplication (Definition~\ref{def:TanStableRing}) is a special kind of differential object \cite[Def.~3.1]{CockettCruttwell:2018}. The defining equation~\eqref{eq:CubDiffDef} for differential cubical forms can be shown to be equivalent to one of the conditions for sector forms \cite[Def.~3.1, Def.~3.2]{CruttwellLucyshynWright:2018}. Since our forms are assumed to be antisymmetric, they correspond to the singular forms \cite[Def.~8.11]{CruttwellLucyshynWright:2018}. The main result of \cite{CruttwellLucyshynWright:2018} is that the sector forms on $X$ constitute a symmetric cosimplicial commutative monoid. The corresponding Dold-Kan differential restricts to the singular forms. For a smooth manifold and $E = \bbR$ the resulting complex of singular forms is isomorphic to the de Rham complex \cite[Thm.~9.25]{CruttwellLucyshynWright:2018}. The key feature of differential cubical forms is that they take values in a differential ring object $R$. This allows us to define the wedge product~\eqref{eq:CubeWedge} of forms, which was listed as an open problem in \cite[\S 10]{CruttwellLucyshynWright:2018}. Moreover, $R$ acts on every tangent bundle and satisfies the additional properties of a scalar multiplication (Definition~\ref{def:RScalar}), which makes the inner derivative~\eqref{eq:CubeInnDer}, and hence the Cartan calculus of Theorem~\ref{thm:CartanCube}, available.

The relation between forms on the iterated tangent bundle $T^n X$ and forms on the $n$-fold fiber product $T_n X$ is raised as an open problem in \cite[\S 10]{CruttwellLucyshynWright:2018}, where it is suggested that a positive answer should require $X$ to be locally a differential object or to carry a symmetric connection. In most of our examples, this requirement would be too strong. Section~\ref{sec:DeRhamForms} takes up this question in the presence of a scalar multiplication. The comparison morphism $\mu_n: T^n \to T_n$ pulls de Rham forms (Definition~\ref{def:DeRhamForms}) back to antisymmetric morphisms on $T^nX$ that are $R$-linear at every position, and we formulate an exactness condition on the vertical lift under which this pullback is an isomorphism in degree two.

Forms defined on the iterated tangent bundle also originate in synthetic differential geometry, where $T^n X \cong X^{D^n}$ for a microlinear object $X$, so that such forms are functions on infinitesimal $n$-cubes. Under the names of singular forms or forms on microcubes they appear in \cite[\S I.14.2]{Kock:2006}, \cite[Ch.~IV, \S 1]{MoerdijkReyes:1991}, and \cite[\S 4.1.1]{Lavendhomme:1996}. Their differential is defined, as in~\eqref{eq:dCubeDef}, by an alternating sum of permutations of $T\omega$ \cite[\S 4.2.3]{Lavendhomme:1996}, and in the well-adapted models of \cite{MoerdijkReyes:1991} they agree with the de Rham forms of a smooth manifold \cite[Ch.~IV, Prop.~3.7]{MoerdijkReyes:1991}. There the wedge product and the full Cartan calculus are available because the coefficient object is a ring of line type. Synthetic differential geometry also has a genuinely combinatorial theory of forms, based on infinitesimal simplices and infinitesimal parallelepipeda instead of iterated tangent bundles \cite[\S 3.1, \S 3.2]{Kock:2010}, \cite{Kock:2000, MoerdijkReyes:1986}; the comparison of the two theories is carried out in \cite[\S 4.7]{Kock:2010}. Whether the combinatorial theory can be set up in a tangent category appears to be unknown.

The de Rham forms of Definition~\ref{def:DeRhamForms} are the direct transcription of the classical notion, in which an $n$-form takes $n$ tangent vectors with a common base point as its arguments. Their exterior derivative is classically defined by a local formula \cite[7.8]{KolarMichorSlovak:1993}, which presupposes a linear structure on the underlying object. In a cartesian differential category every object has such a structure, and Cruttwell~\cite{Cruttwell:2013} has used this to define forms and an exterior derivative there, obtaining a cochain complex. Section~\ref{sec:DeRhamForms} explains why this approach does not extend to a tangent category.

Finally, there are the algebraic notions of forms, namely K\"ahler forms and the forms of a Lie-Rinehart algebra \cite{Rinehart:1963, Huebschmann:1998}. Both give rise to a Cartan calculus in every cartesian tangent category with scalar $R$-multiplication \cite{AintablianBlohmann:2026}, and Section~\ref{sec:Comparison} shows that these are the initial and the terminal Cartan calculus, with the differential cubical forms in between. On diffeological spaces, the accepted notion of de Rham forms is obtained by a left Kan extension from smooth manifolds \cite[\S 6.28]{IglesiasZemmour:2013}; the comparison with cubical forms on elastic diffeological spaces is carried out in Section~\ref{sec:Elastic}.

\begin{Terminology}
Here is why we call our forms ``cubical'' rather than ``singular'', as some authors do. The term ``singular'' arose in algebraic topology to distinguish $n$-chains given by \emph{embeddings} of $n$-cells ($n$-simplices or $n$-cubes) from those given by \emph{arbitrary} maps from the cells, which can have self-intersections and other types of singular behavior. The infinitesimal counterparts of singular \emph{simplicial} cochains are the de Rham forms given by maps on fiber products $T_n X \to R$, whereas infinitesimal singular \emph{cubical} cochains \cite[Ch.~II]{Massey:1980} are given by maps on powers $T^n X \to R$ \cite[\S I.14.2]{Kock:2006}.
\end{Terminology}

\subsection{Outline of the paper}

In Section~\ref{sec:AbstractTangent} we give a review of the notion of a cartesian tangent category with scalar multiplication. In Section~\ref{sec:cubical} we define differential cubical forms, their ring structure, differential, and inner derivative with respect to a vector field. Section~\ref{sec:Proof} contains the proof of Theorem~\ref{thm:CartanCube}, which is lengthy and involved. In Section~\ref{sec:Comparison} we define comparison maps from algebraic K\"ahler forms to differential cubical forms to Lie-Rinehart forms. In Section~\ref{sec:Examples} we give a number of examples and show that differential cubical forms recover the usual notion of forms in all of these cases.

\subsection{Notation and conventions}

We will use the notation and conventions of the companion paper~\cite{AintablianBlohmann:2026}. In particular, for two morphisms $f: X \to Y$ and $g: X \to Z$, the associated morphism to the product will be denoted by $(f,g): X \to Y \times Z$.

\section{Review of Cartesian tangent categories with scalar multiplication}
\label{sec:AbstractTangent}

\subsection{Bundles with algebraic structure}
\label{sec:BunAbelGp}

The algebraic structures we will consider are given by multi-sorted finite product theories, which we recall.

\begin{Definition}
\label{def:TAlg}
Let $\calT$ be a small category with finite products (including a terminal object as empty product). A \textbf{$\calT$-algebra} in a category $\calC$ is a functor $\calT \to \calC$ that preserves finite products. A \textbf{morphism of $\calT$-algebras} is a natural transformation that preserves finite products.
\end{Definition}


The \textbf{right cone} of $\calT$ will be denoted by $\calTcone$ \cite[Notation~1.2.8.4]{Lurie:2009}. The tip of the cone will be denoted by $*$, since it is the terminal object in $\calTcone$.

\begin{Definition}[Definition~2.3 in \cite{AintablianBlohmann:2026}]
\label{def:BundleOfTAlgs}
A \textbf{bundle of $\calT$-algebras} in the category $\calC$ is a functor $A: \calTcone \to \calC$ that preserves the limits of the diagrams $t \to * \leftarrow t'$ for all $t, t' \in \calT$. A morphism of bundles of $\calT$-algebras is a natural transformation that intertwines those limits.
\end{Definition}

Let $A: \calTcone \to \calC$ be a bundle of $\calT$-algebras. The object $X \coloneqq A(*) \in \calC$ is called the \textbf{base} of the bundle. If the pullback $A_x \coloneqq * \times_X A: \calT \to \calC$ over a point $x: * \to X$ exists in $\calC$, then $A_x$ is a $\calT$-algebra in $\calC$, called the \textbf{fiber over $x$}.

\begin{Example}
Let $R$ be a commutative ring object in the category $\calC$ with addition $\Radd$, zero $\Rzero$, multiplication $\Rmult$, and unit $\Runit$. Let $X \in \calC$ be an object. Let $X \times R$ denote the trivial bundle of rings over $X$. A \textbf{bundle of $R$-modules} consists of a bundle $p: A \to X$ of abelian groups, together with a bundle morphism $\kappa: R \times A \to A$ satisfying the usual conditions of a linear action of a ring. The axioms are spelled out explicitly for the case of the tangent bundle $TX \to X$ in a tangent category in \cite[Prop.~4.1]{AintablianBlohmann:2026}.

Our notion of bundles makes no assumption about local trivializations whatsoever. Therefore, it is somewhat surprising that bundles of $\bbR$-modules in the category of smooth finite-dimensional manifolds are vector bundles in the usual sense (with local trivializations) \cite[Example~2.7]{AintablianBlohmann:2025}.
\end{Example}

The functor of sections
\begin{equation}
\label{eq:SectionFunctor}
\begin{aligned}
  \Gamma: \calC^{[1]}
  &\longrightarrow \Set
  \\
  (A \xrightarrow{p} X)
  &\longmapsto
  \Gamma(X,A)
  = \{a:X \to A \ | \ p \circ a = \id_X \}
\end{aligned}
\end{equation}
is monoidal, mapping the fiber product of bundles to the cartesian product,
\begin{equation*}
  \Gamma(X, A \times_X B)
  \cong
  \Gamma(X,A) \times \Gamma(X,B)
  \,.
\end{equation*}
This implies that for every bundle of $\calT$-algebras $A: \calTcone \to \calC$ over $X$, the functor
\begin{align*}
  \Gamma(X,A): \calT &\longrightarrow \Set
  \\
  t &\longmapsto \Gamma(X, A_t)
\end{align*}
is a $\calT$-algebra.

\begin{Example}
Let $p: A \to X$ be a bundle of abelian groups in $\calC$ with structure morphisms $(+,0,i)$. Then $\Gamma(X,A)$ is an abelian group in the usual sense. The sum of two sections $a,b: X \to A$ is given by
\begin{equation*}
  a+b \coloneqq + \circ (a,b)
  \,.
\end{equation*}
The zero section $0: X \to A$ is the zero of the group and $i \circ a$ is the inverse.
\end{Example}

\begin{Example}
\label{ex:PointsOfR}
Let $R$ be a ring object in $\calC$. The set of sections of the trivial bundle of rings $X \times R \to X$ is a ring, which is naturally isomorphic via
\begin{equation}
\label{eq:SectionsTrivial}
\begin{aligned}
  \Gamma(X, X \times R) &\xrightarrow{~\cong~}
  \Hom(X, R) \\
  a &\longmapsto \pr_2 \circ a
\end{aligned}
\end{equation}
to the ring $\Hom(X,R)$ of $R$-valued functions with pointwise addition and multiplication
\begin{equation*}
  f + g \coloneqq \Radd \circ (f,g)
  \,,\quad
  fg \coloneqq \Rmult \circ (f,g)
  \,.
\end{equation*}
The zero of this ring is $X \to * \stackrel{\Rzero}{\to} R$, the multiplicative unit is $X \to * \stackrel{\Runit}{\to} R$.
\end{Example}

\begin{Example}
Let $A \to X$ be a bundle of $R$-modules with $R$-action $\kappa: R \times A \to A$. Using the isomorphism \eqref{eq:SectionsTrivial}, we see that $\Gamma(X,A)$ has the structure of a $\Hom(X,R)$-module given by
\begin{equation}
\label{eq:SectionsModuleStr}
  fa \coloneqq \kappa \circ (f,a) \,,
\end{equation}
for all $f \in \Hom(X,R)$ and $a \in \Gamma(X,A)$. Moreover, if $\phi: A \to A'$ is a morphism of bundles of $R$-modules with base map $\id_X$, then $\phi_*$ is a morphism of modules over $\Hom(X,R)$.
\end{Example}

\subsection{Symmetric structure on an endofunctor}




\begin{Definition}[Definition~2.4 in \cite{AintablianBlohmann:2026}]
\label{def:SymStruc}
Let $F: \calC \to \calC$ be a functor and $\tau: F^2 \to F^2$ a natural transformation. Let $\tau_{12} \coloneqq \tau\,F$ and $\tau_{23}\coloneqq F\,\tau$ be the two trivial extensions of $\tau$ to natural transformations $F^3 \to F^3$. We call $\tau$ a \textbf{braiding on $F$} if it satisfies the braid relations $\tau_{12} \circ \tau_{23} \circ \tau_{12} = \tau_{23} \circ \tau_{12} \circ \tau_{23}$. A braiding $\tau$ is called a \textbf{symmetric structure on $F$} if it satisfies $\tau\circ \tau = F^2$.
\end{Definition}

\begin{Remark}
A symmetric structure $\tau$ on a functor $F:\calC \to \calC$ defines an action of the symmetric group $S_n$ on $F^n$. The action of a transposition $\tau_{i,i+1}$, $1 \leq i < n$, is given by the natural isomorphism
\begin{equation*}
  F^{i-1} \tau F^{n-i-1}: F^n \longrightarrow F^n \,.
\end{equation*}
Since the symmetric group on $n$ elements is generated by adjacent transpositions, this induces a group homomorphism $S_n \to \Aut(F^n)$, where $\Aut(F^n)$ is the group of natural isomorphisms $F^n \to F^n$.
\end{Remark}

\begin{Definition}[Definition~2.5 in \cite{AintablianBlohmann:2026}]
\label{def:PreserveFibProd}
An endofunctor $F: \calC \to \calC$ \textbf{preserves the fiber products} of a bundle $p:A \to X$ if for all $k \geq 1$ the natural morphism of bundles over $FX$,
\begin{equation}
\label{eq:MorphBundles4}
  \nu_{k,X}:
  F(A \times_X^{p,p} \ldots \times_X^{p,p} A) \longrightarrow
  FA \times_{FX}^{Fp, Fp} \ldots \times_{FX}^{Fp, Fp} FA
  \,,
\end{equation}
where both sides have the same number $k$ of factors, is an isomorphism.
\end{Definition}


\subsection{Rosick\'{y}'s axioms}
\label{sec:RosickysAxioms}

In~\cite{Rosicky:1984}, Rosick\'{y} introduced the notion of abstract tangent functors, which axiomatizes the natural categorical structure of the tangent functor of manifolds that is needed to define the Lie bracket of vector fields.

\begin{Definition}[Sec.~2 in \cite{Rosicky:1984}, Def.~2.3 in \cite{CockettCruttwell:2014}]
\label{def:TangentStructure}
A \textbf{Rosick\'y tangent structure} on a category $\calC$ consists of a functor $T: \calC \to \calC$ together with natural transformations $\pi: T \to \Id$, $0: \Id \to T$, $+: T_2 \to T$, $\lambda: T \to T^2$, and $\tau: T^2 \to T^2$, such that the following axioms hold:

\textbf{(i)} The \textbf{fiber products}
\begin{equation}
\label{eq:TanFun1}
  T_k \coloneqq \underbrace{T \times_\Id T \times_\Id \ldots \times_\Id T}_{k \text{ factors}}
\end{equation}
over $T \stackrel{\pi}{\to} \Id$ exist for all $k \geq 1$, are pointwise, and preserved by $T$ (Definition~\ref{def:PreserveFibProd}).

\textbf{(ii)} $T \stackrel{\pi}{\to} \Id$ with neutral element $0$ and addition $+$ is a \textbf{bundle of abelian groups} over $\Id$ (Definition~\ref{def:BundleOfTAlgs}).

\textbf{(iii)}
$\tau: T^2 \to T^2$ is a \textbf{symmetric structure} on $T$ (Definition~\ref{def:SymStruc}) and a morphism of bundles of abelian groups, that is,
\begin{equation}
\label{eq:TanFun2}
\begin{tikzcd}[column sep=tiny]
T^2 \ar[rr, "\tau"] \ar[dr, "T \pi"'] && T^2 \ar[dl, "\pi T"]
\\
& T &
\end{tikzcd}
\qquad
\begin{tikzcd}[column sep=large]
T^2 \times_T^{T\pi, T\pi} T^2
\ar[r, "\tau \times_T \tau"]
\ar[d, "\nu_2^{-1}"']
&
T_2 T
\ar[dd, "+T"]
\\
T T_2  \ar[d, "T+"']
&
\\
T^2 \ar[r, "\tau"']
&
T^2
\end{tikzcd}
\end{equation}
commutes, where $\nu_2$ is the morphism~\eqref{eq:MorphBundles4} for $A=TX \xrightarrow{\pi_X} X$, $F=T$, and $k=2$.


\textbf{(iv)} The \textbf{vertical lift} $\lambda$ is a morphism of bundles of abelian groups, such that
\begin{equation}
\label{eq:TanFun3}
\begin{tikzcd}
T \ar[r, "\lambda"] \ar[d, "\pi"'] & T^2 \ar[d, "\pi T"]
\\
\Id \ar[r, "0"'] & T
\end{tikzcd}
\qquad\qquad
\begin{tikzcd}
T \ar[r, "\lambda"] \ar[d, "\lambda"'] &
T^2 \ar[d, "\lambda T"]
\\
T^2 \ar[r, "T \lambda"'] & T^3
\end{tikzcd}
\end{equation}
commutes, and
\begin{equation}
\label{eq:TanFun3add}
  (+T) \circ (\lambda \times_0 \lambda) = \lambda \circ +
  \,.
\end{equation}

\textbf{(v)} The diagrams
\begin{equation}
\label{eq:TanFun5}
\begin{tikzcd}[column sep=tiny]
& T \ar[dl, "\lambda"'] \ar[dr, "\lambda"] &
\\
T^2 \ar[rr, "\tau"'] & &
T^2
\end{tikzcd}
\qquad\qquad
\begin{tikzcd}
T^2 \ar[r, "T\lambda"] \ar[d, "\tau"'] &
T^3 \ar[r, "\tau T"] &
T^3 \ar[d, "T\tau"]
\\
T^2 \ar[rr, "\lambda T"'] & &
T^3
\end{tikzcd}
\end{equation}
commute.

\textbf{(vi)} The diagram
\begin{equation}
\label{eq:TanFun4}
\begin{tikzcd}
T \ar[r, "\lambda"] \ar[d, "\pi"']
\ar[dr, phantom, "\lrcorner", very near start]
& T^2 \ar[d, "{(\pi T, T\pi)}"]
\\
\Id \ar[r, "{(0, 0)}"'] & T_2
\end{tikzcd}
\end{equation}
is a pointwise pullback square.
\end{Definition}

\begin{Terminology}
\label{term:negatives}
In \cite{CockettCruttwell:2014} and subsequent work, Rosick\'{y}'s original requirement that $T \to \Id$ be a bundle of abelian groups is relaxed to a bundle of abelian monoids. The authors later refer to Rosick\'{y}'s original notion as tangent categories with negatives \cite{CockettCruttwell:2015,CockettCruttwellConnections:2017}. We will follow a more recent terminology and call these \textbf{Rosick\'y tangent categories} \cite{CruttwellLemay:2023,IkonicoffLanfranchiLemay:2023,LanfranchiLemay:2025}.
\end{Terminology}

The vertical lift can be extended by the additive bundle structure to the map
\begin{equation*}
  \lambda_2:
  T_2 \xrightarrow{~T0 \times_0 \lambda~}
  T_2 T
  \xrightarrow{~+T~}
  T^2
  \xrightarrow{~\tau~}
  T^2
  \,.
\end{equation*}
In components,
\begin{equation}
\label{eq:VertLiftExt}
\begin{split}
  (\lambda_2)_X
  &= \tau_X \circ {+_{TX}} \circ (T0_X \times_{0_X} \lambda_X)
  \\
  &= {T+_X} \circ \nu_{2,X}^{-1} \circ (0_{TX} \times_{0_X} \lambda_X)
  \,,
\end{split}
\end{equation}
for all $X \in \calC$. It was shown in \cite[Lem.~3.10]{CockettCruttwell:2014}, assuming all other axioms of a Rosick\'y tangent structure, that Axiom~\eqref{eq:TanFun4} is satisfied if and only if
\begin{equation}
\label{diag:lambda2pullback}
\begin{tikzcd}
T_2 \ar[r, "\lambda_2"] \ar[d, "\pi \circ \pr_1"']
\ar[dr, phantom, "\lrcorner", very near start]
&
T^2 \ar[d, "{T\pi}"]
\\
\Id \ar[r, "0"'] & T
\end{tikzcd}
\end{equation}
is a pointwise pullback.

\begin{Definition}[Definition~2.7 in \cite{AintablianBlohmann:2026}]
Let $\calC$ be a category with a tangent structure. A \textbf{vector field} on $X \in \calC$ is a section of $\pi_X: TX \to X$.
\end{Definition}

The natural transformation that maps a pair of tangent vectors in the same fiber to their difference will be denoted by
\begin{equation}
\label{eq:Difference}
  -: T_2 \xrightarrow{~\id_T \times_{\Id} i~}
  T_2 \xrightarrow{~+~} T
  \,,
\end{equation}
where $i: T \to T$ is the additive inverse of the bundle of abelian groups. The bracket of two vector fields $v,w: X \to TX$ is defined as follows: The morphism
\begin{equation}
\label{eq:DeltaOrig}
  \delta(v,w)
  \coloneqq
  -_{TX} \circ
  (Tw \circ v, \tau_X \circ Tv \circ w): X \longrightarrow T^2X
\end{equation}
satisfies
\begin{equation}
\label{eq:DeltaKerOrig}
\begin{aligned}
  \pi_{TX} \circ \delta(v,w) &= w
  \\
  T\pi_X \circ \delta(v,w) &= 0_X
  \,.
\end{aligned}
\end{equation}
It follows from the universal property of the pullback that there is a unique vector field $[v,w]$ such that
\begin{equation}
\label{eq:deltaBracketRel}
  \delta(v,w) = (\lambda_2)_X \circ \bigl(w, [v,w] \bigr)
  \,.
\end{equation}
It is straightforward to check that $[v,w]$ is additive in $v$ and $w$. It was announced in \cite{Rosicky:1984} and proved in \cite{CockettCruttwell:2015} with the input of Rosick\'y that $[v,w]$ satisfies the Jacobi identity. Observe that all ingredients of the tangent structure, including the additive inverse of $T \to \Id$, are needed for the definition of the bracket of vector fields. The Lie algebra of vector fields will be denoted by $\calX(X)$.

\subsection{Cartesian tangent structures}
\label{sec:CartesianTangent}

Let $\calC$ be a category with finite products. For all $X, Y \in \calC$ we have the natural morphism
\begin{equation}
\label{eq:chiXY}
  \chi_{X,Y} \coloneqq (T\pr_1, T\pr_2): T(X \times Y)
  \longrightarrow TX \times TY
\end{equation}
For a morphism $(f,g): Z \to X \times Y$ we have
\begin{equation}
\label{eq:chifg}
   (Tf,Tg) = \chi_{X,Y} \circ T(f, g)
   \,.
\end{equation}
We recall that a functor $T: \calC \to \calC$ is said to \textbf{preserve finite products} if $\chi_{X,Y}$ has an inverse for all $X,Y \in\calC$ and if $T$ preserves the empty product $T* \cong *$.

\begin{Definition}[Definition~3.1 in \cite{AintablianBlohmann:2026}]
\label{def:CartTan}
A tangent structure is called \textbf{cartesian} if the tangent functor preserves finite products.
\end{Definition}

\begin{Remark}
That the tangent functor preserves finite products is only one of the conditions of the definition of cartesian tangent structures in \cite[Definition~2.8]{CockettCruttwell:2014}. In Proposition~3.5 of \cite{AintablianBlohmann:2026} it was proved that this implies all other requirements of the definition of \cite{CockettCruttwell:2014}, so that the two definitions are equivalent.
\end{Remark}

As was observed in \cite[Def.~2.9]{CockettCruttwell:2014}, with a cartesian tangent structure we can define the \textbf{partial tangent morphisms} of a morphism $f: X \times Y \to Z$ by
\begin{align}
  \Tpart{1} f: TX \times Y
  &\xrightarrow{~\id_{TX} \times 0_Y~}
  TX \times TY \xrightarrow{~\chi_{X,Y}^{-1}~}
  T(X \times Y) \xrightarrow{~Tf~}
  TZ
  \\
  \label{eq:PartialT2}
  \Tpart{2} f: X \times TY
  &\xrightarrow{~0_X \times \id_{TY}~}
  TX \times TY \xrightarrow{~\chi_{X,Y}^{-1}~}
  T(X \times Y)\xrightarrow{~Tf~}
  TZ
  \,,
\end{align}
where the index refers to the factor in the product and where $\chi$ is the natural transformation~\eqref{eq:chiXY}. We use positional indices in order to describe cases where $X = Y$.

The partial tangent morphisms are functorial in $X$, $Y$, and $Z$. That is, if there is a commutative diagram
\begin{equation*}
\begin{tikzcd}
X \times Y
\ar[r, "f"]
\ar[d, "\alpha \times \beta"']
&
Z
\ar[d, "\gamma"]
\\
X' \times Y'
\ar[r, "f'"']
&
Z'
\end{tikzcd}
\end{equation*}
then the diagram
\begin{equation}
\label{eq:T2functorial}
\begin{tikzcd}[column sep=large]
X \times TY
\ar[r, "\Tpart{2} f"]
\ar[d, "\alpha \times T\beta"']
&
TZ
\ar[d, "T\gamma"]
\\
X' \times TY'
\ar[r, "\Tpart{2} f'"']
&
TZ'
\end{tikzcd}
\end{equation}
commutes. Since the tangent morphism is linear, $\Tpart{2}f$ is linear in the second argument. That is, the diagram
\begin{equation}
\label{eq:T2linear}
\begin{tikzcd}[column sep=5em]
X \times T_2 Y
\ar[r, "\cong"]
\ar[d, "\id_X \times +_Y"']
&
(X \times TY) \times_{X \times Y} (X \times TY)
\ar[r, "{\Tpart{2}f \, \times_f \, \Tpart{2} f}"]
&
T_2 Z
\ar[d, "+_Z"]
\\
X \times TY
\ar[rr, "\Tpart{2} f"']
&&
TZ
\end{tikzcd}
\end{equation}
commutes. The analogous statements hold for $\Tpart{1} f$. The tangent morphism is given by the sum of the partial tangent morphisms, as stated in the following result:

\begin{Proposition}[Proposition~2.10 in \cite{CockettCruttwell:2014}]
\label{prop:TSumPartials}
The relation
\begin{equation*}
  Tf = +_Z \circ
  \bigl(\Tpart{1} f \circ
  (\id_{TX} \times \pi_Y),
  \Tpart{2} f \circ
  (\pi_X \times \id_{TY})
  \bigr) \circ
  \chi_{X,Y}
  \,,
\end{equation*}
holds for any morphism $f: X \times Y \to Z$ in a cartesian tangent category.
\end{Proposition}

\subsection{Module structure}
\label{sec:ModStructure}

Let $R$ be a ring object in the category $\calC$. It gives rise to an endofunctor $\Id \times R: \calC \to \calC$, $X \mapsto X \times R$, which is equipped with the projection $\pr_1: \Id \times R \to \Id$. The ring structure of $R$ equips $\Id \times R \to \Id$ with the structure of a ring internal to endofunctors over $\Id$. An $(\Id \times R \to \Id)$-module structure on $(T \to \Id) \in \End(\calC)_{/\Id}$ will be called, for short, an \textbf{$R$-module structure} on $T \to \Id$.

\begin{Proposition}[Proposition~4.1 in \cite{AintablianBlohmann:2026}]
\label{prop:RModuleBundle}
Let $\calC$ be a tangent category; let $R$ be a unital ring object in $\calC$ with addition $\Radd$, zero $\Rzero$, multiplication $\Rmult$, and unit $\Runit$. An $R$-module structure on the bundle $\pi:T \to \Id$ is given by a natural morphism
\begin{equation*}
  \kappa_X: R \times TX
  \longrightarrow TX
  \,,
\end{equation*}
such that the following diagrams commute for all $X \in \calC$:
\begin{itemize}

\item[(i)] Morphism of bundles:
\begin{equation*}
\begin{tikzcd}[column sep={tiny}]
R \times TX \ar[rr, "\kappa_X"] \ar[dr, "\pi_X \circ \, \pr_2"'] &&
TX \ar[dl, "\pi_X"]
\\
& X &
\end{tikzcd}
\end{equation*}

\item[(ii)] Associativity:
\begin{equation*}
\begin{tikzcd}[column sep={large}]
R \times R \times TX \ar[r, "\id_R \times \kappa_X"]
\ar[d, "\Rmult \times \id_{TX}"'] &
R \times TX \ar[d, "\kappa_X"]
\\
R \times TX \ar[r, "\kappa_X"'] & TX
\end{tikzcd}
\end{equation*}

\item[(iii)] Unitality:
\begin{equation*}
\begin{tikzcd}[column sep=3em]
* \times TX \ar[r, "\Runit \times \id_{TX}"] \ar[dr, "\cong"'] &
R \times TX \ar[d, "\kappa_X"]
\\
& TX
\end{tikzcd}
\end{equation*}

\item[(iv)] Linearity in $R$:
\begin{equation*}
\begin{tikzcd}[column sep={large}, row sep=3em]
R \times R \times TX
\ar[r, "\Radd \, \times \, \id_{TX}"]
\ar[d, "{\big( \kappa_X \circ (\pr_1, \pr_3), \, \kappa_X \circ (\pr_2, \pr_3) \big)}"'] &
R \times TX
\ar[d, "\kappa_X"]
\\
TX \times_X TX
\ar[r, "+_X"']
&
TX
\end{tikzcd}
\end{equation*}

\item[(v)] Linearity in $TX$:
\begin{equation*}
\begin{tikzcd}[column sep={large}, row sep=3em]
R \times TX \times_X TX
\ar[r, "\id_R \, \times \, +_X"]
\ar[d, "{\bigl( \kappa_X \circ (\pr_1, \pr_2), \, \kappa_X \circ (\pr_1, \pr_3) \bigr)}"'] &
R \times TX
\ar[d, "\kappa_X"]
\\
TX \times_X TX
\ar[r, "+_X"']
&
TX
\end{tikzcd}
\end{equation*}
\end{itemize}
\end{Proposition}

\begin{Remark}
As is the case for any module structure, linearity in $R$, expressed by diagram~(iv), implies that the scalar multiplication by $\Rzero: * \to R$ sends $TX$ to the zero section, that is, the diagram
\begin{equation}
\label{diag:RzeroPreserve}
\begin{tikzcd}[column sep=4em,row sep=2em]
* \times TX \ar[r, "{\Rzero \times \id_{TX}}"] \ar[d, "{\cong}"'] & R \times TX \ar[dd, "{\kappa_{X}}"] \\
TX \ar[d, "\pi_X"'] & \\
X \ar[r, "0_X"'] & TX
\end{tikzcd}
\end{equation}
is commutative. Similarly, linearity in $TX$, expressed by diagram~(v), implies that the scalar multiplication sends the zero section to the zero section, that is, the diagram
\begin{equation}
\label{diag:RzeroPreserve2}
\begin{tikzcd}[column sep=4em,row sep=2em]
R \times X \ar[r, "{\id_R \times 0_{X}}"] \ar[d, "{\pr_2}"'] & R \times TX \ar[d, "{\kappa_{X}}"]
\\
X \ar[r, "0_X"'] & TX
\end{tikzcd}
\end{equation}
commutes.
\end{Remark}

A tangent category with an $R$-module structure on the tangent bundle $T \to \Id$ will be called, for short, a tangent category with an $R$-module structure. As was observed in \cite[Proposition~4.2]{AintablianBlohmann:2026}, every tangent category in the sense of \cite{CockettCruttwell:2014} with an $R$-module structure is a Rosick\'y tangent category. If we want to work in tangent categories without negatives, the ring $R$ has to be replaced by a rig.


\begin{Proposition}[Proposition~4.3 in \cite{AintablianBlohmann:2026}]
\label{prop:RCartTan}
In a cartesian tangent category with an $R$-module structure $\kappa_X: R \times TX \to TX$, the diagram
\begin{equation*}
\begin{tikzcd}[column sep=4em, row sep=3em]
R \times T(X \times Y)
\ar[r, "\id_R \times \chi_{X,Y}"]
\ar[d, "\kappa_{X \times Y}"']
&
R \times TX \times TY
\ar[d, "{\big( \kappa_X \circ (\pr_1, \pr_2), \, \kappa_Y \circ (\pr_1, \pr_3) \big)}"]
\\
T(X \times Y) \ar[r, "\chi_{X,Y}"']
&
TX \times TY
\end{tikzcd}
\end{equation*}
commutes for all $X, Y \in \calC$.
\end{Proposition}
In other words, Proposition~\ref{prop:RCartTan} states that $\chi_{X,Y}: T(X \times Y) \to TX \times TY$ is an isomorphism of bundles of $R$-modules over $X \times Y$ if we equip $TX \times TY$ with the diagonal $R$-module structure.

\begin{Warning}
In the category of sets, an abelian group is the same thing as a $\bbZ$-module. For an abelian group internal to a category $\calC$, this is only true if $\bbZ$ is naturally a ring object in $\calC$, which is not always the case. For example, a category with finite biproducts contains no other ring object than the zero ring. It follows that the only possible tangent structure with an $R$-module structure on such a category is the trivial tangent structure \cite[Section~8.10.3]{AintablianBlohmann:2026}.
\end{Warning}

\subsection{Tangent-stable modules}

From now on, we will assume that $\calC$ is a cartesian tangent category with an $R$-module structure. A \textbf{trivialization} of the tangent bundle of an object $X$ in $\calC$ is an isomorphism
\begin{equation*}
  TX \cong X \times A
\end{equation*}
of bundles of $R$-modules over $X$, where $A$ is an $R$-module in $\calC$. By taking the fiber at a point $x: * \to X$, we obtain an isomorphism of $R$-modules $T_x X \cong A$.

\begin{Definition}[Definition~4.7 in \cite{AintablianBlohmann:2026}]
\label{def:ModTangentStable}
An $R$-module $A$ in $\calC$ will be called \textbf{tangent-stable} if there is an isomorphism of $R$-modules
\begin{equation*}
  T_{\hat{0}} A \cong A
  \,,
\end{equation*}
where $\hat{0}: * \to A$ is the zero of the module.
\end{Definition}

\begin{Proposition}[Proposition~4.8 in \cite{AintablianBlohmann:2026}]
An $R$-module $A$ in $\calC$ is tangent-stable if and only if its tangent bundle has a trivialization
\begin{equation*}
  TA \cong A \times A
  \,.
\end{equation*}
\end{Proposition}

\begin{Definition}
Let $A$ be a tangent-stable $R$-module. Then
\begin{equation*}
  \Val_A: TA \longrightarrow A
\end{equation*}
will denote the projection onto the fiber of $TA \cong A \times A$.
\end{Definition}

\begin{Proposition}[Proposition~4.9 in \cite{AintablianBlohmann:2026}]
\label{prop:TangStabDiffObj}
If an $R$-module $A$ is tangent-stable, then the diagrams
\begin{equation}
\label{eq:DiffObj1}
\begin{tikzcd}[column sep=large]
TA \times_A TA
\ar[r, "+_A"]
\ar[d, "{(\pr_1, \pr_2)}"']
&
TA
\ar[dd, "\Val_A"]
\\
TA \times TA
\ar[d, "\Val_A \times \Val_A"']
&
\\
A \times A
\ar[r, "\Aadd"']
&
A
\end{tikzcd}
\qquad\qquad
\begin{tikzcd}
A
\ar[r, "0_A"]
\ar[d]
&
TA
\ar[d, "\Val_A"]
\\
*
\ar[r, "\Azero"']
&
A
\end{tikzcd}
\end{equation}
\begin{equation}
\label{eq:DiffObj2}
\begin{tikzcd}[column sep=large]
T(A \times A)
\ar[r, "T\Aadd"]
\ar[d, "\chi_{A,A}"']
&
TA
\ar[dd, "\Val_A"]
\\
TA \times TA
\ar[d, "\Val_A \times \Val_A"']
\\
A \times A
\ar[r, "\Aadd"']
&
A
\end{tikzcd}
\qquad\qquad
\begin{tikzcd}
T*
\ar[r, "T\Azero"]
\ar[d]
&
TA
\ar[d, "\Val_A"]
\\
*
\ar[r, "\Azero"']
&
A
\end{tikzcd}
\end{equation}
\begin{equation}
\label{eq:DiffObj3}
\begin{tikzcd}[column sep=large]
TA
\ar[r, "\lambda_A"]
\ar[d, "\Val_A"']
&
T^2 A
\ar[d, "T\Val_A"]
\\
A
&
TA
\ar[l, "\Val_A"]
\end{tikzcd}
\end{equation}
commute.
\end{Proposition}

\begin{Remark}
In \cite[Def.~3.1]{CockettCruttwell:2018}, a differential object in a tangent category was defined to be a commutative monoid $A$ such that there is an isomorphism $TA \cong A \times A$ and such that the diagrams~\eqref{eq:DiffObj1}, \eqref{eq:DiffObj2}, and \eqref{eq:DiffObj3} commute. In the earlier paper \cite[Def.~4.8]{CockettCruttwell:2014} of the same authors, the second diagram of \eqref{eq:DiffObj2} was shown to be redundant \cite[Lem.~4.9]{CockettCruttwell:2014}. Proposition~\ref{prop:TangStabDiffObj} shows that an abelian group object is a differential object if and only if it is tangent-stable.
\end{Remark}

\begin{Proposition}[Proposition~4.10 in \cite{AintablianBlohmann:2026}]
\label{prop:frakgTanStable}
The tangent fibers of a group object are tangent-stable.
\end{Proposition}

\begin{Lemma}[Lemma~4.11 in \cite{AintablianBlohmann:2026}]
For a tangent-stable module $A$, the following relations hold:
\begin{align}
\label{eq:TEtaEtaTtau}
  T\Val_{A}
  &= \Val_{TA} \circ \tau_A
  \\
\label{eq:EtaTEtalambda}
  \Val_A
  &=
  \Val_A \circ T\Val_{A} \circ \lambda_{A}
  \\
\label{eq:EtaTEtalambda2}
  \Val_A \circ \pr_2
  &=
  \Val_A \circ T\Val_{A} \circ (\lambda_2)_A
  \\
\label{eq:EtaEtaT}
  \Val_A \circ T\Val_{A}
  &= \Val_A \circ \Val_{TA}
\\
\label{eq:EtaTEtaSymmetric01}
  \Val_A \circ T\Val_{A} \circ \tau_A
  &= \Val_A \circ T\Val_{A}
\end{align}
\end{Lemma}

\begin{Lemma}[Lemma~4.12 in \cite{AintablianBlohmann:2026}]
\label{lem:AaddTadd}
Let $A$ be a tangent-stable module with addition $\Aadd$. Then
\begin{equation*}
  \Val_A \circ T\Val_A \circ +_{T\!A}
  = \Aadd \circ
  ( \Val_A \circ T\Val_A \circ \pr_1,
    \Val_A \circ T\Val_A \circ \pr_2 )
  \,,
\end{equation*}
where $\pr_1, \pr_2: T^2\! A \times_{TA} T^2\! A \to T^2\!A$ are the projections.
\end{Lemma}
\begin{Corollary}[Corollary~4.13 in \cite{AintablianBlohmann:2026}]
\label{cor:AaddTminus}
Let $A$ be a tangent-stable module with subtraction $\Aminus$. Then
\begin{equation}
\label{eq:AminusTminus}
  \Val_A \circ T\Val_A \circ -_{T\!A}
  = \Aminus \circ
  ( \Val_A \circ T\Val_A \circ \pr_1,
    \Val_A \circ T\Val_A \circ \pr_2 )
  \,,
\end{equation}
where $\pr_1, \pr_2: T^2\! A \times_{TA} T^2\! A \to T^2\!A$ are the projections.
\end{Corollary}

\subsection{The axioms of scalar multiplication}
\label{sec:ScMult}

In order to obtain a Cartan calculus on every object of a tangent category, we have to require certain compatibility conditions of the $R$-module structure with the tangent structure.

\begin{Definition}[Definition~4.14 in \cite{AintablianBlohmann:2026}]
\label{def:TanStableRing}
A commutative ring object in a cartesian tangent category will be called \textbf{tangent-stable} if it is tangent-stable (Definition~\ref{def:ModTangentStable}) as a module over itself.
\end{Definition}

\begin{Definition}[Definition~4.15 in \cite{AintablianBlohmann:2026}]
\label{def:RScalar}
Let $R$ be a commutative ring object in a cartesian tangent category $\calC$. An $R$-module structure $\kappa_X: R \times TX \to TX$ on the tangent bundle $\pi:T \to \Id$ will be called a \textbf{scalar multiplication} if $R$ is tangent-stable and if the following diagrams commute for all $X \in \calC$:
\begin{equation}
\label{diag:ScalarMult1}
\begin{tikzcd}[column sep={large}]
R \times TX
\ar[r, "\id_R \times \lambda_X"]
\ar[d, "\kappa_X"'] &
R \times T^2 X \ar[d, "\kappa_{TX}"]
\\
TX
\ar[r, "\lambda_X"']
&
T^2 X
\end{tikzcd}
\end{equation}
\begin{equation}
\label{diag:ScalarMult2}
\begin{tikzcd}[column sep=6em]
TR \times TX
\ar[r, "\Tpart{1}\kappa_X"]
\ar[d, "{(\pi_R, \Val_R) \, \times \, \id_{TX}}"']
&[4em]
T^2X
\\
R \times R \times TX
\ar[r, "{\big(\kappa_X \circ (\pr_1,\pr_3), \,
\kappa_X \circ (\pr_2,\pr_3)\big)}"']
&
T_2 X
\ar[u, "(\lambda_2)_X"']
\end{tikzcd}
\end{equation}
\begin{equation}
\label{diag:ScalarMult3}
\begin{tikzcd}[column sep=3em]
R \times T^2 X
\ar[r, "\Tpart{2}\kappa_X"]
\ar[d, "\id_R \times \tau_X"']
&
T^2 X
\\
R \times T^2 X
\ar[r, "\kappa_{TX}"']
&
T^2 X
\ar[u, "\tau_X"']
\end{tikzcd}
\end{equation}

\end{Definition}


\begin{Proposition}[Proposition~4.16 in \cite{AintablianBlohmann:2026}]
\label{prop:RmodLinearization}
Let $A$ be an $R$-module in a cartesian tangent category with scalar $R$-multiplication; let $\Amod: R \times A \to A$ denote the module structure. If $A$ is tangent-stable, then the diagram
\begin{equation*}
\begin{tikzcd}
R \times TA
\ar[r, "\Tpart{2} \Amod"]
\ar[d, "\id_R \times \Val_A"']
&
TA
\ar[d, "\Val_A"]
\\
R \times A
\ar[r, "\Amod"']
&
A
\end{tikzcd}
\end{equation*}
commutes.
\end{Proposition}
In other words, Proposition~\ref{prop:RmodLinearization} states that the tangent map of the $R$-multiplication of a tangent-stable $R$-module is the $R$-multiplication itself.

\begin{Proposition}[Proposition~4.17 in \cite{AintablianBlohmann:2026}]
\label{prop:LeibnizRuleR}
In a cartesian tangent category with scalar $R$-multiplication, the multiplication $\Rmult$ and the addition $\Radd$ of $R$ satisfy
\begin{equation*}
  \Val_R \circ T\Rmult
  =
  \Radd \circ \bigl( \Rmult \circ (\Val_R \times \pi_R),
  \Rmult \circ (\pi_R \times \Val_R) \bigr)
  \circ \chi_{R,R}
  \,.
\end{equation*}
\end{Proposition}

From Proposition~\ref{prop:RmodLinearization} we obtain
\begin{align}
\label{eq:LeibnizRuleR1}
  \Val_R \circ \Tpart{2} \Rmult
  &= \Rmult \circ (\id_R \times \Val_R)
  \\
\label{eq:LeibnizRuleR2}
  \Val_R \circ \Tpart{1} \Rmult
  &= \Rmult \circ (\Val_R \times \id_R)
  \,.
\end{align}

\subsection{The tangent Lie algebroid}

Let $X$ be an object in a cartesian tangent category with scalar $R$-multiplication. Recall from Example~\ref{ex:PointsOfR} that the set $\Hom(X,R)$ of $R$-valued morphisms is equipped with a ring structure given by
\begin{equation*}
  f + g \coloneqq \Radd \circ (f,g)
  \,,\qquad
  fg \coloneqq \Rmult \circ (f,g)
\end{equation*}
for all morphisms $f, g \in \Hom(X,R)$.
Assume that $R$ is tangent-stable. Then we can define an action of vector fields on functions by
\begin{equation}
\label{eq:VecFieldFunc}
  v \cdot f: X \xrightarrow{~v~}
  TX \xrightarrow{~Tf~} TR \xrightarrow{~\Val_R~}
  R
\end{equation}
for all $v \in \Gamma(X,TX)$ and $f \in \Hom(X,R)$.

\begin{Proposition}[Proposition~5.3 in \cite{AintablianBlohmann:2026}]
\label{prop:VecActFunc}
In a cartesian tangent category with scalar $R$-multiplication, the action~\eqref{eq:VecFieldFunc} is a representation of the Lie algebra of vector fields by derivations on the ring of $R$-valued functions. That is,
\begin{subequations}
\begin{align}
\label{eq:VecActFuncA}
  v \cdot (f + g)
  &= v \cdot f + v \cdot g
  \\
\label{eq:VecActFuncB}
  (v+w) \cdot f
  &=
  v \cdot f + w \cdot f
  \\
\label{eq:VecActFunc1}
  [v,w] \cdot f
  &= v \cdot (w \cdot f) - w \cdot (v \cdot f)
  \\
\label{eq:VecActFunc2}
  v \cdot (fg)
  &= (v \cdot f)\,g + f\,(v \cdot g)
  \,,
\end{align}
\end{subequations}
for all $v, w \in \Gamma(X,TX)$ and $f,g \in \Hom(X,R)$.
\end{Proposition}

The scalar multiplication $\kappa_X: R \times TX \to TX$ equips the abelian group of vector fields with the structure of an $\Hom(X,R)$-module, given by
\begin{equation}
\label{eq:fvDef}
  fv \coloneqq \kappa_X \circ (f,v)
  \,.
\end{equation}

\begin{Proposition}[Proposition~5.6 in \cite{AintablianBlohmann:2026}]
\label{prop:LeibnizRule}
In a cartesian tangent category with scalar $R$-multiplication, the \textbf{Leibniz rule}
\begin{equation}
\label{eq:LeibnizRule}
  [v,fw] = (v \cdot f)w + f[v,w]
\end{equation}
holds for all vector fields $v, w \in \Gamma(X,TX)$ and all functions $f: X \to R$.
\end{Proposition}

\begin{Definition}[Definition~7.1 in \cite{AintablianBlohmann:2025}]
\label{def:AbstractLieAlgd}
Let $\calC$ be a cartesian tangent category with scalar $R$-mul\-ti\-pli\-ca\-tion. An \textbf{abstract Lie algebroid} in $\calC$ consists of a bundle of $R$-modules $A \to X$, a morphism $\rho: A \to TX$ of bundles of $R$-modules, called the \textbf{anchor}, and a Lie bracket on the abelian group $\Gamma(X,A)$, such that
\begin{align}
\label{eq:LieAlgdLeibniz}
  [a, fb] &= f[a,b] + \bigl((\rho \circ a) \cdot f \bigr) b
  \\
\label{eq:anchorLieAlgMap}
  \rho \circ [a,b] &= [\rho \circ a , \rho \circ b]
\end{align}
for all sections $a$, $b$ of $A \to X$ and all morphisms $f: X \to R$ in $\calC$.
\end{Definition}

\begin{Theorem}[Theorem 6.12 in \cite{AintablianBlohmann:2026}]
\label{thm:TangentLieAlgd}
Let $X$ be an object in a cartesian tangent category with scalar $R$-multiplication. Then the tangent bundle $TX \to X$ with the anchor $\id: TX \to TX$ and the Lie bracket of vector fields is an abstract Lie algebroid, called the \textbf{tangent Lie algebroid} of $X$.
\end{Theorem}

It was proved in Proposition~3.5 and Proposition~4.5 of \cite{AintablianBlohmann:2026} that the natural isomorphism $\chi_{X,Y}: T(X \times Y) \to TX \times TY$ is compatible with the tangent structure and the scalar $R$-multiplication. Therefore, we can and from now on will identify $T(X \times Y) \cong TX \times TY$, which amounts to a choice of the limit representing the categorical product.

\section{The Cartan calculus of differential cubical forms}
\label{sec:cubical}

\subsection{Differential cubical forms}

We recall that a $(p,q)$-shuffle is a permutation $\sigma \in S_{p+q}$ that satisfies $\sigma(1) < \ldots < \sigma(p)$ and $\sigma(p+1) < \ldots < \sigma(p+q)$. The set of $(p,q)$-shuffles $\Shuff(p,q)$ is a set of representatives of the left cosets $S_{p+q}/(S_p \times S_q)$. The inverse of a shuffle is called an unshuffle. The set of $(p,q)$-unshuffles $\Shuff(p,q)^{-1}$ is a set of representatives of the right cosets $(S_p \times S_q)\backslash S_{p+q}$.

The \textbf{block shuffle} $\Swp_{p,q} \in \Shuff(p,q)$ is the permutation
\begin{equation}
\label{eq:ijShuffle}
  \Swp_{p,q}(k)
  =
  \begin{cases}
     k+q &; 1 \leq k \leq p \\
     k-p &; p < k \leq p+q \,,
  \end{cases}
\end{equation}
which swaps $\{1, \ldots, p\}$ with $\{p+1, \ldots, p+q\}$. By definition, we set
\begin{equation*}
 \Swp_{p+q,0} = \Swp_{0,p+q} = \id
 \,.
\end{equation*}
Observe that $\Swp_{p,q}^{-1} = \Swp_{q,p} \in \Shuff(q,p)$. The length of the permutation is
\begin{equation*}
  |\Swp_{p,q}| = pq
  \,.
\end{equation*}
Moreover, we have the relation
\begin{equation}
\label{eq:ShuffSwp}
  \Shuff(q,p)\circ \Swp_{p,q}
  = \Shuff(p,q)
  \,.
\end{equation}

We recall that the symmetric structure on $T$ gives rise to the action of the symmetric group $S_n$ on $T^n X$. We denote the map $S_n \to \Aut(T^n X)$ by $\sigma \mapsto \sigma_X$.

\begin{Definition}
A morphism $\omega\colon T^n X \to R$ is \textbf{antisymmetric} if
\begin{equation*}
  \omega \circ \sigma_X = (-1)^{|\sigma|} \omega
\end{equation*}
for all $\sigma \in S_n$.
\end{Definition}

For $1 \leq i \leq n$ we introduce the morphisms
\begin{align*}
  +^{(i)}_{T^{n-1} X}
  &\colon
  (T^{i-1}T_2 T^{n-i})X \longrightarrow T^n X
  \\
  \kappa^{(i)}_{T^{n-1}X}
  &\colon
  R \times T^n X \longrightarrow T^n X
\end{align*}
of addition and scalar multiplication \textbf{at the $i$-th position}, defined by
\begin{align}
  \label{eq:TaddAti}
  +^{(i)}_{T^{n-1}X}
  &\coloneqq
  (\Swp_{1,i-1}T^{n-i})_X \circ +_{T^{n-1}X} \circ
  \bigl((\Swp_{i-1,1}T^{n-i})_X \times (\Swp_{i-1,1}T^{n-i})_X \bigr)
  \\
  \label{eq:RmultAti}
  \kappa^{(i)}_{T^{n-1}X}
  &\coloneqq
  (\Swp_{1,i-1}T^{n-i})_X \circ \kappa_{T^{n-1}X} \circ \bigl(\id_R \times (\Swp_{i-1,1}T^{n-i})_X \bigr)
  \,.
\end{align}
At the first position, we retrieve the usual operations
\begin{equation*}
  +^{(1)}_{T^{n-1}X} = +_{T^{n-1}X}
  \,,\quad
  \kappa^{(1)}_{T^{n-1}X} = \kappa_{T^{n-1}X}
  \,.
\end{equation*}

\begin{Lemma}
We have the following recursion relations:
\begin{align}
\label{eq:AddPosRecurs}
  +^{(i+1)}_{T^n X}
  &= T+^{(i)}_{T^{n-1}X}
  \\
\label{eq:ScalMultPosRecurs}
  \kappa^{(i+1)}_{T^n X}
  &= T_{(2)} \kappa^{(i)}_{T^{n-1}X}
  \,.
\end{align}
\end{Lemma}

\begin{proof}
As a special case of the shuffle defined in~\eqref{eq:ijShuffle}, we have
\begin{equation*}
  \Swp_{1,i}
  = \tau_{i,i+1} \circ \ldots \circ
  \tau_{23} \circ \tau_{12}
  \,.
\end{equation*}
This shows that for the action of the shuffle on $T^n$ we have the recursive relation
\begin{equation*}
  \Swp_{1,i}
  = T\Swp_{1,i-1} \circ \tau T^{i-1}
  \,.
\end{equation*}
Using this relation, we get
\begin{equation*}
\begin{split}
  +_{T^n X}^{(i+1)}
  &=
  (T\Swp_{1,i-1}T^{n-i} \circ \tau T^{n-1})_X \circ +_{T^{n}X}
  \\
  &\qquad{}
  \circ
  \bigl((\tau T^{n-1} \circ T\Swp_{i-1,1}T^{n-i})_X \times
  (\tau T^{n-1} \circ T\Swp_{i-1,1}T^{n-i})_X \bigr)
  \\
  &=
  (T\Swp_{1,i-1}T^{n-i})_X \circ
  \bigl(\tau \circ +T \circ (\tau \times \tau)
  \bigr)_{T^{n-1}X}
  \\
  &\qquad{}
  \circ
  \bigl((T\Swp_{i-1,1}T^{n-i})_X \times
  (T\Swp_{i-1,1}T^{n-i})_X \bigr)
  \\
  &=
  (T\Swp_{1,i-1}T^{n-i})_X \circ  (T+)_{T^{n-1}X} \circ
  \bigl((T\Swp_{i-1,1}T^{n-i})_X \times
  (T\Swp_{i-1,1}T^{n-i})_X \bigr)
  \\
  &= T+_{T^{n-1}X}^{(i)}
  \,,
\end{split}
\end{equation*}
where we have used Diagram~\eqref{eq:TanFun2} and the functoriality of $T$. Similarly, we obtain for the scalar multiplication
\begin{equation*}
\begin{split}
  \kappa^{(i+1)}_{T^n X}
  &=
  (T\Swp_{1,i-1}T^{n-i})_X \circ
  (\tau T^{n-1})_X \circ \kappa_{T^n X} \circ
  \bigl(\id_R \times (\tau T^{n-1})_X \bigr)
  \\
  &\qquad{}
  \circ
  \bigl(\id_R \times (T\Swp_{i-1,1}T^{n-i})_X \bigr)
  \\
  &=
  (T\Swp_{1,i-1}T^{n-i})_X \circ
  T_{(2)}\kappa_{T^{n-1} X} \circ
  \bigl(\id_R \times T(\Swp_{i-1,1}T^{n-i})_X \bigr)
  \\
  &=
  (T\Swp_{1,i-1}T^{n-i})_X \circ
  T\kappa_{T^{n-1} X} \circ
  (0_R \times \id_{T^{n+1} X}) \circ
  \bigl(\id_R \times T(\Swp_{i-1,1}T^{n-i})_X \bigr)
  \\
  &=
  (T\Swp_{1,i-1}T^{n-i})_X \circ
  T\kappa_{T^{n-1} X} \circ
  \bigl(T\id_R \times T(\Swp_{i-1,1}T^{n-i})_X \bigr) \circ
  (0_R \times \id_{T^{n+1} X})
  \\
  &=
  T\kappa^{(i)}_{T^{n-1}X} \circ
  (0_R \times \id_{T^{n+1} X})
  \\
  &=
  T_{(2)}\kappa^{(i)}_{T^{n-1}X}
  \,,
\end{split}
\end{equation*}
where we have used Diagram~\eqref{diag:ScalarMult3}, the defining Equation~\eqref{eq:PartialT2} of $T_{(2)}$, the functoriality of $T$, and again Equation~\eqref{eq:PartialT2}.
\end{proof}

\begin{Definition}
\label{def:CubeFormAti}
A morphism $\omega\colon T^n X \to R$ will be called \textbf{$R$-linear at the $i$-th position} if the diagrams
\begin{equation}
\label{diag:CubForm1and2}
\begin{tikzcd}[column sep=2em]
\bigl( T^{i-1} (T_2) T^{n-i} \bigr)X
\ar[d, "+^{(i)}_{T^{n-1}X}"']
\ar[r, "\omega \times \omega"]
&
R \times R
\ar[d, "\Radd"]
\\
T^n X
\ar[r, "\omega"']
&
R
\end{tikzcd}
\qquad
\begin{tikzcd}
R \times T^n X
\ar[d, "\kappa^{(i)}_{T^{n-1}X}"']
\ar[r, "{\id_R \times \omega}"]
&
R \times R
\ar[d, "\Rmult"]
\\
T^n X
\ar[r, "\omega"']
&
R
\end{tikzcd}
\end{equation}
commute.
\end{Definition}

For later reference, we write the commutativity of the diagrams of Definition~\ref{def:CubeFormAti} as equations:
\begin{align}
\label{eq:CubForm1}
  \omega \circ {+^{(i)}_{T^{n-1}X}}
  &= \Radd \circ (\omega \times \omega)
  \\
\label{eq:CubForm2}
  \omega \circ \kappa^{(i)}_{T^{n-1}X}
  &= \Rmult \circ (\id_R \times \omega)
  \,.
\end{align}

\begin{Lemma}
\label{lem:AtiPermute}
If $\omega: T^n X \to R$ is $R$-linear at the $i$-th position, then for every $\sigma \in S_n$ the morphism $\omega \circ \sigma_X$ is $R$-linear at the $\sigma^{-1}(i)$-th position. If $\omega$ is antisymmetric and $R$-linear at the first position, then it is $R$-linear at every position.
\end{Lemma}

\begin{Definition}
\label{def:CubeForms1}
An antisymmetric morphism $\omega\colon T^n X \to R$ that is $R$-linear at the $i$-th position for all $1 \leq i \leq n$ will be called a \textbf{cubical $n$-form} on $X$.
\end{Definition}

The notion of cubical forms is not sufficient for the construction of a Cartan calculus. In addition, we have to require the condition
\begin{equation}
\label{eq:CubDiffDef}
  \Val_R \circ T\omega \circ \lambda_{T^{n-1}X} =  \omega
\end{equation}
for forms of degree $n \geq 1$. The condition states that the tangent map of $\omega$ in the vertical fiber direction is the form itself. In categories where the tangent functor is defined as linear approximation, such as manifolds of any kind, this condition holds since the linear approximation of a linear map is the map itself. In general tangent categories, \eqref{eq:CubDiffDef} must be promoted to an axiom.

\begin{Lemma}
\label{lem:DiffImpliesLin}
If $\omega: T^n X \to R$ satisfies~\eqref{eq:CubDiffDef}, then $\omega$ is $R$-linear at the first position.
\end{Lemma}

\begin{proof}
We have
\begin{equation*}
\begin{split}
  \omega \circ +_{T^{n-1} X}
  &=
  \Val_R \circ T\omega \circ \lambda_{T^{n-1}X} \circ +_{T^{n-1} X}
  \\
  &=
  \Val_R \circ T\omega \circ
  +_{T^n X} \circ (\lambda_{T^{n-1}X} \times \lambda_{T^{n-1}X})
  \\
  &=
  \Val_R \circ +_R \circ (T\omega \times T\omega) \circ
  (\lambda_{T^{n-1}X} \times \lambda_{T^{n-1}X})
  \\
  &=
  \Radd \circ (\Val_R \times \Val_R) \circ
  (T\omega \times T\omega) \circ
  (\lambda_{T^{n-1}X} \times \lambda_{T^{n-1}X})
  \\
  &=
  \Radd \circ (\omega \times \omega)
  \,,
\end{split}
\end{equation*}
where we have used the assumption~\eqref{eq:CubDiffDef}, the additivity~\eqref{eq:TanFun3add} of the vertical lift, that tangent morphisms are additive, the left diagram of~\eqref{eq:DiffObj1} for $A = R$, and again the assumption~\eqref{eq:CubDiffDef}. Similarly, we compute
\begin{equation*}
\begin{split}
  \omega \circ \kappa_{T^{n-1} X}
  &=
  \Val_R \circ T\omega \circ \lambda_{T^{n-1}X} \circ
  \kappa_{T^{n-1} X}
  \\
  &=
  \Val_R \circ T\omega \circ
  \kappa_{T^n X} \circ (\id_R \times \lambda_{T^{n-1}X})
  \\
  &=
  \Val_R \circ \kappa_R \circ (\id_R \times T\omega) \circ
  (\id_R \times \lambda_{T^{n-1}X})
  \\
  &=
  \Rmult \circ (\id_R \times \Val_R) \circ (\id_R \times T\omega) \circ
  (\id_R \times \lambda_{T^{n-1}X})
  \\
  &=
  \Rmult \circ (\id_R \times \omega)
  \,,
\end{split}
\end{equation*}
where we have used the assumption~\eqref{eq:CubDiffDef}, Diagram~\eqref{diag:ScalarMult1} for the object $T^{n-1}X$, the naturality of $\kappa$, the $R$-linearity of $\Val_R$ (which holds because the trivialization $TR \cong R \times R$ is an isomorphism of bundles of $R$-modules), and again the assumption~\eqref{eq:CubDiffDef}.
\end{proof}

\begin{Corollary}
An antisymmetric morphism $\omega: T^n X \to R$ that satisfies~\eqref{eq:CubDiffDef} is a cubical $n$-form.
\end{Corollary}

\begin{proof}
This follows from Lemma~\ref{lem:DiffImpliesLin} and Lemma~\ref{lem:AtiPermute}.
\end{proof}

This corollary justifies the following terminology.

\begin{Definition}
\label{def:CubeForms}
An antisymmetric morphism $\omega\colon T^n X \to R$ that satisfies~\eqref{eq:CubDiffDef} will be called a \textbf{differential cubical $n$-form}. The space of all differential cubical $n$-forms on $X$ will be denoted by $\Omega^n_\mathrm{cub}(X)$. By definition, $\Omega^0_\mathrm{cub}(X) = \Hom(X,R)$.
\end{Definition}


\begin{Example}
On a smooth manifold there is a natural bijection between differential cubical forms and de Rham forms (Section~\ref{sec:SmoothMflds}).
\end{Example}

\subsection{The main theorem}

We can define the multiplication of differential cubical forms as follows. Consider the natural transformations
\begin{align*}
  T^p \pi^q\colon T^{p+q} &\longrightarrow T^p
  \\
  \pi^p T^q\colon T^{p+q} &\longrightarrow T^q
  \,,
\end{align*}
which evaluate at $X$ to
\begin{align*}
  (T^p \pi^q)_X
  &\coloneqq T^p \pi_{X} \circ T^p \pi_{TX} \circ \ldots \circ
  T^p \pi_{T^{q-1}X}
  \\
  (\pi^p T^q)_X
  &\coloneqq
  \pi_{T^{q} X} \circ \pi_{T^{1+q}X} \circ \ldots \circ
  \pi_{T^{p-1+q}X}
  \,.
\end{align*}
For $\alpha \in \Omega^p_\mathrm{cub}(X)$ and $\beta \in \Omega^q_\mathrm{cub}(X)$ we define a morphism
\begin{equation}
\label{eq:FormProd}
  \alpha \cdot \beta
  \coloneqq \Rmult\circ\bigl(\alpha\circ (T^p\pi^q)_X,\beta\circ(\pi^pT^q)_X\bigr)
  \colon T^{p+q} X \longrightarrow R
  \,.
\end{equation}

\begin{Definition}
\label{def:WedgeDiffInnder}
The \textbf{wedge product} of $\alpha \in \Omega^p_\mathrm{cub}(X)$ and $\beta \in \Omega^q_\mathrm{cub}(X)$ is
\begin{align}
\label{eq:CubeWedge}
  \alpha \wedge \beta
  &\coloneqq
  \sum_{\mathclap{\sigma \in \Shuff(p,q)^{-1}}}
  (-1)^{|\sigma|}
  (\alpha \cdot \beta) \circ \sigma_X
  \,.
  \\
\intertext{The \textbf{differential} of $\omega \in \Omega^n_\mathrm{cub}(X)$ is}
\label{eq:dCubeDef}
  d\omega
  &\coloneqq
  \sum_{\mathclap{\sigma \in \Shuff(1,n)^{-1} }}
  (-1)^{|\sigma|}\Val_R \circ T\omega \circ \sigma_X
  \,.
  \\
\intertext{The \textbf{inner derivative} of $\omega$ with respect to a vector field $v \in \calX(X)$ is}
\label{eq:CubeInnDer}
  \iota_v \omega
  &\coloneqq
  (-1)^{n+1} \omega \circ T^{n-1}v
  \,.
\end{align}
\end{Definition}

\begin{Definition}[Def.~6.5 in \cite{AintablianBlohmann:2026}]
\label{def:CartanCalc}
An abstract \textbf{Cartan calculus} consists of a commutative differential graded algebra $(\Omega, d)$, an $\Omega^0$-module $\calX$ equipped with a Lie bracket, and for every $v \in \calX$ a graded derivation $\iota_v$ of degree $-1$, such that
\begin{equation}
\label{eq:InnDerOlinear}
  \iota_{v+w} = \iota_v + \iota_w
  \,,\quad
  \iota_{fv}\omega = f \wedge \iota_v\omega
\end{equation}
and, setting $\Lie_v \coloneqq [\iota_v, d]$, the bracket relations
\begin{equation}
\label{eq:CartanCalcBrackets}
  [\iota_v, \iota_w] = 0 \,,\quad
  [\Lie_v, \iota_w] = \iota_{[v,w]} \,,\quad
  [\Lie_v, \Lie_w] = \Lie_{[v,w]} \,,\quad
  [\Lie_v, d] = 0
\end{equation}
hold for all $v, w \in \calX$, $f \in \Omega^0$, and $\omega \in \Omega$.
\end{Definition}

\begin{Remark}
\label{rmk:CartanRedund}
The bracket relations~\eqref{eq:CartanCalcBrackets} are redundant. It is enough to assume that $[\iota_v, \iota_w] = 0$ and $[\Lie_v, \iota_w] = \iota_{[v,w]}$. The remaining two relations follow from the graded Jacobi identity and $[d,d] = 2d^2 = 0$ (see the remarks after Definition~6.5 in \cite{AintablianBlohmann:2026}).
\end{Remark}

\begin{Theorem}
\label{thm:CartanCube}
Let $X$ be an object in a cartesian tangent category with scalar $R$-multiplication.
\begin{itemize}

\item[(i)] The collection of differential cubical forms $\Omega^n_\mathrm{cub}(X)$, $n \geq 0$, with the wedge product~\eqref{eq:CubeWedge} and the differential~\eqref{eq:dCubeDef} is a commutative differential graded $R$-algebra.

\item[(ii)] If $R$ has no $2$-torsion, then the differential graded algebra $(\Omega_\mathrm{cub}(X), d)$ together with the Lie algebra of vector fields $\calX(X)$ and the inner derivative~\eqref{eq:CubeInnDer} constitutes a Cartan calculus (Definition~\ref{def:CartanCalc}).

\end{itemize}
\end{Theorem}

\section{Proof of Theorem~\ref{thm:CartanCube}}
\label{sec:Proof}

The proof of Theorem~\ref{thm:CartanCube} is long and involved. We will break it up into a number of more manageable steps. We begin with the proof of Statement~(i) of Theorem~\ref{thm:CartanCube}.

\subsection{The ring structure}
\label{sec:ProofRing}

This section is devoted to the proof that the wedge product equips $\Omega_\mathrm{cub}(X)$ with the structure of a commutative graded ring. First, we need a technical lemma about shuffles. Let $\sigma \in \Shuff(p,q)^{-1}$ and $\rho \in S_{p+q}$. Since $\Shuff(p,q)^{-1}$ is a set of representatives of the right cosets $(S_p \times S_q)\backslash S_{p+q}$, there are unique $\sigma_\rho \in \Shuff(p,q)^{-1}$, $\mu_\rho \in S_p$, and $\nu_\rho \in S_q$ such that
\begin{equation}
\label{eq:SnShuffAct}
  \sigma \circ \rho = (\mu_\rho, \nu_\rho) \circ \sigma_\rho
  \,.
\end{equation}
This shows that there is a unique map
\begin{equation}
\label{eq:ShuffIntertwine}
\begin{aligned}
  \Shuff(p,q)^{-1} \times S_{p+q}
  &\longrightarrow
  (S_p \times S_q) \times \Shuff(p,q)^{-1}
  \\
  (\sigma, \rho)
  &\longmapsto
  \bigl( (\mu_\rho, \nu_\rho), \sigma_\rho \bigr)
\end{aligned}
\end{equation}
such that Equation~\eqref{eq:SnShuffAct} holds. The following lemma is elementary:

\begin{Lemma}
\label{lem:ShuffShuff}
The analogue of the map~\eqref{eq:ShuffIntertwine} for the subgroup $S_{p+1} \times S_q \subset S_{p+q+1}$ induces a bijection
\begin{align*}
  \Shuff(p,q)^{-1} \times \Shuff(1,p+q)^{-1}
  &\xrightarrow{~\cong~}
  \Shuff(1,p)^{-1} \times \Shuff(p+1,q)^{-1}
  \\
  (\sigma, \rho) &\longmapsto (\mu_\rho, \sigma_\rho)
\end{align*}
such that for the action of $S_{p+q+1}$ on $T^{p+q+1}$ we have
\begin{equation}
\label{eq:ShuffShuff01}
  T\sigma \circ \rho
  = \mu_\rho T^q \circ \sigma_\rho
  \,.
\end{equation}
\end{Lemma}

Now we will show that $\alpha \wedge \beta$ is antisymmetric for $\alpha \in \Omega^p_\mathrm{cub}(X)$ and $\beta \in \Omega^q_\mathrm{cub}(X)$. Let $\sigma \in \Shuff(p,q)^{-1}$ and $\rho \in S_{p+q}$. With Equation~\eqref{eq:SnShuffAct} we obtain
\begin{equation*}
\begin{split}
  (\alpha \cdot \beta) \circ \sigma_X \circ \rho_X
  &=
  \Rmult\circ\bigl(\alpha\circ (T^p\pi^q)_X,\beta\circ(\pi^pT^q)_X\bigr) \circ
  \sigma_X \circ \rho_X
  \\
  &=
  \Rmult\circ\bigl(
  \alpha \circ (\mu_\rho)_X \circ (T^p\pi^q)_X,
  \beta  \circ (\nu_\rho)_X \circ (\pi^pT^q)_X\bigr) \circ
  (\sigma_\rho)_X
  \\
  &= (-1)^{|\mu_\rho| + |\nu_\rho|}
  \Rmult\circ\bigl(
  \alpha \circ (T^p\pi^q)_X,
  \beta  \circ (\pi^pT^q)_X\bigr) \circ
  (\sigma_\rho)_X
  \\
  &= (-1)^{|\mu_\rho| + |\nu_\rho|}
  (\alpha \cdot \beta) \circ (\sigma_\rho)_X
  \,,
\end{split}
\end{equation*}
where we have used Equations~\eqref{eq:FormProd} and~\eqref{eq:SnShuffAct}, that $(T^p\pi^q)\circ(\mu_\rho,\nu_\rho)=\mu_\rho\circ(T^p\pi^q)$ and $(\pi^pT^q)\circ(\mu_\rho,\nu_\rho)=\nu_\rho\circ(\pi^pT^q)$, and the assumption that $\alpha$ and $\beta$ are antisymmetric. Inserting this relation into~\eqref{eq:CubeWedge}, we get
\begin{equation}
\begin{split}
  (\alpha \wedge \beta) \circ \rho_X
  &=
  \sum_{\mathclap{\sigma \in \Shuff(p,q)^{-1}}}
  (-1)^{|\sigma|}
  (\alpha \cdot \beta) \circ \sigma_X \circ \rho_X
  \\
  &=
  \sum_{\mathclap{\sigma \in \Shuff(p,q)^{-1}}}
  (-1)^{|\sigma| + |\mu_\rho| + |\nu_\rho|}
  (\alpha \cdot \beta) \circ (\sigma_\rho)_X
  \\
  &=
  \sum_{\mathclap{\sigma_\rho \in \Shuff(p,q)^{-1}}}
  (-1)^{|\rho| + |\sigma_\rho|}
  (\alpha \cdot \beta) \circ (\sigma_\rho)_X
  \\
  &=
  (-1)^{|\rho|} \alpha \wedge \beta
  \,,
\end{split}
\end{equation}
where we have used that $|\sigma| + |\rho| \equiv |\sigma \circ \rho| = |(\mu_\rho,\nu_\rho)\circ \sigma_\rho| \equiv |\mu_\rho| + |\nu_\rho| + |\sigma_\rho| \pmod 2$. This shows that $\alpha \wedge \beta$ is antisymmetric.

Next, we will show that $\alpha \wedge \beta$ satisfies the defining condition~\eqref{eq:CubDiffDef} of a differential cubical form. We have
\begin{equation}
\label{eq:WedgeDiff01}
\begin{split}
  \Val_R \circ T(\alpha \wedge \beta) \circ (\lambda T^{p+q-1})_X
  &=
  \Val_R \circ T\Bigl(
    \sum_{\mathclap{\sigma \in \Shuff(p,q)^{-1}}} (-1)^{|\sigma|}
    (\alpha \cdot \beta) \circ \sigma_X \Bigr) \circ (\lambda T^{p+q-1})_X  \\
  &=
  \sum_{\mathclap{\sigma \in \Shuff(p,q)^{-1}}}
  (-1)^{|\sigma|}
  \Val_R \circ T(\alpha \cdot \beta) \circ (T\sigma)_X \circ (\lambda T^{p+q-1})_X
  \,.
\end{split}
\end{equation}
We further compute
\begin{equation}
\label{eq:TDerProd}
\begin{split}
  \Val_R \circ T(\alpha \cdot \beta)
  &=
  \Val_R \circ
  T\bigl( \Rmult\circ\bigl(\alpha\circ (T^p\pi^q)_X,\beta\circ(\pi^pT^q)_X\bigr) \bigr)
  \\
  &=
  \Val_R \circ
  T\Rmult
  \circ \bigl(
  T\alpha\circ (TT^p\pi^q)_X,
  T\beta \circ (T\pi^pT^q)_X \bigr)
  \\
  &=
  \Radd \circ \bigl( \Rmult \circ (\Val_R \times \pi_R),
  \Rmult \circ (\pi_R \times \Val_R) \bigr)
  \circ
  \\
  &\qquad{} \bigl(
  T\alpha\circ (TT^p\pi^q)_X,
  T\beta \circ (T\pi^pT^q)_X \bigr)
  \\
  &=
  \Rmult \circ \bigl(
    \Val_R \circ T \alpha\circ (TT^p\pi^q)_X,
    \pi_R \circ  T\beta \circ (T\pi^pT^q)_X \bigr)
  \\
  &\quad{} +
  \Rmult \circ \bigl(
    \pi_R \circ T\alpha \circ (TT^p\pi^q)_X,
    \Val_R \circ T\beta \circ (T\pi^pT^q)_X \bigr)
  \\
  &=
  \Rmult \circ \bigl(
    \Val_R \circ T\alpha \circ (T^{p+1}\pi^q)_X,
    \beta \circ (\pi T^q)_X \circ (T\pi^p T^q)_X \bigr)
  \\
  &\quad{} +
  \Rmult \circ \bigl(
    \alpha \circ (\pi T^p)_X \circ (TT^p\pi^q)_X,
    \Val_R \circ  T\beta \circ (T\pi^pT^q)_X \bigr)
  \\
  &=
  \Rmult \circ \bigl(
    \Val_R \circ T\alpha \circ (T^{p+1}\pi^q)_X,
    \beta \circ (\pi^{p+1} T^q)_X \bigr)
  \\
  &\quad{} +
  \Rmult \circ \bigl(
    \alpha \circ (\pi T^p\pi^q)_X,
    \Val_R \circ  T\beta \circ (T\pi^pT^q)_X \bigr)
  \\
  &=
    (\Val_R \circ T\alpha) \cdot \beta
  \\
  &\quad{}
  + \Rmult \circ \bigl(
    \alpha \circ (T^p \pi^{q+1})_X,
    \Val_R \circ T\beta \circ (\pi^p T^{q+1})_X \bigr)
    \circ (\Swp_{1,p} T^q)_X
  \\
  &=
    (\Val_R \circ T\alpha) \cdot \beta
    + \bigl(\alpha \cdot (\Val_R \circ T\beta)\bigr) \circ (\Swp_{1,p}T^q)_X
  \,,
\end{split}
\end{equation}
where $\Swp_{1,p}$ was defined in~\eqref{eq:ijShuffle}, and where we have used Equation~\eqref{eq:FormProd}, the functoriality of~$T$, Proposition~\ref{prop:LeibnizRuleR}, the naturality of~$\pi$, and finally the properties of the ring structure of~$R$. We will deal with the two summands on the right side of~\eqref{eq:TDerProd} separately.

For any unshuffle $\sigma \in \Shuff(p,q)^{-1}$ we have
\begin{equation*}
  T^p \pi^q \circ \sigma
  = \xi_\sigma^1 \xi_\sigma^2 \cdots \xi_\sigma^{p+q}
\,,
\end{equation*}
where
\begin{equation*}
  \xi_\sigma^i =
  \begin{cases}
    T &; 1 \leq \sigma(i) \leq p \\
    \pi &;  p+1 \leq \sigma(i) \leq p+q
  \end{cases}
  \,.
\end{equation*}
By horizontal composition with $T$, this yields $T^{p+1}\pi^q \circ T\sigma = T\xi^1_\sigma \xi^2_\sigma \cdots \xi^{p+q}_\sigma$. The unshuffle must map $1$ either to $1$ or to $p+1$. If $\sigma(1) = 1$, then $\xi^1_\sigma = T$. In this case
\begin{equation*}
\begin{split}
  T^{p+1} \pi^q \circ T\sigma \circ \lambda T^{p+q-1}
  &= T^2 \xi^2_\sigma \cdots \xi^{p+q}_\sigma \circ \lambda T^{p+q-1}
  \\
  &= \lambda T^{p-1} \circ T\xi^2_\sigma \cdots \xi^{p+q}_\sigma
  \\
  &= \lambda T^{p-1} \circ T^p \pi^q \circ \sigma
  \,,
\end{split}
\end{equation*}
where we have used the naturality of~$\lambda$. If $\sigma(1) = p+1$, then
\begin{equation*}
\begin{split}
  T^{p+1} \pi^q \circ T\sigma \circ \lambda T^{p+q-1}
  &= T\pi \xi^2_\sigma \cdots \xi^{p+q}_\sigma \circ \lambda T^{p+q-1}
  \\
  &= 0 T^p \circ \pi\xi^2_\sigma \cdots \xi^{p+q}_\sigma
  \\
  &= 0 T^p \circ T^p \pi^q \circ \sigma
  \,,
\end{split}
\end{equation*}
where we have used the left diagram of~\eqref{eq:TanFun3} together with $\tau \circ \lambda = \lambda$ (left diagram of~\eqref{eq:TanFun5}), which imply $T\pi \circ \lambda = 0 \circ \pi$. In summary, we have:
\begin{align}
\label{eq:piTsigmalambda1}
  T^{p+1} \pi^q \circ T\sigma \circ \lambda T^{p+q-1}
  &=
  \begin{cases}
    \lambda T^{p-1} \circ T^p \pi^q \circ \sigma
    &; \sigma(1) = 1 \\
    0T^p \circ T^p \pi^q \circ \sigma
    &;  \sigma(1) = p+1
  \end{cases}
  \\
\intertext{By analogous computations, we obtain}
\label{eq:piTsigmalambda2}
  \pi^{p+1} T^q \circ T\sigma \circ \lambda T^{p+q-1}
  &=
  \begin{cases}
    \pi^p T^q \circ \sigma
    &; \sigma(1) = 1 \\
    0 T^{q-1} \circ \pi T^{q-1} \circ \pi^p T^q \circ \sigma
    &;  \sigma(1) = p+1
  \end{cases}
  \\
\label{eq:piTsigmalambda3}
  \pi T^p \pi^q \circ T\sigma \circ \lambda T^{p+q-1}
  &=
  \begin{cases}
    0 T^{p-1} \circ \pi T^{p-1} \pi^q \circ \sigma
    &; \sigma(1) = 1 \\
    T^p \pi^q \circ \sigma
    &;  \sigma(1) = p+1
  \end{cases}
  \\
\label{eq:piTsigmalambda4}
  T \pi^p T^q \circ T\sigma \circ \lambda T^{p+q-1}
  &=
  \begin{cases}
    0 T^q \circ \pi^p T^q \circ \sigma
    &; \sigma(1) = 1 \\
    \lambda T^{q-1} \circ \pi^p T^q \circ \sigma
    &;  \sigma(1) = p+1
  \end{cases}
\end{align}
Assume that $\sigma(1) = 1$. Using~\eqref{eq:piTsigmalambda1}, we obtain
\begin{equation*}
\begin{split}
  &
  \Val_R \circ T\alpha \circ
  (T^{p+1} \pi^q)_X \circ (T\sigma)_X \circ (\lambda T^{p+q-1})_X
  \\
  ={}&
  \Val_R \circ T\alpha \circ (\lambda T^{p-1})_X \circ
  (T^p \pi^q)_X \circ \sigma_X
  \\
  ={}&
  \alpha \circ (T^p \pi^q)_X \circ \sigma_X
  \,,
\end{split}
\end{equation*}
where we have used that by assumption $\alpha$ is differential~\eqref{eq:CubDiffDef}. Using~\eqref{eq:piTsigmalambda2}, we obtain
\begin{equation*}
  \beta \circ
  (\pi^{p+1} T^q)_X \circ (T\sigma)_X \circ (\lambda T^{p+q-1})_X
  =
  \beta \circ (\pi^p T^q)_X \circ \sigma_X
  \,.
\end{equation*}
By inserting the last two equations into the first summand of~\eqref{eq:TDerProd}, we obtain
\begin{equation*}
  \bigl((\Val_R \circ T\alpha) \cdot \beta \bigr)
  \circ (T\sigma)_X \circ (\lambda T^{p+q-1})_X
  = (\alpha \cdot \beta) \circ \sigma_X
\end{equation*}
for $\sigma(1) = 1$. Now assume $\sigma(1) = p+1$. By analogous computations, we obtain
\begin{equation*}
  \bigl((\Val_R \circ T\alpha) \cdot \beta \bigr)
  \circ (T\sigma)_X \circ (\lambda T^{p+q-1})_X
  = 0
  \,.
\end{equation*}
In summary:
\begin{align*}
  \bigl((\Val_R \circ T\alpha) \cdot \beta \bigr)
  \circ (T\sigma)_X \circ (\lambda T^{p+q-1})_X
  &=
  \begin{cases}
    (\alpha \cdot \beta) \circ \sigma_X &; \sigma(1) = 1 \\
    0 &;  \sigma(1) = p+1
  \end{cases}
  \,.
  \\
  \intertext{
By analogous computations using~\eqref{eq:piTsigmalambda3} and \eqref{eq:piTsigmalambda4}, we get
  }
  \bigl( \alpha \cdot (\Val_R \circ T\beta) \bigr) \circ
  (\Swp_{1,p}T^q)_X \circ
  (T\sigma)_X \circ (\lambda T^{p+q-1})_X
  &=
  \begin{cases}
    0 &; \sigma(1) = 1 \\
    (\alpha \cdot \beta) \circ \sigma_X &;  \sigma(1) = p+1
  \end{cases}
  \,.
\end{align*}
Summing the last two equations and using~\eqref{eq:TDerProd}, we obtain
\begin{equation*}
  \Val_R \circ T(\alpha \cdot \beta) \circ (T\sigma)_X \circ (\lambda T^{p+q-1})_X
  =
  (\alpha \cdot \beta) \circ \sigma_X
  \,.
\end{equation*}
Inserting this into~\eqref{eq:WedgeDiff01}, we conclude that
\begin{equation*}
  \Val_R \circ T(\alpha \wedge \beta) \circ (\lambda T^{p+q-1})_X
  =  \alpha \wedge \beta
  \,,
\end{equation*}
which shows that $\alpha \wedge \beta$ satisfies~\eqref{eq:CubDiffDef}.

It follows from the associativity of $\Rmult: R \times R \to R$ that $(\alpha \cdot \beta) \cdot \gamma = \alpha \cdot (\beta \cdot \gamma)$. Using standard arguments for the shuffles, we can show that this in turn implies that the wedge product is associative. To show that the product is commutative graded, we first note that
\begin{equation*}
  T^p\pi^q \circ (\Swp_{q,p})_X = \pi^q T^p
  \,,\quad
  \pi^p T^q \circ (\Swp_{q,p})_X = T^q \pi^p
  \,,
\end{equation*}
where $\Swp_{q,p} \in S_{p+q}$ is the block shuffle defined in~\eqref{eq:ijShuffle}. This implies
\begin{equation*}
\begin{split}
  (\alpha \cdot \beta) \circ (\Swp_{q,p})_X
  &= \Rmult\circ\bigl(
       \alpha\circ (T^p\pi^q)_X,\beta\circ(\pi^pT^q)_X
     \bigr) \circ (\Swp_{q,p})_X
  \\
  &= \Rmult\circ\bigl(
       \alpha\circ (\pi^q T^p)_X,\beta\circ(T^q \pi^p)_X
     \bigr)
  \\
  &= \Rmult\circ\bigl(
       \beta\circ(T^q \pi^p)_X, \alpha\circ (\pi^q T^p)_X
     \bigr)
  \\
  &= \beta \cdot \alpha
  \,,
\end{split}
\end{equation*}
where we have used that $\Rmult$ is commutative. By the defining Equation~\eqref{eq:CubeWedge} of the wedge product, we obtain
\begin{equation*}
\begin{split}
  \beta \wedge \alpha
  &=
  \sum_{\mathclap{\sigma \in \Shuff(q,p)^{-1}}}
  (-1)^{|\sigma|} (\beta \cdot \alpha) \circ \sigma_X
  \\
  &=
  \sum_{\mathclap{\sigma \in \Shuff(q,p)^{-1}}}
  (-1)^{|\sigma|} (\alpha \cdot \beta) \circ (\Swp_{q,p})_X \circ \sigma_X
  \\
  &=
  (-1)^{qp}
  \sum_{\mathclap{\rho \in \Swp_{q,p} \Shuff(q,p)^{-1}}}
  (-1)^{|\rho|} (\alpha \cdot \beta) \circ \rho_X
  \\
  &=
  (-1)^{qp}
  \sum_{\mathclap{\rho \in \Shuff(p,q)^{-1}}}
  (-1)^{|\rho|} (\alpha \cdot \beta) \circ \rho_X
  \\
  &=
  (-1)^{|\alpha|\,|\beta|} \alpha \wedge \beta
  \,,
\end{split}
\end{equation*}
where we have used that by Equation~\eqref{eq:ShuffSwp} we have
\begin{equation*}
  \Swp_{q,p} \Shuff(q,p)^{-1}
  = \Swp_{p,q}^{-1} \Shuff(q,p)^{-1}
  = \Shuff(p,q)^{-1}
  \,.
\end{equation*}
This shows that the wedge product is commutative graded. We conclude that $\Omega_\mathrm{cub}(X)$ with the wedge product is a commutative graded ring. As is the case for any commutative graded ring, this implies that $\Omega_\mathrm{cub}(X)$ is a commutative graded algebra over the ring of degree 0 elements $\Omega^0_\mathrm{cub}(X) = \Hom(X,R)$.

\subsection{The differential}

In this section we prove that $d$ is a graded derivation of $\Omega_\mathrm{cub}(X)$ and a differential, $d^2 = 0$. First, we have to show that $d\omega$ is a cubical form. Analogous to the proof of the antisymmetry of $\alpha \wedge \beta$ in Section~\ref{sec:ProofRing}, for every $(1,n)$-unshuffle $\sigma$ and every $\rho \in S_{1+n}$, we have $\sigma \circ \rho = (\mu_\rho,\nu_\rho) \circ \sigma_\rho$, where $\mu_\rho = \id \in S_1$ and $\nu_\rho \in S_n$. We obtain
\begin{equation*}
\begin{split}
  d\omega \circ \rho_X
  &=
  \sum_{\mathclap{\sigma \in \Shuff(1,n)^{-1} }}
  (-1)^{|\sigma|}\Val_R \circ T\omega \circ \sigma_X \circ \rho_X
  \\
  &=
  \sum_{\mathclap{\sigma \in \Shuff(1,n)^{-1} }}
  (-1)^{|\sigma|}\Val_R \circ T\omega \circ (T\nu_\rho)_X
  \circ (\sigma_\rho)_X
  \\
  &=
  \sum_{\mathclap{\sigma \in \Shuff(1,n)^{-1} }}
  (-1)^{|\sigma| + |\nu_\rho|}\Val_R \circ T\omega
  \circ (\sigma_\rho)_X
  \\
  &=
  \sum_{\mathclap{\sigma_\rho \in \Shuff(1,n)^{-1} }}
  (-1)^{|\rho| + |\sigma_\rho|}\Val_R \circ T\omega
  \circ (\sigma_\rho)_X
  \\
  &=
  (-1)^{|\rho|} d\omega
  \,.
\end{split}
\end{equation*}
This shows that $d\omega$ is antisymmetric. To show that $d\omega$ satisfies~\eqref{eq:CubDiffDef}, we compute
\begin{equation}
\label{eq:domegaDiff1}
\begin{split}
  \Val_R \circ T(d\omega) \circ \lambda_{T^n X}
  &=
  \Val_R \circ T\Bigl(
  \sum_{\mathclap{\sigma \in \Shuff(1,n)^{-1} }}
  (-1)^{|\sigma|} \Val_R \circ T\omega \circ \sigma_X
  \Bigr) \circ \lambda_{T^n X}
  \\
  &=
  \sum_{\mathclap{\sigma \in \Shuff(1,n)^{-1} }}
  (-1)^{|\sigma|}
  \Val_R \circ T\Val_R \circ T^2 \omega \circ T\sigma_X
  \circ \lambda_{T^n X}
  \,.
\end{split}
\end{equation}
A $(1,n)$-shuffle must map either $1$ or $2$ to $1$. In the first case the shuffle is the identity. In the second case, the shuffle is of the form $T\mu \circ \tau T^{n-1}$ for $\mu \in \Shuff(1,n-1)$. It follows that an unshuffle $\sigma \in \Shuff(1,n)^{-1}$ is either the identity, $\sigma = T^{n+1}$, or of the form $\sigma = \tau T^{n-1} \circ T\mu$ for $\mu \in \Shuff(1,n-1)^{-1}$. In the case $\sigma = T^{n+1}$, the summand is
\begin{equation*}
\begin{split}
  \Val_R \circ T\Val_R \circ T^2 \omega \circ T\sigma_X
  \circ \lambda_{T^n X}
  &=
  \Val_R \circ T\Val_R \circ T^2 \omega \circ \id_{T^{n+2}X}
  \circ \lambda_{T^n X}
  \\
  &=
  \Val_R \circ T\Val_R \circ T^2 \omega \circ \lambda_{T^n X}
  \\
  &=
  \Val_R \circ T\Val_R \circ \lambda_R \circ T\omega
  \\
  &=
  \Val_R \circ T\omega \circ \sigma_X
  \,,
\end{split}
\end{equation*}
where we have used the naturality of the vertical lift, $T^2\omega \circ \lambda_{T^n X} = \lambda_R \circ T\omega$, then Equation~\eqref{eq:EtaTEtalambda} for the tangent-stable ring $R$, and finally $\sigma_X = \id_{T^{n+1}X}$. In the case $\sigma = \tau T^{n-1} \circ T\mu$, the summand is
\begin{equation*}
\begin{split}
  &\hphantom{{}={}}
  \Val_R \circ T\Val_R \circ T^2 \omega \circ T\sigma_X
  \circ \lambda_{T^n X}
  \\
  &=
  \Val_R \circ T\Val_R \circ T^2 \omega \circ (T\tau)_{T^{n-1}X} \circ T^2 \mu_X \circ \lambda_{T^n X}
  \\
  &=
  \Val_R \circ T\Val_R \circ T^2 \omega \circ (T\tau)_{T^{n-1}X} \circ \lambda_{T^n X} \circ T\mu_X
  \\
  &=
  \Val_R \circ T\Val_R \circ T^2 \omega \circ
  (\tau T)_{T^{n-1}X} \circ (T\lambda)_{T^{n-1} X} \circ
  \tau_{T^{n-1} X} \circ T\mu_X
  \\
  &=
  \Val_R \circ T\Val_R \circ \tau_R \circ
  T^2 \omega \circ
  (T\lambda)_{T^{n-1} X} \circ \tau_{T^{n-1} X} \circ T\mu_X
  \\
  &=
  \Val_R \circ
  T\Val_R \circ T^2 \omega \circ (T\lambda)_{T^{n-1} X}
  \circ \sigma_X
  \\
  &=
  \Val_R \circ
  T(\Val_R \circ T\omega \circ \lambda_{T^{n-1} X})
  \circ \sigma_X
  \\
  &=
  \Val_R \circ T\omega \circ \sigma_X
  \,,
\end{split}
\end{equation*}
where we have used the functoriality of $T$, the naturality of the vertical lift, the second diagram of Axiom~\eqref{eq:TanFun5} and the involutivity $\tau \circ \tau = T^2$ of the symmetric structure (Definition~\ref{def:SymStruc}), the naturality of $\tau$, followed by Equation~\eqref{eq:EtaTEtaSymmetric01}, the identification $\sigma_X = \tau_{T^{n-1}X} \circ T\mu_X$, the functoriality of $T$ again, and finally the assumption~\eqref{eq:CubDiffDef} for $\omega$. This shows that
\begin{equation}
\label{eq:domegaDiff2}
  \Val_R \circ T\Val_R \circ T^2 \omega \circ T\sigma_X
  \circ \lambda_{T^n X}
  =
  \Val_R \circ T\omega \circ \sigma_X
  \,,
\end{equation}
for all $\sigma \in \Shuff(1,n)^{-1}$ and all $\omega$ satisfying~\eqref{eq:CubDiffDef}. By inserting~\eqref{eq:domegaDiff2} into \eqref{eq:domegaDiff1}, we obtain
\begin{equation*}
\begin{split}
  \Val_R \circ T(d\omega) \circ \lambda_{T^n X}
  &=
  \sum_{\mathclap{\sigma \in \Shuff(1,n)^{-1} }}
  (-1)^{|\sigma|}
  \Val_R \circ T\omega \circ \sigma_X
  \\
  &= d\omega
  \,,
\end{split}
\end{equation*}
which is condition~\eqref{eq:CubDiffDef} for $d\omega$. We conclude that $d\omega$ is a differential cubical $(n+1)$-form.

Next, we will show that $d$ is a graded derivation. The differential of the wedge product of $\alpha \in \Omega^p_\mathrm{cub}(X)$ and $\beta \in \Omega^q_\mathrm{cub}(X)$ is given by
\begin{equation}
\label{eq:dWedge01}
\begin{split}
  d(\alpha \wedge \beta)
  &=
  \sum_{\mathclap{\rho \in \Shuff(1,p+q)^{-1} }}
  (-1)^{|\rho|}\Val_R \circ T(\alpha \wedge \beta) \circ \rho_X
  \\
  &=
  \sum_{\mathclap{\rho \in \Shuff(1,p+q)^{-1} }}
  (-1)^{|\rho|}\Val_R \circ T
  \Bigl(
  \sum_{\mathclap{\sigma \in \Shuff(p,q)^{-1}}} (-1)^{|\sigma|}
  (\alpha \cdot \beta) \circ \sigma_X
  \Bigr) \circ \rho_X
  \\
  &=
  \sum_{
    \mathclap{
      \substack{
        \rho \in \Shuff(1,p+q)^{-1} \\
        \sigma \in \Shuff(p,q)^{-1}
      }
    }
  }
  (-1)^{|\sigma| + |\rho|}\Val_R \circ T
  (\alpha \cdot \beta) \circ (T\sigma \circ \rho)_X
  \,.
\end{split}
\end{equation}
Now we insert both summands of the right side of~\eqref{eq:TDerProd} into the sum of the right side of~\eqref{eq:dWedge01}. For the first term, we get
\begin{equation}
\label{eq:dWedge02}
\begin{split}
  &
  \sum_{
    \mathclap{
      \substack{
        \rho \in \Shuff(1,p+q)^{-1} \\
        \sigma \in \Shuff(p,q)^{-1}
      }
    }
  }
  (-1)^{|\sigma| + |\rho|}
  \bigl( (\Val_R \circ T\alpha) \cdot \beta \bigr) \circ (T\sigma \circ \rho)_X
  \\
  ={}&
  \sum_{
    \mathclap{
      \substack{
        \sigma_\rho \in \Shuff(p+1,q)^{-1} \\
        \mu_\rho \in \Shuff(1,p)^{-1}
      }
    }
  }
  (-1)^{|\mu_\rho| + |\sigma_\rho|}
  \bigl( (\Val_R \circ T\alpha) \cdot \beta \bigr) \circ
  (\mu_\rho T^q)_X \circ (\sigma_\rho)_X
  \\
  ={}&
  \sum_{\mathclap{\sigma_\rho \in \Shuff(p+1,q)^{-1}}}
  (-1)^{|\sigma_\rho|}
  \Bigl(\Bigl( \sum_{\mathclap{\mu_\rho \in \Shuff(1,p)^{-1}}} (-1)^{|\mu_\rho|}
    (\Val_R \circ T\alpha) \circ (\mu_\rho)_X \Bigr)\cdot \beta \Bigr) \circ
  (\sigma_\rho)_X
  \\
  ={}&
  d\alpha \wedge \beta
  \,,
\end{split}
\end{equation}
where we have used Lemma~\ref{lem:ShuffShuff}. Analogously, we have
\begin{equation}
\label{eq:dWedge03}
\begin{split}
  &
  \sum_{
    \mathclap{
      \substack{
        \rho \in \Shuff(1,p+q)^{-1} \\
        \sigma \in \Shuff(p,q)^{-1}
      }
    }
  }
  (-1)^{|\sigma| + |\rho|}
  \bigl((\Val_R \circ T\beta) \cdot \alpha \bigr) \circ
  (T\Swp_{p,q})_X \circ (T\sigma \circ \rho)_X
  \\
  ={}&
  \sum_{
    \mathclap{
      \substack{
        \sigma_\rho \in \Shuff(q+1,p)^{-1} \\
        \mu_\rho \in \Shuff(1,q)^{-1}
      }
    }
  }
  (-1)^{|\Swp_{p,q}|+|\mu_\rho| + |\sigma_\rho|}
  \bigl( (\Val_R \circ T\beta) \cdot \alpha \bigr) \circ
  (\mu_\rho T^p)_X \circ (\sigma_\rho)_X
  \\
  ={}&
  (-1)^{pq}
  \sum_{\mathclap{\sigma_\rho \in \Shuff(q+1,p)^{-1}}}
  (-1)^{|\sigma_\rho|}
  \Bigl(\Bigl( \sum_{\mathclap{\mu_\rho \in \Shuff(1,q)^{-1}}} (-1)^{|\mu_\rho|}
    (\Val_R \circ T\beta) \circ (\mu_\rho)_X \Bigr)\cdot \alpha \Bigr) \circ
  (\sigma_\rho)_X
  \\
  ={}&
  (-1)^{pq}
  d\beta \wedge \alpha
  \\
  ={}&
  (-1)^p \alpha \wedge d\beta
  \,,
\end{split}
\end{equation}
where we have used Lemma~\ref{lem:ShuffShuff} and, in the last step, the graded commutativity of the wedge product. By inserting~\eqref{eq:TDerProd} into~\eqref{eq:dWedge01} and then inserting~\eqref{eq:dWedge02} and~\eqref{eq:dWedge03}, we obtain
\begin{equation*}
  d(\alpha \wedge \beta)
  =
  d\alpha \wedge \beta
    + (-1)^{p}\alpha \wedge d\beta
  \,,
\end{equation*}
which shows that $d$ is a graded derivation of $\Omega_\mathrm{cub}(X)$. Finally, we will show that $d^2 = 0$. Let $\omega$ be an $n$-form. We have the equality
\begin{equation}
\label{eq:dWedge04}
\begin{split}
  \Val_R \circ T\Val_R \circ T^2\omega \circ (\tau T^n)_X
  &=
  \Val_R \circ T\Val_R \circ T^2\omega \circ \tau_{T^n X}
  \\
  &=
  \Val_R \circ T\Val_R \circ \tau_R \circ T^2\omega
  \\
  &=
  \Val_R \circ T\Val_R \circ T^2\omega
  \,,
\end{split}
\end{equation}
where we have used the naturality of $\tau$ and Equation~\eqref{eq:EtaTEtaSymmetric01}. Now we can compute
\begin{equation*}
\begin{split}
  d^2 \omega
  &=
  \sum_{\mathclap{\rho \in \Shuff(1,n+1)^{-1}}}(-1)^{|\rho|}
  \Bigl( \Val_R \circ T
  \bigl( \sum_{\mathclap{\sigma \in \Shuff(1,n)^{-1}}}(-1)^{|\sigma|}
  \Val_R \circ T\omega \circ \sigma_X \bigr) \Bigr) \circ \rho_X
  \\
  &=
  \sum_{
    \mathclap{
      \substack{
        \rho \in \Shuff(1,n+1)^{-1} \\
        \sigma \in \Shuff(1,n)^{-1}
      }
    }
  }(-1)^{|\sigma|+|\rho|}
  \Val_R \circ T\Val_R \circ T^2\omega \circ (T\sigma \circ \rho)_X
  \\
  &=
  \sum_{
    \mathclap{
      \substack{
        \sigma_\rho \in \Shuff(2,n)^{-1} \\
        \mu_\rho \in \Shuff(1,1)^{-1}
      }
    }
  }
  (-1)^{|\mu_\rho|+|\sigma_\rho|}
  \Val_R \circ T\Val_R \circ T^2\omega \circ (\mu_\rho T^n \circ \sigma_\rho)_X
  \\
  &=
  \sum_{
    \mathclap{
        \sigma_\rho \in \Shuff(2,n)^{-1}
    }
  }
  (-1)^{|\sigma_\rho|}
  \Bigl(
  \sum_{
    \mathclap{
      \mu_\rho \in S_2
    }
  }
  (-1)^{|\mu_\rho|}
  \Val_R \circ T\Val_R \circ T^2\omega \circ (\mu_\rho T^n)_X
  \Bigr) \circ (\sigma_\rho)_X
  \\
  &=
  \sum_{
    \mathclap{
        \sigma_\rho \in \Shuff(2,n)^{-1}
    }
  }
  (-1)^{|\sigma_\rho|}
  \bigl(
    \Val_R \circ T\Val_R \circ T^2\omega
    - \Val_R \circ T\Val_R \circ T^2\omega \circ (\tau T^n)_X
  \bigr) \circ (\sigma_\rho)_X
  \\
  &= 0
  \,,
\end{split}
\end{equation*}
where we have used Lemma~\ref{lem:ShuffShuff} and~\eqref{eq:dWedge04}. This concludes the proof of Statement~(i) of Theorem~\ref{thm:CartanCube}. \qed

\subsection{Prolongation of vector fields}

It remains to prove Statement~(ii) of Theorem~\ref{thm:CartanCube}. We begin with some auxiliary statements. By the identification
\begin{equation*}
  T^{n+1}X = T^n TX = T^n X'
\end{equation*}
where $X' = TX$, a map $\omega\colon T^{n+1} X \to R$ can be viewed as a map
\begin{equation*}
  \omega'\colon T^n X' \longrightarrow R
  \,.
\end{equation*}
If $\omega$ is antisymmetric under the group $S_{n+1}$ acting on $T^{n+1} X$, it is antisymmetric under the subgroup $S_n \subset S_{n+1}$ that acts by $\sigma_{X'} = (\sigma T)_X$ on $T^n TX$. If $\omega$ is $R$-linear at the positions $1, \ldots, n+1$, then $\omega'$ is $R$-linear at the positions $1, \ldots, n$, since $+^{(i)}_{T^{n-1}X'} = +^{(i)}_{T^n X}$ and $\kappa^{(i)}_{T^{n-1}X'} = \kappa^{(i)}_{T^n X}$ for $1 \leq i \leq n$. The defining Condition~\eqref{eq:CubDiffDef} for differential cubical forms is not affected by the equality $T^{n+1} = T^n T$. We conclude that we have the tautological injective map
\begin{equation}
\label{eq:FormsTautMap}
\begin{aligned}
  \Omega^{n+1}_\mathrm{cub}(X)
  &\longrightarrow
  \Omega^{n}_\mathrm{cub}(X')
  \\
  \omega &\longmapsto \omega'
  \,.
\end{aligned}
\end{equation}

For every vector field $v\colon X \to TX$, we define
\begin{equation}
\label{eq:vPrimeDef}
  v' \coloneqq \tau_X \circ Tv
  \colon X' \longrightarrow TX'
  \,,
\end{equation}
which satisfies
\begin{equation*}
\begin{split}
  \pi_{X'} \circ v'
  &=
  \pi_{TX} \circ \tau_X \circ Tv
  =
  T\pi_X \circ Tv
  =
  T\id_X
  \\
  &=
  \id_{X'}
  \,,
\end{split}
\end{equation*}
so $v'$ is a vector field on $X'$. We call $v'$ the \textbf{prolongation} of $v$. For $\omega \in \Omega^{n+1}_\mathrm{cub}(X)$ with $n \geq 1$ we have
\begin{equation}
\label{eq:InnPrime}
\begin{split}
  \iota_{v'}\omega'
  &= (-1)^{n+1} \omega' \circ T^{n-1} v'
  \\
  &= (-1)^{n+1} \omega \circ
  T^{n-1}(\tau_X \circ Tv)
  \\
  &= (-1)^{n+1} \omega \circ
  (T^{n-1}\tau)_X \circ T^n v
  \\
  &= -(-1)^{n+1} \omega \circ T^n v
  \\
  &= (\iota_v\omega)'
  \,,
\end{split}
\end{equation}
where we have used the definition~\eqref{eq:CubeInnDer} of the inner derivative, the definition of~$v'$, the functoriality of~$T$, and the antisymmetry of~$\omega$.

\begin{Lemma}
\label{lem:ProlongBracket}
Let $v$, $w$ be vector fields on $X$ and $v'$, $w'$ their prolongations to vector fields on $X' = TX$. Then $[v',w']$ is the prolongation of $[v,w]$, that is,
\begin{equation*}
  [v',w'] = [v,w]'
  \,.
\end{equation*}
\end{Lemma}

\begin{proof}
By applying the tangent functor to Equation~\eqref{eq:deltaBracketRel} we obtain the equation
\begin{equation}
\label{eq:ProlongBracket01}
  T\delta(v,w) = T(\lambda_2)_X \circ (Tw, T[v,w])
  \,,
\end{equation}
where $\delta(v,w)$ is given by Equation~\eqref{eq:DeltaOrig} and the vertical lift $\lambda_2$ by Equation~\eqref{eq:VertLiftExt}. To compute the left side of~\eqref{eq:ProlongBracket01} we will need the relation
\begin{equation*}
\begin{split}
  T(-_{TX})
  &= \tau_{TX} \circ -_{TTX} \circ (\tau_{TX} \times \tau_{TX})
  \\
  &= \tau_{TX} \circ T\tau_X \circ -_{T^2 X} \circ
  (T\tau_X \times T\tau_X) \circ
  (\tau_{TX} \times \tau_{TX})
  \\
  &= \tau_{TX} \circ T\tau_X \circ -_{T^2 X} \circ
    \bigl( (T\tau_X \circ \tau_{TX}) \times (T\tau_X \circ \tau_{TX}) \bigr)
  \,,
\end{split}
\end{equation*}
where we have used that~$\tau$ is a morphism of bundles of abelian groups~\eqref{eq:TanFun2} and the naturality of $-_X$. We also need the relation
\begin{equation*}
\begin{split}
  Tv' \circ w'
  &= T(\tau_X \circ Tv) \circ \tau_X \circ Tw
  \\
  &= T\tau_X \circ T^2 v \circ \tau_X \circ Tw
  \\
  &= T\tau_X \circ \tau_{TX} \circ T^2 v \circ Tw
  \,,
\end{split}
\end{equation*}
where we have used the naturality of $\tau$. Using the last two relations, we compute
\begin{equation}
\label{eq:ProlongBracket02}
\begin{split}
  T\delta(v,w)
  &= T(-_{TX}) \circ\bigl( T(Tw \circ v), T(\tau_X \circ Tv \circ w) \bigr)
  \\
  &= \tau_{TX} \circ T\tau_X \circ -_{T^2 X} \circ
    \bigl( (T\tau_X \circ \tau_{TX}) \times (T\tau_X \circ \tau_{TX}) \bigr) \circ
  \\
  &\qquad{}
  \bigl( T^2 w \circ Tv, T\tau_X \circ T^2 v \circ Tw \bigr)
  \\
  &= \tau_{TX} \circ T\tau_X \circ -_{T^2 X} \circ
     ( Tw' \circ v', T\tau_X \circ \tau_{TX} \circ T\tau_X \circ T^2 v \circ Tw )
  \\
  &= \tau_{TX} \circ T\tau_X \circ -_{T^2 X} \circ
     ( Tw' \circ v', \tau_{TX} \circ T\tau_X \circ \tau_{TX} \circ T^2 v \circ Tw )
  \\
  &= \tau_{TX} \circ T\tau_X \circ -_{T^2 X} \circ
     ( Tw' \circ v', \tau_{TX} \circ Tv' \circ w')
  \\
  &= \tau_{TX} \circ T\tau_X \circ \delta(v',w')
  \,,
\end{split}
\end{equation}
where we have also used the braid relations for $\tau$. Let us abbreviate $u = [v,w]$. For the right side of~\eqref{eq:ProlongBracket01}, we obtain
\begin{equation}
\label{eq:ProlongBracket03}
\begin{split}
  &\hphantom{=}{} T\bigl( (\lambda_2)_X \circ (w,u) \bigr)
  \\
  &=
  T\tau_X \circ T(+_{TX}) \circ (T^2 0_X \times T\lambda_X) \circ (Tw,Tu)
  \\
  &=
  T\tau_X \circ \tau_{TX} \circ T\tau_X \circ +_{T^2 X} \circ
  \bigl( (T\tau_X \circ \tau_{TX}) \times (T\tau_X \circ \tau_{TX}) \bigr) \circ
  \\
  &\qquad{}
  (T^2 0_X \times T\lambda_X) \circ (Tw,Tu)
  \\
  &=
  \tau_{TX} \circ T\tau_X \circ \tau_{TX} \circ +_{T^2 X} \circ
  (T0_{TX} \times \lambda_{TX}) \circ (\tau_X \times \tau_X) \circ (Tw,Tu)
  \\
  &=
  \tau_{TX} \circ T\tau_X \circ (\lambda_2)_{TX} \circ (w',u')
  \,.
\end{split}
\end{equation}
By inserting~\eqref{eq:ProlongBracket02} and~\eqref{eq:ProlongBracket03} into~\eqref{eq:ProlongBracket01}, we get the equation
\begin{equation*}
  \tau_{TX} \circ T\tau_X \circ \delta(v',w')
  =
  \tau_{TX} \circ T\tau_X \circ (\lambda_2)_{TX} \circ (w', [v,w]')
  \,,
\end{equation*}
where $\delta(v',w')$ denotes the morphism~\eqref{eq:DeltaOrig} formed for the object $X' = TX$. Since $\tau_{TX} \circ T\tau_X$ is an isomorphism, it follows that $\delta(v',w') = (\lambda_2)_{X'} \circ (w', [v,w]')$. On the other hand, $\delta(v',w') = (\lambda_2)_{X'} \circ (w', [v',w'])$ by Equation~\eqref{eq:deltaBracketRel} for $X'$. By Diagram~\eqref{diag:lambda2pullback}, $(\lambda_2)_{X'}$ is the pullback of the split monomorphism $0_{X'}$, hence a monomorphism. We conclude that $[v',w'] = [v,w]'$.
\end{proof}

\begin{Lemma}
\label{lem:FormLambdaZero}
Assume that $R$ has no $2$-torsion; let $\omega\colon T^n X \to R$ be a morphism, where $n \geq 2$. If $\omega$ is antisymmetric then
\begin{equation}
\label{eq:FormLambdaZero01}
  \omega \circ T^{i-1}\lambda_{T^{n-i-1}X} = 0
\end{equation}
for all $1 \leq i \leq n-1$. If $\omega$ is antisymmetric and $R$-linear at every position, then
\begin{equation}
\label{eq:FormLambdaZero02}
  \omega \circ T^{i-1}(\lambda_2)_{T^{n-i-1}X} = 0
\end{equation}
for all $1 \leq i \leq n-1$.
\end{Lemma}

\begin{proof}
Since $\tau \circ \lambda = \lambda$ by the left diagram of~\eqref{eq:TanFun5}, we have
\begin{equation*}
\begin{split}
  \omega \circ (T^{i-1}\lambda T^{n-i-1})_X
  &=
  \omega \circ \bigl(T^{i-1}(\tau \circ \lambda) T^{n-i-1} \bigr)_X
  \\
  &=
  - \omega \circ (T^{i-1}\tau T^{n-i-1})_X \circ
  (T^{i-1}\lambda T^{n-i-1})_X
  \\
  &=
  - \omega \circ
  (T^{i-1}\lambda T^{n-i-1})_X
  \,,
\end{split}
\end{equation*}
where we have used that $\omega$ is antisymmetric. Hence $2\omega \circ T^{i-1}\lambda_{T^{n-i-1}X} = 0$. Since $R$ is assumed to have no $2$-torsion, it follows that $\omega \circ T^{i-1}\lambda_{T^{n-i-1}X} = 0$. Assume now that $\omega$ is antisymmetric and $R$-linear at every position. $R$-linearity implies that
\begin{equation}
\label{eq:LinFormZero01}
  \omega \circ T^{i-1}0_{T^{n-i}X} = 0
  \,,
\end{equation}
for all $1 \leq i \leq n$. With this relation, we obtain
\begin{equation*}
\begin{split}
  \omega \circ T^{i-1}(\lambda_2)_{T^{n-i-1}X}
  &=
  \omega \circ (T^{i-1}\tau T^{n-i-1})_X \circ
  \bigl( T^{i-1}\bigl(+T \circ (T0 \times_0 \lambda)\bigr)T^{n-i-1}
  \bigr)_X
  \\
  &=
  - \omega \circ
  \bigl(T^{i-1}(+)T^{n-i}\bigr)_X \circ
  \bigl( T^{i-1}(T0 \times_0 \lambda)T^{n-i-1} \bigr)_X
  \\
  &=
  - \omega \circ T^i 0_{T^{n-i-1}X} \circ \pr_1
  - \omega \circ T^{i-1}\lambda_{T^{n-i-1}X} \circ \pr_2
  \\
  &=
  0
  \,,
\end{split}
\end{equation*}
where we have used the defining Equation~\eqref{eq:VertLiftExt} of $\lambda_2$, the antisymmetry of $\omega$, the $R$-linearity of $\omega$, and in the last step~\eqref{eq:LinFormZero01} and~\eqref{eq:FormLambdaZero01}.
\end{proof}

For every $f\colon X \to R$, we define the function
\begin{equation*}
  f' \coloneqq f \circ \pi_X\colon X' \longrightarrow R
\end{equation*}
on $X' \coloneqq TX$. For $\omega \in \Omega^{n+1}_\mathrm{cub}(X)$, where $n \geq 0$ and $(\pi^0)_{X'} \coloneqq \id_{X'}$, we have
\begin{equation}
\label{eq:InnPrime2}
\begin{split}
  f' \wedge' \omega'
  &= \Rmult \circ \bigl( f' \circ (\pi^n)_{X'}, \omega' \bigr)
  \\
  &= \Rmult \circ \bigl( (f \circ \pi_X) \circ (\pi^n T)_X, \omega \bigr)
  \\
  &= \Rmult \circ (f \circ (\pi^{n+1})_X, \omega)
  \\
  &= f \wedge \omega
  \,,
\end{split}
\end{equation}
where the first wedge product is in $\Omega(X')$ and the second wedge product in $\Omega(X)$.

The prolongation of $fv$ is in general not equal to $f'v'$, since it also contains a vertical term involving the derivative of $f$. The next statement shows that, if $R$ has no $2$-torsion, this additional term is in the kernel of every cubical form.

\begin{Lemma}
\label{lem:fPrimevPrime}
Assume that $R$ has no $2$-torsion; let $\omega\colon T^{n+1} X \to R$ be a cubical $(n+1)$-form on $X$ and $\omega'$ the corresponding $n$-form on $X' = TX$. Then
\begin{equation}
\label{eq:fPrimevPrime}
  \iota_{(fv)'} \omega'
  = \iota_{f'v'} \omega'
\end{equation}
for all $f\colon X \to R$ and $v \in \calX(X)$.
\end{Lemma}

\begin{proof}
We have the equality
\begin{equation}
\label{eq:fvPrime}
\begin{split}
  (fv)'
  &=
  \tau_X \circ T(fv)
  \\
  &= \tau_X \circ T\kappa_X \circ (Tf, Tv)
  \\
  &= \tau_X \circ +_{TX} \circ
  \bigl( T_{(1)} \kappa_X \circ (Tf, v \circ \pi_X),
  T_{(2)} \kappa_X \circ (f \circ \pi_X, Tv) \bigr)
  \\
  &= {T+_X} \circ \nu_{2,X}^{-1} \circ (\tau_X \times_{TX} \tau_X) \circ
  \bigl( T_{(1)} \kappa_X \circ (Tf, v \circ \pi_X),
  T_{(2)} \kappa_X \circ (f', Tv) \bigr)
  \\
  &= {T+_X} \circ \nu_{2,X}^{-1} \circ
  \bigl( \tau_X \circ (\lambda_2)_X \circ
  \bigl( fv \circ \pi_X, (df)(v \circ \pi_X) \bigr),
  f' v' \bigr)
  \,,
\end{split}
\end{equation}
where we have used the definition~\eqref{eq:vPrimeDef} of the prolongation, the definition~\eqref{eq:fvDef} of $fv$, Proposition~\ref{prop:TSumPartials}, the naturality of $\pi$ and the definition of $f'$, that $\tau$ is a morphism of bundles of abelian groups in the form of the right diagram of~\eqref{eq:TanFun2}, and, in the last step, the following two relations: By Diagram~\eqref{diag:ScalarMult2} and $(\pi_R, \Val_R) \circ Tf = (f \circ \pi_X, df)$, where $df = \Val_R \circ Tf$ by~\eqref{eq:dCubeDef}, we have $\Tpart{1}\kappa_X \circ (Tf, v \circ \pi_X) = (\lambda_2)_X \circ \bigl( fv \circ \pi_X, (df)(v \circ \pi_X) \bigr)$; by Diagram~\eqref{diag:ScalarMult3} and the definition of $v'$, we have $\tau_X \circ \Tpart{2}\kappa_X \circ (f', Tv) = \kappa_{TX} \circ (f', \tau_X \circ Tv) = f'v'$. For the inner derivative, we obtain
\begin{equation*}
\begin{split}
  \iota_{(fv)'} \omega'
  &=
  (-1)^{n+1} \omega \circ T^{n-1}(fv)'
  \\
  &=
  (-1)^{n} \omega \circ T^{n-1}(\lambda_2)_X \circ
  T^{n-1}\bigl( fv \circ \pi_X, (df)(v \circ \pi_X) \bigr)
  \\
  &\quad{}+ (-1)^{n+1} \omega \circ T^{n-1}(f'v')
  \\
  &=
  \iota_{f' v'} \omega'
  \,,
\end{split}
\end{equation*}
where we have used the definition~\eqref{eq:CubeInnDer} of the inner derivative, Equation~\eqref{eq:fvPrime}, the $R$-linearity of $\omega$ at every position~\eqref{eq:CubForm1}, the antisymmetry of $\omega$, and, in the last step, Equation~\eqref{eq:FormLambdaZero02} of Lemma~\ref{lem:FormLambdaZero} applied to the $(n+1)$-form $\omega$ with $i = n$, that is, $\omega \circ T^{n-1}(\lambda_2)_X = 0$, which is where the assumption that $R$ has no $2$-torsion enters.
\end{proof}

\subsection{Properties of the inner derivative}

\begin{Proposition}
\label{prop:InnDerProp}
The inner derivative $\iota_v$ is a graded derivation of degree $-1$ of $\Omega_\mathrm{cub}(X)$ that satisfies
\begin{align}
\label{eq:InnDerProp03}
  \iota_{v+w} \omega
  &= \iota_v \omega + \iota_w \omega
  \\
\intertext{and, if $R$ has no 2-torsion,}
\label{eq:InnDerProp02}
  \iota_{fv} \omega
  &= f \wedge \iota_v \omega
  \\
\label{eq:InnDerProp01}
  \iota_w \iota_v \omega
  &= - \iota_v \iota_w \omega
\end{align}
for all $f \in \Omega^0_\mathrm{cub}(X)$ and $v, w \in \calX(X)$. For forms $\omega$ of degree at most $1$, Equation~\eqref{eq:InnDerProp02} holds without the assumption on $2$-torsion.
\end{Proposition}

\begin{proof}
Let $\omega \in \Omega^n_\mathrm{cub}(X)$ and $v \in \calX(X)$. The symmetric group $S_{n-1}$ acts on $T^{n-1}$ by natural transformations. The naturality implies that
\begin{equation}
\label{eq:SigmaNatTv}
  T^{n-1}v \circ \sigma_X
  =
  \sigma_{TX} \circ T^{n-1}v
\end{equation}
for all $\sigma \in S_{n-1}$. Therefore,
\begin{equation*}
\begin{split}
  (\iota_v \omega) \circ \sigma_X
  &= (-1)^{n+1} \omega \circ T^{n-1}v \circ \sigma_X
  \\
  &= (-1)^{n+1} \omega \circ (\sigma T)_X \circ T^{n-1}v
  \\
  &= (-1)^{|\sigma|}(-1)^{n+1} \omega \circ T^{n-1}v
  \\
  &= (-1)^{|\sigma|} \iota_v \omega
  \,,
\end{split}
\end{equation*}
which shows that $\iota_v \omega$ is antisymmetric. For $n = 1$, $\iota_v\omega$ is a function and there is nothing more to show, so assume that $n \geq 2$. For $1 \leq i \leq n-1$, the $R$-linearity of $\omega$ at the $i$-th position implies the $R$-linearity of $\iota_v\omega$ at the $i$-th position. The naturality of the vertical lift implies that
\begin{equation*}
\begin{tikzcd}
T^{n-1} X
\ar[d, "\lambda_{T^{n-2}X}"']
\ar[r, "T^{n-1}v"]
&
T^n X
\ar[d, "\lambda_{T^{n-1}X}"]
\\
T^n X
\ar[r, "T^n v"']
&
T^{n+1} X
\end{tikzcd}
\end{equation*}
commutes. With this, we obtain
\begin{equation*}
\begin{split}
  \Val_R \circ T(\iota_v \omega) \circ \lambda_{T^{n-2}X}
  &=
  (-1)^{n+1}\Val_R \circ T(\omega \circ T^{n-1} v) \circ \lambda_{T^{n-2}X}
  \\
  &=
  (-1)^{n+1}\Val_R \circ T\omega \circ T^n v \circ \lambda_{T^{n-2}X}
  \\
  &=
  (-1)^{n+1}\Val_R \circ T\omega \circ \lambda_{T^{n-1}X} \circ T^{n-1} v
  \\
  &=
  (-1)^{n+1}\omega \circ T^{n-1} v
  \\
  &=
  \iota_v \omega
  \,,
\end{split}
\end{equation*}
where we have used that $\omega$ satisfies~\eqref{eq:CubDiffDef}. We conclude that $\iota_v \omega$ is a differential cubical $(n-1)$-form.

Next, we show that $\iota_v$ is a graded derivation. Let $\alpha \in \Omega^p_\mathrm{cub}(X)$ and $\beta \in \Omega^q_\mathrm{cub}(X)$; let $n \coloneqq p+q$. If $p = 0$ or $q = 0$, the derivation property is trivial. From now on, let $p, q \geq 1$. The set of $(p,q)$-unshuffles can be decomposed as follows. Either $\sigma(p+q) = p+q$, then $\sigma_X =  (\rho T)_X$, for a unique $\rho \in \Shuff(p,q-1)^{-1}$. Or $\sigma(p+q) = p$, then $\sigma_X =  (\Swp_{q,p})_X \circ (\rho T)_X$ for a unique $\rho \in \Shuff(q,p-1)^{-1}$. This shows that we can write the wedge product as
\begin{equation}
\label{eq:InDer10}
\begin{split}
  \alpha \wedge \beta
  &=
  \sum_{\mathclap{\rho \in \Shuff(q,p-1)^{-1}}}
  (-1)^{|\rho|+pq} (\alpha \cdot \beta) \circ (\Swp_{q,p})_X \circ (\rho T)_X
  +
  \sum_{\mathclap{\rho \in \Shuff(p,q-1)^{-1}}}
  (-1)^{|\rho|} (\alpha \cdot \beta) \circ  (\rho T)_X
  \\
  &=
  (-1)^{pq}
  \sum_{\mathclap{\rho \in \Shuff(q,p-1)^{-1}}}
  (-1)^{|\rho|} (\beta \cdot \alpha) \circ (\rho T)_X
  +
  \sum_{\mathclap{\rho \in \Shuff(p,q-1)^{-1}}}
  (-1)^{|\rho|} (\alpha \cdot \beta) \circ  (\rho T)_X
  \,.
\end{split}
\end{equation}
In order to compute the inner derivative of every term, we compute
\begin{equation}
\label{eq:InDer11}
\begin{split}
  (\alpha \cdot \beta) \circ (\rho T)_X \circ T^{p+q-1}v
  &=
  \Rmult\circ\bigl(\alpha\circ (T^p\pi^q)_X,\beta\circ(\pi^pT^q)_X\bigr)
  \circ T^{p+q-1}v \circ \rho_X
  \\
  &=
  \Rmult\circ\bigl(
    \alpha \circ (T^p \pi^{q-1})_X,
    \beta \circ T^{q-1}v \circ (\pi^p T^{q-1})_X \bigr) \circ \rho_X
  \\
  &=
  (-1)^{q+1} \bigl(\alpha \cdot \iota_v \beta\bigr) \circ \rho_X
  \,,
\end{split}
\end{equation}
where we have used Equation~\eqref{eq:SigmaNatTv}, the definition of $\alpha \cdot \beta$, the relation $(T^p\pi^q)_X \circ T^{p+q-1}v = T^p\bigl((\pi^q)_X \circ T^{q-1}v\bigr) = (T^p\pi^{q-1})_X$, the relation $(\pi^pT^q)_X \circ T^{p+q-1}v = T^{q-1}v \circ (\pi^pT^{q-1})_X$, which follows from the naturality of $\pi^p$, and the definition of $\iota_v \beta$. Analogously,
\begin{equation}
\label{eq:InDer12}
  (\beta \cdot \alpha) \circ (\rho T)_X \circ T^{p+q-1}v
  =
  (-1)^{p+1} \bigl(\beta \cdot \iota_v \alpha\bigr) \circ \rho_X
  \,.
\end{equation}
By inserting \eqref{eq:InDer11} and \eqref{eq:InDer12} into \eqref{eq:InDer10}, we obtain
\begin{equation*}
\begin{split}
  \iota_v (\alpha \wedge \beta)
  &=
  (-1)^{p+q+1} (\alpha \wedge \beta) \circ T^{p+q-1}v
  \\
  &=
  (-1)^{p+q+1}(-1)^{pq}
  \sum_{\mathclap{\rho \in \Shuff(q,p-1)^{-1}}}
  (-1)^{|\rho|} (-1)^{p+1} (\beta \cdot \iota_v \alpha) \circ \rho_X
  \\
  &\quad{}
  +
  (-1)^{p+q+1}
  \sum_{\mathclap{\rho \in \Shuff(p,q-1)^{-1}}}
  (-1)^{|\rho|} (-1)^{q+1} (\alpha \cdot \iota_v \beta) \circ \rho_X
  \\
  &=
  (-1)^{pq} (-1)^q \beta \wedge \iota_v \alpha
  +
  (-1)^p \alpha \wedge \iota_v \beta
  \\
  &=
  (-1)^{pq} (-1)^q (-1)^{(p-1)q} \iota_v \alpha \wedge \beta
  +
  (-1)^p \alpha \wedge \iota_v \beta
  \\
  &=
  \iota_v \alpha \wedge \beta
  + (-1)^p \alpha \wedge \iota_v \beta
  \,,
\end{split}
\end{equation*}
which proves that $\iota_v$ is a graded derivation.

Equation~\eqref{eq:InnDerProp03} follows from the $R$-linearity of $\omega$ at the $n$-th position: by the recursion relation~\eqref{eq:AddPosRecurs} we have $+^{(n)}_{T^{n-1}X} = T^{n-1}{+_X}$, so that
\begin{equation*}
\begin{split}
  \omega \circ T^{n-1}(v+w)
  &= \omega \circ {+^{(n)}_{T^{n-1}X}} \circ (T^{n-1}v, T^{n-1}w)
  \\
  &= \Radd \circ (\omega \circ T^{n-1}v, \omega \circ T^{n-1}w)
  \,,
\end{split}
\end{equation*}
which is~\eqref{eq:InnDerProp03} after multiplication by $(-1)^{n+1}$.

We proceed to prove Equation~\eqref{eq:InnDerProp02}. For a form $\omega$ of degree $n=0$, we have $\iota_{fv} \omega = 0$ for degree reasons. It follows that \eqref{eq:InnDerProp02} is trivially satisfied. If $\omega$ has degree $n=1$, we have $\iota_{fv} \omega = \omega \circ \kappa_X \circ (f,v) = \Rmult \circ (f, \omega \circ v) = f \wedge \iota_v\omega$ by the $R$-linearity~\eqref{eq:CubForm2} of $\omega$ at the first position. Neither case requires an assumption on $2$-torsion. The case $n > 1$ will be proved by induction over $n$, assuming that $R$ has no $2$-torsion.

Assume that~\eqref{eq:InnDerProp02} holds for forms of degree $n \geq 1$ on all objects, and let $\omega \in \Omega^{n+1}_\mathrm{cub}(X)$. Let $f' = f \circ \pi_X: TX \to R$. Since $\omega' \in \Omega^n_\mathrm{cub}(X')$, we obtain
\begin{equation*}
\begin{split}
  \iota_{fv}\omega
  &=
  (\iota_{fv} \omega)'
  = \iota_{(fv)'} \omega'
  = \iota_{f' v'} \omega'
  = f' \wedge' \iota_{v'}\omega'
  = f' \wedge' (\iota_v \omega)'
  \\
  &= f \wedge \iota_v \omega
  \,,
\end{split}
\end{equation*}
where we have used Equation~\eqref{eq:InnPrime}, Equation~\eqref{eq:fPrimevPrime} of Lemma~\ref{lem:fPrimevPrime}, the induction hypothesis, and Equation~\eqref{eq:InnPrime2} for the $n$-form $\iota_v\omega$. We conclude by induction over $n$ that Equation~\eqref{eq:InnDerProp02} holds for forms of all degrees.

Finally, we prove Equation~\eqref{eq:InnDerProp01}. Let $w\colon X \to TX$ be another vector field. For a form $\omega$ of degree $n<2$ we have $\iota_w \iota_v \omega = 0$ for degree reasons, so \eqref{eq:InnDerProp01} is trivially satisfied. Let $\omega$ be a form of degree $n=2$. We compute
\begin{equation}
\label{eq:SingToAlgForm03}
\begin{split}
  (\iota_v \iota_w + \iota_w \iota_v)\omega
  &=
  - (\omega \circ Tw \circ v
   + \omega \circ Tv \circ w)
  \\
  &=
  - (\omega \circ Tw \circ v - \omega \circ \tau_X \circ Tv \circ w)
  \\
  &=
  - \omega\circ -_{TX} \circ
  (Tw \circ v, \tau_X \circ Tv \circ w)
  \\
  &=
  - \omega \circ \delta(v,w)
  \\
  &=
  - \omega \circ (\lambda_2)_X \circ (w,[v,w])
  \\
  &=
  0
  \,,
\end{split}
\end{equation}
where we have used the definition~\eqref{eq:CubeInnDer} of the inner derivative, the antisymmetry of~$\omega$, the additivity of~$\omega$ at the first position~\eqref{eq:CubForm1} together with the definition~\eqref{eq:Difference} of the difference, Equation~\eqref{eq:DeltaOrig}, Equation~\eqref{eq:deltaBracketRel}, and finally Equation~\eqref{eq:FormLambdaZero02} of Lemma~\ref{lem:FormLambdaZero} using that $R$ has no 2-torsion. This shows that Equation~\eqref{eq:InnDerProp01} holds for 2-forms. Assume that~\eqref{eq:InnDerProp01} holds for forms of degree $n$ on all objects, and let $\omega \in \Omega^{n+1}_\mathrm{cub}(X)$. By applying Equation~\eqref{eq:InnPrime} first to $\omega$ and then to the $n$-form $\iota_v\omega$, we obtain
\begin{equation*}
  \iota_{w'}\iota_{v'}\omega'
  = \iota_{w'}(\iota_v\omega)'
  = (\iota_w \iota_v\omega)'
  \,.
\end{equation*}
Together with the analogous relation for $v$ and $w$ exchanged, this implies
\begin{equation*}
  \bigl( (\iota_w \iota_v + \iota_v \iota_w)  \omega \bigr)'
  =
  (\iota_{w'} \iota_{v'} + \iota_{v'} \iota_{w'}) \omega'
  = 0
  \,,
\end{equation*}
where the last equality is the induction hypothesis applied to the $n$-form $\omega'$ on $X'$ and the vector fields $v'$, $w'$. Since the map~\eqref{eq:FormsTautMap} is injective, it follows that $(\iota_w \iota_v + \iota_v \iota_w)\omega = 0$. We conclude by induction over $n$ that~\eqref{eq:InnDerProp01} holds in all degrees.
\end{proof}

\subsection{Properties of the Lie derivative}

We recall that in any Cartan calculus the Lie derivative with respect to a vector field $v \in \calX(X)$ is defined by the graded commutator
\begin{equation*}
\begin{split}
  \Lie_v \coloneqq [\iota_v, d] = \iota_v d + d \iota_v
  \,.
\end{split}
\end{equation*}
Since graded derivations are closed under graded commutators, $\Lie_v$ is a graded derivation of $\Omega_\mathrm{cub}(X)$.

\begin{Proposition}
\label{prop:LieInnDerBracket}
The bracket of the Lie derivative and the inner derivative is
\begin{equation}
\label{eq:LieInnDerBracket}
  [\Lie_v, \iota_w] = \iota_{[v,w]}
\end{equation}
for all vector fields $v, w \in \calX(X)$.
\end{Proposition}

\begin{proof}
From the definition of the inner derivative~\eqref{eq:CubeInnDer} and the differential~\eqref{eq:dCubeDef}, we obtain for a vector field $v$ and an $n$-form $\omega$ the formula
\begin{equation}
\label{eq:LieInnDer01}
\begin{split}
  \iota_v d\omega
  &=
  (-1)^n \sum_{\mathclap{\sigma \in \Shuff(1,n)^{-1} }}
  (-1)^{|\sigma|}\Val_R \circ T\omega \circ \sigma_X \circ T^n v
  \\
  &=
  \Val_R \circ T\omega \circ (\Swp_{n,1})_X \circ T^n v
  +
  (-1)^n \sum_{\mathclap{\sigma \in \Shuff(1,n-1)^{-1} }}
  (-1)^{|\sigma|}\Val_R \circ T\omega \circ (\sigma T)_X \circ T^n v
  \\
  &=
  \Val_R \circ T\omega \circ (\Swp_{n,1})_X \circ T^n v
  +
  (-1)^n \sum_{\mathclap{\sigma \in \Shuff(1,n-1)^{-1} }}
  (-1)^{|\sigma|}\Val_R \circ T\omega \circ T^n v \circ \sigma_X
  \,,
\end{split}
\end{equation}
where we have decomposed the sum over $\Shuff(1,n)^{-1}$ as in~\eqref{eq:InDer10} and, in the last step, used the naturality of $\sigma_X$. The inner derivative followed by the differential is given by
\begin{equation}
\label{eq:LieInnDer02}
  d \iota_v \omega
  =
  (-1)^{n+1} \sum_{\mathclap{\sigma \in \Shuff(1,n-1)^{-1} }}
  (-1)^{|\sigma|}\Val_R \circ T\omega \circ T^n v \circ \sigma_X
  \,.
\end{equation}
Inserting~\eqref{eq:LieInnDer01} and \eqref{eq:LieInnDer02} into $\Lie_v \omega = \iota_v d\omega + d \iota_v \omega$, we obtain the simple formula
\begin{equation}
\label{eq:LieDerFormula}
  \Lie_v \omega
  =
    \Val_R \circ T\omega \circ
    (\Swp_{n,1})_X \circ T^n v
  \,.
\end{equation}
In particular, for a function $f: X \to R$, viewed as a form of degree $0$, we retrieve the action of vector fields,
\begin{equation*}
  \Lie_v f = \Val_R \circ Tf \circ v = v \cdot f
  \,.
\end{equation*}

For a form of degree $0$, the left side of~\eqref{eq:LieInnDerBracket} vanishes for degree reasons. The first non-trivial case is that of a 1-form $\omega \in \Omega^1_\mathrm{cub}(X)$. Equation~\eqref{eq:LieDerFormula} yields
\begin{equation*}
  \Lie_v \omega
  =  \Val_R \circ T\omega \circ \tau_X \circ Tv
  \,.
\end{equation*}
With this, we get
\begin{equation}
\label{eq:LieInnDerBrack03}
\begin{split}
  [\Lie_v, \iota_w]\omega
  &=
  (\Lie_v \iota_w - \iota_w \Lie_v)\omega
  \\
  &=
  \Lie_v (\omega \circ w)
  - \iota_w (\Val_R \circ T\omega \circ \tau_X \circ Tv)
  \\
  &=
    \Val_R \circ T\omega \circ Tw \circ v
  - \Val_R \circ T\omega \circ \tau_X \circ Tv \circ w
  \\
  &=
    \Val_R \circ T\omega \circ \bigl(
    -_{TX} \circ (Tw \circ v, \tau_X \circ Tv \circ w)
    \bigr)
  \\
  &=
    \Val_R \circ T\omega \circ (\lambda_2)_X \circ
    (w, [v,w])
  \,.
\end{split}
\end{equation}
To simplify this further, we compute
\begin{equation}
\label{eq:lambda2inForm}
\begin{split}
  \Val_R \circ T\omega \circ (\lambda_2)_X
  &=
  \Val_R \circ T\omega \circ
  \tau_X \circ +_{TX} \circ (T0_X \times \lambda_X)
  \\
  &=
  \Val_R \circ T\omega \circ
  {T+_X} \circ (\tau_X \times \tau_X) \circ
  (T0_X \times \lambda_X)
  \\
  &=
  \Val_R \circ T(\omega \circ +_X) \circ
  (0_{TX} \times \lambda_X)
  \\
  &=
  \Val_R \circ T\bigl(\Radd \circ (\omega \times \omega) \bigr)
  \circ (0_{TX} \times \lambda_X)
  \\
  &=
  \Val_R \circ T\Radd \circ (T\omega \times T\omega) \circ
  (0_{TX} \times \lambda_X)
  \\
  &=
  \Radd \circ (\Val_R \times \Val_R) \circ
  (T\omega \times T\omega) \circ
  (0_{TX} \times \lambda_X)
  \\
  &=
  \Radd \circ
  (\Val_R \circ 0_R \circ \omega,
   \Val_R \circ T\omega \circ \lambda_X)
  \\
  &=
  \Val_R \circ T\omega \circ \lambda_X \circ \pr_2
  \,,
\end{split}
\end{equation}
where we have used the defining Equation~\eqref{eq:VertLiftExt} of~$\lambda_2$, Diagram~\eqref{eq:TanFun2}, the left diagram of~\eqref{eq:TanFun5}, the functoriality of~$T$, that~$\omega$ is additive~\eqref{eq:CubForm1}, the left diagram of~\eqref{eq:DiffObj2} for~$A=R$, and finally the properties of the additive structure of~$R$. By inserting~\eqref{eq:lambda2inForm} into~\eqref{eq:LieInnDerBrack03}, we get
\begin{equation*}
\begin{split}
  [\Lie_v, \iota_w]\omega
  &=
    \Val_R \circ T\omega \circ \lambda_X \circ [v,w]
  \\
  &=
    \omega \circ [v,w]
  \\
  &=
  \iota_{[v,w]} \omega
  \,,
\end{split}
\end{equation*}
where we have used that $\omega$ is differential~\eqref{eq:CubDiffDef}.

For forms of degree $>1$ we will prove~\eqref{eq:LieInnDerBracket} by induction. Assume that~\eqref{eq:LieInnDerBracket} holds for all $X$ and forms up to degree $n$, $n \geq 1$. Let $\omega$ be an $(n+1)$-form on $X$, which we can view as an $n$-form on $X' \coloneqq TX$. Let $v\colon X \to TX$ be a vector field, so $v' \coloneqq \tau_X \circ Tv$ is a vector field on $X'$. Using formula~\eqref{eq:LieDerFormula} for the Lie derivative, we obtain
\begin{equation*}
\begin{split}
  \Lie_{v'} \omega'
  &=
  \Val_R \circ T\omega' \circ
  (\Swp_{n,1})_{X'} \circ T^n v'
  \\
  &=
  \Val_R \circ T\omega \circ
  (\Swp_{n,1})_{TX} \circ T^n(\tau_X \circ Tv)
  \\
  &=
  \Val_R \circ T\omega \circ
  (\Swp_{n,1}T \circ T^n \tau)_X \circ T^{n+1}v
  \\
  &=
  \Val_R \circ T\omega \circ
  (\Swp_{n+1,1})_X \circ T^{n+1}v
  \\
  &= \Lie_v \omega
  \,,
\end{split}
\end{equation*}
where we have used the definition of~$v'$, the functoriality of~$T$, and the relation $\Swp_{n,1}T \circ T^n\tau = \Swp_{n+1,1}$ for the block shuffle~\eqref{eq:ijShuffle}. With this relation we compute
\begin{equation}
\label{eq:LieDerInn05}
\begin{split}
  [\Lie_v, \iota_w]\omega
  &=
  \Lie_v \iota_w \omega - \iota_w \Lie_v \omega
  \\
  &=
  \Lie_{v'} (\iota_w \omega)'
  - \iota_{w'} (\Lie_v \omega)'
  \\
  &=
  \Lie_{v'} \iota_{w'} \omega'
  - \iota_{w'} \Lie_{v'} \omega'
  \\
  &=
  [\Lie_{v'}, \iota_{w'}] \omega'
  \\
  &= \iota_{[v',w']} \omega'
  \\
  &= \iota_{[v,w]'} \omega'
  \\
  &= \iota_{[v,w]} \omega
  \,,
\end{split}
\end{equation}
where we have also used Equation~\eqref{eq:InnPrime} for the inner derivatives and Lemma~\ref{lem:ProlongBracket}. The proof follows by induction over $n$.
\end{proof}

In Proposition~\ref{prop:InnDerProp} we have proved that the inner derivative satisfies the relation~\eqref{eq:InnDerOlinear} of a Cartan calculus, as well as the first of the bracket relations~\eqref{eq:CartanCalcBrackets}. In Proposition~\ref{prop:LieInnDerBracket}, we have proved the second bracket relation. By Remark~\ref{rmk:CartanRedund}, the other two relations are implied. We conclude that the conditions of a Cartan calculus in the sense of Definition~\ref{def:CartanCalc} are all satisfied, which finishes the proof of Statement~(ii) of Theorem~\ref{thm:CartanCube}.

\section{Relation to K\"ahler forms, Lie-Rinehart forms, and de Rham forms}

\label{sec:Comparison}

\subsection{The initial Cartan calculus of K\"ahler forms}

On a Rosick\'y tangent category, we have a natural notion of vector fields with Lie bracket. The new feature on a cartesian tangent category with scalar $R$-multiplication is the action~\eqref{eq:VecFieldFunc} of the Lie algebra of vector fields $\calX(X) = \Gamma(X,TX)$ on functions $\calO_X = \Hom(X,R)$ by derivations (Proposition~\ref{prop:VecActFunc}).

Let us denote by $\Omega^1_\text{K\"ah}(X)$ the $\calO_X$-module of \textbf{K\"ahler forms}, that is, the $\calO_X$-module generated by the symbols $d_\mathrm{K}f$ for all $f \in \calO_X$ modulo the relations $d_\mathrm{K}(rf) = r\, d_\mathrm{K}f$, $d_\mathrm{K}(f + g) = d_\mathrm{K}f + d_\mathrm{K}g$, and $d_\mathrm{K}(fg) = (d_\mathrm{K}f) g + f(d_\mathrm{K}g)$ for all $r \in \Hom(*,R)$ and $f,g \in \calO_X$. The module of K\"ahler forms corepresents $R$-linear derivations of $\calO_X$ with values in an $\calO_X$-module $M$. In particular, we have
\begin{equation*}
  \Der_R(\calO_X) \cong \Hom_{\calO_X}(\Omega^1_\text{K\"ah}(X), \calO_X)
  \,.
\end{equation*}
Under this isomorphism the derivation $f \mapsto v\cdot f$ for $v \in \calX(X)$ is sent to the map $d_\mathrm{K}f \mapsto v\cdot f$ that we denote by $\iota^\mathrm{K}_v$. That is, we define
\begin{equation}
\label{eq:KaehlerInn1}
  \iota^\mathrm{K}_v (d_\mathrm{K}f)
  \coloneqq v \cdot f
\end{equation}
for all vector fields $v$.

Let $\Omega_\text{K\"ah}(X)$ denote the free commutative graded $\calO_X$-algebra generated by the K\"ahler differentials in degree 1. Then we have a unique graded differential on $\Omega_\text{K\"ah}(X)$, defined by
\begin{equation}
\label{eq:KaehlerDiff}
  d_\mathrm{K}(f) = d_\mathrm{K}f
  \,,\qquad
  d_\mathrm{K}(d_\mathrm{K}f) = 0
\end{equation}
on the generators of $\Omega_\text{K\"ah}(X)$ as a ring. In the same vein, $\iota^\mathrm{K}$ extends to a unique graded derivation of $\Omega_\text{K\"ah}(X)$, defined by~\eqref{eq:KaehlerInn1} and
\begin{equation}
\label{eq:KaehlerInn2}
  \iota^\mathrm{K}_v(f) = 0
  \,,
\end{equation}
which holds for degree reasons.

\begin{Proposition}
\label{prop:KaehlerCartan}
Let $X$ be an object in a cartesian tangent category with scalar $R$-multiplication. Then the Lie algebra of vector fields $\calX(X)$ and the commutative graded $\calO_X$-algebra of K\"ahler forms $\Omega_\text{\rm K\"ah}(X)$ together with the differential $d_\mathrm{K}$ and the inner derivatives $\iota^\mathrm{K}_v$ constitute a Cartan calculus in the sense of Definition~\ref{def:CartanCalc}.
\end{Proposition}
\begin{proof}
The proof follows from standard arguments, known from the Cartan calculus of de Rham forms on a smooth manifold.
\end{proof}

The Cartan calculus of K\"ahler forms is initial in the following sense.

\begin{Proposition}
\label{prop:KaehlerInitial}
Let $X$ be an object in a cartesian tangent category with scalar $R$-multiplication. Let $(\calX(X), \Omega(X), d, \iota)$ be a Cartan calculus in the sense of Definition~\ref{def:CartanCalc}. Then there is a unique morphism of graded $\calO_X$-algebras $\phi: \Omega_{\text{\rm K\"ah}}(X) \to \Omega(X)$ such that
\begin{equation*}
  \phi\bigl( d_\mathrm{K} \omega \bigr)
  = d \bigl( \phi(\omega) \bigr)
  \,,\qquad
  \phi\bigl( \iota^\mathrm{K}_v \omega \bigr)
  = \iota_v \bigl( \phi(\omega) \bigr)
\end{equation*}
for all $\omega \in \Omega_\text{\rm K\"ah}(X)$ and $v \in \calX(X)$.
\end{Proposition}

\begin{proof}
Since $\phi$ is required to be a morphism of $\calO_X$-algebras, we have
\begin{equation}
\label{eq:Kaehcompmap01}
  \phi(f) = f
\end{equation}
for all $f \in \calO_X = \Omega^0_\text{K\"ah}(X)$. Because $\phi$ is required to intertwine the differentials, we have
\begin{equation}
\label{eq:Kaehcompmap02}
  \phi(d_\mathrm{K}f)
  = d \bigl( \phi(f) \bigr)
  = df
\end{equation}
for all $f$. Since $\Omega_\text{K\"ah}(X)$ is generated as a ring by $f$ and $d_\mathrm{K}f$ for all $f \in \calO_X$, Equations~\eqref{eq:Kaehcompmap01} and \eqref{eq:Kaehcompmap02} determine a morphism of algebras uniquely. We still have to check that $\phi$ is well defined. We have
\begin{equation*}
\begin{split}
  \phi\bigl( d_\mathrm{K}(fg) - (d_\mathrm{K}f)g - f(d_\mathrm{K}g) \bigr)
  &=
  \phi\bigl( d_\mathrm{K}(fg) \bigr)
  - \bigl(\phi(d_\mathrm{K}f)\bigr) g
  - f\bigl(\phi(d_\mathrm{K}g) \bigr)
  \\
  &=
  d(fg) - (df) g - f(dg)
  \\
  &= 0
  \,,
\end{split}
\end{equation*}
where we have used the $\calO_X$-linearity of $\phi$, then \eqref{eq:Kaehcompmap02}, and finally the fact that $d$ is a differential on $\Omega(X)$. Similar computations show that $\phi$ preserves all defining relations of $\Omega_\text{K\"ah}(X)$. We conclude that $\phi:\Omega_\text{K\"ah}(X) \to \Omega(X)$ is a uniquely determined morphism of $\calO_X$-algebras.

Equation~\eqref{eq:Kaehcompmap02} states that $\phi$ intertwines $d_\mathrm{K}$ and $d$ on functions $f$. Since $\phi(d_\mathrm{K}(d_\mathrm{K}f)) = \phi(0) = 0 = d^2 f = d^2(\phi(f))$, $\phi$ intertwines the differentials also on $d_\mathrm{K}f$. Since $\phi$ intertwines the differentials on the generators of the ring $\Omega_\text{K\"ah}(X)$ and $\phi$ is a morphism of rings, it follows that $\phi$ intertwines the differentials on all of $\Omega_\text{K\"ah}(X)$. By an analogous argument, we conclude that $\phi$ intertwines the inner derivatives.
\end{proof}


\subsection{The terminal Cartan calculus of Lie-Rinehart forms}
\label{sec:LRforms}

An algebraic structure that gives rise to a Cartan calculus is a Lie-Rinehart algebra, a concept introduced in \cite{Rinehart:1963}.

\begin{Definition}
\label{def:LieRinehartAlg}
A \textbf{Lie-Rinehart algebra} over a commutative ring $\calO$ is an $\calO$-module $\calX$ that is equipped with a Lie bracket and an $\calO$-linear Lie algebra action on $\calO$ by derivations, satisfying the Leibniz rule
\begin{equation}
\label{eq:LieRineLeibniz}
  [v, fw] = (v \cdot f)\, w + f[v,w]
\end{equation}
for all $v, w \in \calX$ and $f \in \calO$.
\end{Definition}

\begin{Remark}
Definition~\ref{def:LieRinehartAlg} generalizes to a commutative $\calR$-algebra $\calO$, where $\calR$ is some commutative ring, typically a field \cite{Huebschmann:1998}. In this generalization, the Lie bracket is required to be $\calR$-bilinear.
\end{Remark}

\begin{Proposition}[Proposition~6.4 in \cite{AintablianBlohmann:2026}]
\label{prop:LieAlgdRinehart}
Let $A \to X$ be an abstract Lie algebroid in a cartesian tangent category with scalar $R$-multiplication. Let $\calO \coloneqq \Hom(X,R)$ and $\calA \coloneqq \Gamma(X,A)$. Then the $\calO$-module $\calA$ with the action $a \cdot f \coloneqq (\rho \circ a) \cdot f$ on $\calO$ is a Lie-Rinehart algebra.
\end{Proposition}

To every Lie-Rinehart algebra $\calX$ over $\calO$, we can associate a Cartan calculus as follows. A \textbf{Lie-Rinehart $n$-form} is an $\calO$-linear map
\begin{equation*}
  \omega: \wedge^n_\calO \calX
  \longrightarrow \calO
  \,.
\end{equation*}
The graded $\calO$-module of Lie-Rinehart forms
\begin{equation*}
  \Omega_\mathrm{LR}^\bullet(\calX)
  \coloneqq \Hom_\calO(\wedge^\bullet_\calO \calX, \calO)
\end{equation*}
has the structure of a commutative graded $\calO$-algebra defined by the antisymmetrization of the pointwise multiplication of forms,
\begin{equation*}
  (\alpha \wedge \beta)(v_1, \ldots, v_{p+q})
  \coloneqq \sum_{\mathclap{\sigma \in \Shuff(p,q)^{-1}}}
    (-1)^{|\sigma|}
    \alpha(v_{\sigma(1)}, \ldots, v_{\sigma(p)}) \,
    \beta(v_{\sigma(p+1)}, \ldots, v_{\sigma(p+q)})
    \,.
\end{equation*}
The \textbf{differential} of $\omega$ is the $(n+1)$-form given by the usual Chevalley-Eilenberg formula
\begin{equation}
\label{eq:differential}
\begin{split}
  (d_\mathrm{LR}\omega)(v_0, \ldots, v_n)
  \coloneqq &\sum_{\mathclap{0 \leq i \leq n}} (-1)^i
    v_i \cdot \omega(v_0, \ldots, \hat{v}_i, \ldots, v_n)
  \\
    + &\sum_{\mathclap{0 \leq i < j \leq n}}
    (-1)^{i+j} \omega([v_i, v_j], v_0, \ldots, \hat{v}_i, \ldots, \hat{v}_j, \ldots, v_n)
  \,.
\end{split}
\end{equation}
The \textbf{inner derivative} is the operator that inserts an element $v \in \calX$ into a form:
\begin{equation}
\label{eq:InnDeriv}
\begin{aligned}
   \iota^\mathrm{LR}_v: \Omega_\mathrm{LR}^n(\calX)
   &\longrightarrow \Omega_\mathrm{LR}^{n-1}(\calX)
   \\
   (\iota^\mathrm{LR}_v \omega)(v_1, \ldots, v_{n-1})
   &\coloneqq
   \omega(v, v_1, \ldots, v_{n-1})
   \,.
\end{aligned}
\end{equation}

\begin{Proposition}[Theorem~4.2 and \S6 in \cite{Rinehart:1963}]
\label{prop:ACartanCalc}
Let $\calX$ be a Lie-Rinehart algebra over $\calO$. The commutative differential graded algebra $(\Omega_\mathrm{LR}(\calX), d_\mathrm{LR})$, the Lie algebra $\calX$, and the inner derivative~\eqref{eq:InnDeriv} constitute a Cartan calculus.
\end{Proposition}

\begin{Example}[Example~6.8 in \cite{AintablianBlohmann:2026}]
Let $M$ be a manifold. The Lie-Rinehart algebra of the tangent algebroid $TM \to M$ is given by the $\bbR$-algebra $\calO = C^\infty(M)$ of smooth functions, the Lie algebra $\calX \coloneqq \calX(M)$ of vector fields with the Lie bracket, and the action of vector fields as derivations on $C^\infty(M)$. From Proposition~\ref{prop:ACartanCalc} we retrieve the usual Cartan calculus given by the graded algebra of differential forms with the de Rham differential, inner derivative, and Lie derivative.
\end{Example}

\begin{Proposition}[Proposition~6.11 in \cite{AintablianBlohmann:2026}]
\label{prop:CartanOnA}
Let $A \to X$ be an abstract Lie algebroid in a cartesian tangent category with scalar $R$-multiplication; let $\calO_X \coloneqq \Hom(X,R)$. Then the Lie algebra of sections $\calA \coloneqq \Gamma(X,A)$ and the commutative graded $\calO_X$-algebra of forms,
\begin{equation*}
  \Omega_\mathrm{LR}^\bullet(\calA)
  \coloneqq
  \Hom_{\calO_X}\bigl(\wedge^\bullet_{\calO_X} \calA,
    \calO_X \bigr)
  \,,
\end{equation*}
together with the differential~\eqref{eq:differential} and the inner derivative~\eqref{eq:InnDeriv} constitute a Cartan calculus in the sense of Definition~\ref{def:CartanCalc}.
\end{Proposition}
\begin{proof}
The statement follows from Proposition~\ref{prop:LieAlgdRinehart} and Proposition~\ref{prop:ACartanCalc}.
\end{proof}

\begin{Theorem}[Theorem~6.13 in \cite{AintablianBlohmann:2026}]
\label{thm:CartanOnX}
Let $X$ be an object in a cartesian tangent category with scalar $R$-multiplication. Then the Lie algebra of vector fields $\calX(X) \coloneqq \Gamma(X,TX)$ and the commutative graded $\calO_X$-algebra of forms,
\begin{equation*}
  \Omega_\mathrm{LR}^\bullet(X)
  \coloneqq
  \Hom_{\calO_X}\bigl(\wedge^\bullet_{\calO_X} \calX(X),
    \calO_X \bigr)
  \,,
\end{equation*}
together with the differential~\eqref{eq:differential} and the inner derivative~\eqref{eq:InnDeriv} constitute a Cartan calculus in the sense of Definition~\ref{def:CartanCalc}.
\end{Theorem}

The Cartan calculus of Lie-Rinehart forms is terminal in the following sense.

\begin{Proposition}
\label{prop:LRterminal}
Let $X$ be an object in a cartesian tangent category with scalar $R$-multiplication. Let $(\calX(X), \Omega(X), d, \iota)$ be a Cartan calculus in the sense of Definition~\ref{def:CartanCalc}. Then there is a unique morphism of graded $\calO_X$-algebras $\psi: \Omega(X) \to \Omega_\mathrm{LR}(X)$ such that
\begin{equation*}
  \psi\bigl( d\omega \bigr)
  = d_\mathrm{LR} \bigl( \psi(\omega) \bigr)
  \,,\qquad
  \psi\bigl( \iota_v \omega \bigr)
  = \iota^\mathrm{LR}_v \bigl( \psi(\omega) \bigr)
\end{equation*}
for all $\omega \in \Omega(X)$ and $v \in \calX(X)$.
\end{Proposition}

\begin{proof}
Since $\psi$ is required to be a morphism of $\calO_X$-algebras, we have
\begin{equation}
\label{eq:LRcompmap01}
  \psi(f) = f
\end{equation}
for all $f \in \calO_X = \Omega^0_\mathrm{LR}(X)$. Let $\omega \in \Omega^n(X)$, $n \geq 1$. Because $\psi$ is required to intertwine the inner derivatives, we have
\begin{equation}
\label{eq:LRcompmap02}
\begin{split}
  \bigl(\psi(\omega)\bigr)(v_1, \ldots, v_n)
  &=
  \iota^\mathrm{LR}_{v_n} \cdots \iota^\mathrm{LR}_{v_2}
  \iota^\mathrm{LR}_{v_1} \psi(\omega)
  \\
  &=
  \psi\bigl( \iota_{v_n} \cdots \iota_{v_1} \omega \bigr)
  \\
  &=
  \iota_{v_n} \cdots \iota_{v_1} \omega
  \,,
\end{split}
\end{equation}
for all $v_1, \ldots, v_n \in \calX(X)$. This determines $\psi(\omega) \in \Omega^n_\mathrm{LR}(X)$ uniquely.

To show that $\psi$ is a morphism of algebras, we proceed by induction over the total degree $p+q$ of $\alpha \in \Omega^p(X)$ and $\beta \in \Omega^q(X)$. For $p = q = 0$, the claim holds because of~\eqref{eq:LRcompmap01}. Assume now that the claim holds for all pairs of forms of total degree less than $p+q \geq 1$. The inner derivatives of $\Omega(X)$ and $\Omega_\mathrm{LR}(X)$ are derivations of degree $-1$, by the properties of a Cartan calculus. Using \eqref{eq:LRcompmap02}, we obtain
\begin{equation}
\label{eq:LRcompmap03}
\begin{split}
  \iota^\mathrm{LR}_v \bigl( \psi(\alpha \wedge \beta) \bigr)
  &= \psi\bigl( \iota_v (\alpha \wedge \beta) \bigr)
  \\
  &= \psi\bigl( (\iota_v \alpha) \wedge \beta \bigr)
  + (-1)^p \psi\bigl( \alpha \wedge (\iota_v \beta) \bigr)
  \\
  &= \psi(\iota_v \alpha) \wedge \psi(\beta)
  + (-1)^p \psi(\alpha) \wedge \psi(\iota_v \beta)
  \\
  &= \iota^\mathrm{LR}_v \bigl( \psi(\alpha)\bigr)
     \wedge \psi(\beta)
  + (-1)^p \psi(\alpha) \wedge \iota^\mathrm{LR}_v \bigl( \psi(\beta) \bigr)
  \\
  &= \iota^\mathrm{LR}_v \bigl( \psi(\alpha) \wedge \psi(\beta) \bigr)
\end{split}
\end{equation}
for all $v \in \calX(X)$, where the third equality holds by the induction hypothesis, applied to the pairs $(\iota_v\alpha, \beta)$ and $(\alpha, \iota_v\beta)$ of total degree $p+q-1$. Two Lie-Rinehart forms $\omega$ and $\omega'$ are equal if and only if $\iota^\mathrm{LR}_v \omega = \iota^\mathrm{LR}_v \omega'$ for all $v \in \calX(X)$. We conclude that $\psi(\alpha \wedge \beta) = \psi(\alpha) \wedge \psi(\beta)$.

Using only the bracket relations of the Cartan calculus on $\Omega(X)$, we can show that
\begin{equation}
\label{eq:LRcompmap04}
\begin{split}
  \iota_{v_n}\cdots \iota_{v_0} d\omega
  = &\sum_{\mathclap{0 \leq i \leq n}} (-1)^i \Lie_{v_i}
    \iota_{v_n} \cdots \widehat{\iota_{v_i}} \cdots \iota_{v_0} \omega
  \\
    + &\sum_{\mathclap{0 \leq i < j \leq n}} (-1)^{i+j}
    \iota_{v_n} \cdots \widehat{\iota_{v_j}} \cdots  \widehat{\iota_{v_i}}
    \cdots \iota_{v_0} \iota_{[v_i, v_j]} \omega
\end{split}
\end{equation}
for all $\omega \in \Omega^n(X)$. Using~\eqref{eq:LRcompmap01} and~\eqref{eq:LRcompmap02}, we can rewrite the terms of~\eqref{eq:LRcompmap04} as
\begin{align*}
  \iota_{v_n}\cdots \iota_{v_0} d\omega
  &= \bigl( \psi(d\omega) \bigr)(v_0, \ldots, v_n)
  \\
  \Lie_{v_i}
    \iota_{v_n} \cdots \widehat{\iota_{v_i}} \cdots \iota_{v_0} \omega
  &=
  v_i \cdot \bigl( \psi(\omega) \bigr)
  (v_0, \ldots, \widehat{v_i}, \ldots, v_n)
  \\
  \iota_{v_n} \cdots \widehat{\iota_{v_j}} \cdots  \widehat{\iota_{v_i}}
    \cdots \iota_{v_0} \iota_{[v_i, v_j]} \omega
  &=
  \bigl( \psi(\omega) \bigr)( [v_i, v_j], v_0, \ldots,
  \widehat{v_i}, \ldots, \widehat{v_j}, \ldots, v_n)
  \,.
\end{align*}
By inserting these expressions in~\eqref{eq:LRcompmap04} and comparing the right side with the Chevalley-Eilenberg formula~\eqref{eq:differential}, we obtain $\psi(d\omega) = d_\mathrm{LR}( \psi(\omega) )$.
\end{proof}

The upshot is the following result:

\begin{Theorem}
\label{thm:CompMaps}
Let $X$ be an object in a cartesian tangent category with scalar $R$-multiplication, where $R$ has no $2$-torsion. Then there is a unique sequence
\begin{equation}
\label{eq:CompMaps}
  \Omega_\text{\rm K\"ah}(X)
  \longrightarrow
  \Omega_\mathrm{cub}(X)
  \longrightarrow
  \Omega_\mathrm{LR}(X)
\end{equation}
of morphisms of Cartan calculi between the Cartan calculi of K\"ahler forms, differential cubical forms, and Lie-Rinehart forms.
\end{Theorem}

\begin{proof}
The first morphism is given by Proposition~\ref{prop:KaehlerInitial}, the second by Proposition~\ref{prop:LRterminal}, applied in both cases to the Cartan calculus of differential cubical forms of Theorem~\ref{thm:CartanCube}.
\end{proof}

It is a natural question for future research in which cases the comparison maps~\eqref{eq:CompMaps} are quasi-isomorphisms.

\subsection{The issue of de Rham forms}
\label{sec:DeRhamForms}

Assume that $X$ is an object in a cartesian tangent category $\calC$ with scalar $R$-multiplication. The symmetric group $S_n$ acts naturally by permutation on the $n$-fold product of objects in any category. In particular, it acts on the products in the slice category $\calC_{/X}$, which are the fiber products over $X$. We denote the action of $\sigma \in S_n$ by $\sigma_{/X}$.

\begin{Definition}
A morphism $\omega\colon T_n X \to R$ is \textbf{antisymmetric} if
\begin{equation*}
  \omega \circ \sigma_{/X} = (-1)^{|\sigma|} \omega
\end{equation*}
for all $\sigma \in S_n$.
\end{Definition}

The natural morphisms of addition and scalar multiplication \textbf{at the $i$-th position} will be denoted by
\begin{align*}
  +^{(i)}_{/X}
  &\colon
  T_{n+1} X \longrightarrow T_n X
  \\
  \kappa^{(i)}_{/X}
  &\colon
  R \times T_n X \longrightarrow T_n X
  \,,
\end{align*}
and are defined by
\begin{align}
  \label{eq:TfibaddAti}
  +^{(i)}_{/X}
  &\coloneqq
  \id_{T_{i-1}X} \times (+_X) \times \id_{T_{n-i}X}
  \\
  \label{eq:RfibmultAti}
  \kappa^{(i)}_{/X}
  &\coloneqq
  \id_{T_{i-1}X} \times \kappa_X \times \id_{T_{n-i}X}
  \,.
\end{align}

\begin{Definition}
\label{def:DeRhamFormAti}
A morphism $\omega\colon T_n X \to R$ will be called \textbf{additive at the $i$-th position} if the diagram
\begin{equation}
\label{diag:DeRhamForm1}
\begin{tikzcd}[column sep=2em]
T_{n+1}X
\ar[d, "+^{(i)}_{/X}"']
\ar[r, "{ (\id_{T_{i-1}X} \times \pr_1 \times \id_{T_{n-i}X},\, \id_{T_{i-1}X} \times \pr_2 \times \id_{T_{n-i}X}) }"]
&[17em]
T_n X \times T_n X
\ar[r, "\omega \times \omega"]
&
R \times R
\ar[d, "\Radd"]
\\
T_n X
\ar[rr, "\omega"']
&&
R
\end{tikzcd}
\end{equation}
commutes and \textbf{$R$-equivariant at the $i$-th position} if
\begin{equation}
\label{diag:DeRhamForm2}
\begin{tikzcd}
R \times T_n X
\ar[d, "\kappa^{(i)}_{/X}"']
\ar[r, "{\id_R \times \omega}"]
&
R \times R
\ar[d, "\Rmult"]
\\
T_n X
\ar[r, "\omega"']
&
R
\end{tikzcd}
\end{equation}
commutes.
\end{Definition}

\begin{Definition}
\label{def:DeRhamForms}
An antisymmetric morphism $\omega\colon T_n X \to R$ that is additive and $R$-equivariant at the $i$-th position for all $1 \leq i \leq n$ will be called a \textbf{de Rham $n$-form} on $X$. The space of all de Rham $n$-forms on $X$ will be denoted by $\Omega^n_\mathrm{dR}(X)$.
\end{Definition}

The wedge product of de Rham forms $\alpha \in \Omega^p_\mathrm{dR}(X)$ and $\beta \in \Omega^q_\mathrm{dR}(X)$ is the antisymmetrization of the morphism
\begin{equation*}
  T_{p+q} X \xrightarrow{~\cong~}
  T_p X \times_X T_q X
  \xrightarrow{~\alpha \times \beta~}
  R \times R \xrightarrow{~\Rmult~}
  R
  \,,
\end{equation*}
given by the signed summation over the $(p,q)$-unshuffles. This equips $\Omega_\mathrm{dR}(X)$ with the structure of a commutative graded $\calO_X$-algebra. The inner derivative of a de Rham $n$-form $\omega$ with respect to a vector field $v$ is given by
\begin{equation*}
  \iota_v^\mathrm{dR} \omega:
  T_{n-1} X \xrightarrow{~\cong~}
  X \times_X T_{n-1} X \xrightarrow{~v \times \id_{T_{n-1}X}~}
  T_n X \xrightarrow{~\omega~}
  R
  \,.
\end{equation*}
The issue is that there is no clear intrinsic way to define a differential on de Rham forms.

Two possible approaches to solve this issue come to mind. The first approach is the observation that we have a natural morphism of graded $\calO_X$-algebras
\begin{equation*}
  \rho:
  \Omega_\mathrm{dR}(X)
  \longrightarrow
  \Omega_\mathrm{LR}(X)
\end{equation*}
defined in degree $n$ by
\begin{equation*}
  \bigl(\rho(\omega)\bigr)(v_1, \ldots, v_n)
  \coloneqq
  \iota^\mathrm{dR}_{v_n} \cdots \iota^\mathrm{dR}_{v_1} \omega
  \,.
\end{equation*}
In order for de Rham forms to inherit a Cartan calculus from the Cartan calculus on Lie-Rinehart forms, $\rho$ must be injective. We do not see how to obtain this property without leaving the framework of tangent categories, e.g.~by adding sheaf-theoretic properties of forms and vector fields.

The second approach that comes to mind is to map de Rham forms to cubical forms, by pulling them back by the universal morphism
\begin{equation*}
  (\mu_n)_X \coloneqq (T\pi^{n-1}, \pi T \pi^{n-2}, \ldots, \pi^{n-1}T)_X:
  T^n X \longrightarrow T_n X
  \,.
\end{equation*}
It is straightforward to show that $(\mu_n)_X$ is $S_n$-equivariant and preserves linearity at every position. In addition, the cubical forms are required to satisfy the additional condition of differentiability~\eqref{eq:CubDiffDef}. This imposes an additional condition on the de Rham forms. Moreover, for de Rham forms to inherit a differential and, therefore, a Cartan calculus from cubical forms, the pullback $(\mu_2)_X^*$ has to be injective. It is not clear how to impose this property within the framework of tangent categories.

Tentatively, we could require the morphism
\begin{equation*}
\begin{tikzcd}[column sep=5em]
T_2 \times_{T}^{\pr_1, \pi T} T^2
\ar[r, shift left=4pt, "\pr_2"]
\ar[r, shift right=4pt, "(+T) \circ (\lambda_2 \times \id_{T^2})"']
&
T^2
\ar[r, "{(T\pi, \pi T)}"]
&
T_2
\end{tikzcd}
\end{equation*}
to be a pointwise coequalizer. This means that for every $X$ the sequence
\begin{equation*}
  T_2 X \xrightarrow{~(\lambda_2)_X~}
  T^2 X \xrightarrow{~(T\pi, \pi T)_X~} T_2 X
\end{equation*}
is an exact sequence of bundles of abelian groups over $TX$. We posit that this implies that the map $(\mu_2)_X^*$ from de Rham 2-forms to differential cubical 2-forms is an isomorphism. Generalizing this property to higher forms is a task for future research.

\section{Examples}

\label{sec:Examples}

In Section~7 of \cite{AintablianBlohmann:2026} a long list of examples for tangent categories with scalar multiplications and their Cartan calculi of Lie-Rinehart forms was given. Here, we will go through this list and explain in which cases the differential cubical forms differ from the Lie-Rinehart forms.

\subsection{Smooth manifolds}
\label{sec:SmoothMflds}

The guiding example of a cartesian tangent category with scalar multiplication is the category of smooth manifolds, the ring of scalars being given by $\bbR$. This is explained in detail in Section~7.1 of \cite{AintablianBlohmann:2026}. On a smooth manifold $X$, the sheaf of vector fields is locally free and finitely generated. As a consequence, $C^\infty(X)$-linear functions on vector fields are in natural bijection with fiberwise linear maps on the tangent bundle $TX$. It follows that Lie-Rinehart forms are in natural bijection with de Rham forms $\Omega_\mathrm{LR}(X) \cong \Omega_\mathrm{dR}(X)$. What about the relation to differential cubical forms?

Let $X$ be an $n$-dimensional smooth manifold and $U \subset \bbR^n$ a local chart. The tangent functor is given by $TU = U \times \bbR^n$, which iterates to $T^k U = U \times (\bbR^n)^{2^k-1}$. The $k$-fold fiber product of the tangent bundle is given by $T_k U \cong U \times (\bbR^n)^k$. By definition, a cubical $1$-form is a fiberwise $\bbR$-linear map $TU \to \bbR$, which is the same thing as a de Rham $1$-form. Since on a finite-dimensional $\bbR$-vector space, the differential of a map is defined to be the linear approximation, the differential of a linear map is the map itself. It follows that every cubical 1-form is differential.

For $2$-forms, the domains of cubical and de Rham forms are different. Let us choose coordinates
\begin{align*}
  (x^i, x_1^i, x_2^i, x_{12}^i)
  &\in
  U \times \bbR^n \times \bbR^n \times \bbR^n
  = T^2 U
  \\
  (x^i, x_1^i, x_2^i)
  &\in
  U \times \bbR^n \times \bbR^n
  = T_2 U
  \,.
\end{align*}
In the following, we will omit the index $i \in \{1, \ldots, n\}$ from the notation. A cubical $2$-form $\omega: T^2 U \to \bbR$ is antisymmetric,
\begin{align*}
  \omega(x, x_1, x_2, x_{12})
  &= - \omega(x, x_2, x_1, x_{12})
  \\
\intertext{and linear at the first position,}
  \omega(x, x_1 + y_1, x_2, x_{12} + y_{12})
  &=  \omega(x, x_1, x_2, x_{12})
    + \omega(x, y_1, x_2, y_{12})
  \\
  \omega(x, r x_1, x_2, r x_{12})
  &= r \omega(x, x_1, x_2, x_{12})
  \,,
  \\
\intertext{which implies that it is linear at the second position,}
  \omega(x, x_1, x_2 + y_2, x_{12} + y_{12})
  &=  \omega(x, x_1, x_2, x_{12})
    + \omega(x, x_1, y_2, y_{12})
  \\
  \omega(x, x_1, r x_2, r x_{12})
  &= r \omega(x, x_1, x_2, x_{12})
  \,.
\end{align*}
Since $\bbR$ has no $2$-torsion, Lemma~\ref{lem:FormLambdaZero} shows that $\omega \circ \lambda_U = 0$,
\begin{equation*}
  \omega(x,0,0,x_{12}) = 0
  \,.
\end{equation*}
It follows that
\begin{equation*}
\begin{split}
  \omega(x,x_1,x_2,x_{12})
  &= \omega(x,x_1,x_2,0) + \omega(x,x_1,0,x_{12})
  \\
  &= \omega(x,x_1,x_2,0) + \omega(x,x_1,0,0) + \omega(x,0,0,x_{12})
  \\
  &= \omega(x,x_1,x_2,0)
  \,.
\end{split}
\end{equation*}
We conclude that $\omega$ can be identified with an antisymmetric, fiberwise bilinear map $T_2 U \to \bbR$, that is, a de Rham $2$-form.

For the general case, we observe that the $2^k$ coordinates $(x,x_1, \ldots, x_{12\cdots k})$ of $T^k U$ are labelled uniquely by the subsets of $\{1, \ldots, k\}$. The action of the symmetric group permutes the indices, that is, $\sigma \in S_k$ maps $x_{i_1, \ldots, i_r}$ to $x_{\sigma(i_1), \ldots, \sigma(i_r)}$, while the form $\omega$ changes the sign by $(-1)^{|\sigma|}$. Repeating the argument given above for $k=2$ at every position, we conclude that $\omega$ does not depend on any coordinate that has more than one index, but only on $x_1, \ldots, x_k$. It follows that a cubical $k$-form can be identified with a de Rham $k$-form for all $k \geq 0$,
\begin{equation*}
  \Omega_\mathrm{cub}(X) \cong \Omega_\mathrm{dR}(X)
  \cong \Omega_\mathrm{LR}(X)
  \,.
\end{equation*}
We conclude that in the category of smooth manifolds the Cartan calculus of differential cubical forms is isomorphic to the Cartan calculus of Lie-Rinehart forms.

\subsection{Examples based on smooth manifolds}

Categories that are based on smooth finite-dimensional manifolds typically inherit the property that the cubical forms, the de Rham forms, and the Lie-Rinehart forms are in bijection so that the Cartan calculi are isomorphic. In particular, this is true for the following examples:
\begin{itemize}

\item $G$-manifolds for a fixed Lie group $G$ \cite[\S 7.2]{AintablianBlohmann:2026}

\item Lie groupoids \cite[\S 7.3]{AintablianBlohmann:2026}

\item Log manifolds, also called $b$-manifolds \cite[\S 7.4]{AintablianBlohmann:2026}

\item Pro-manifolds \cite[\S 7.5]{AintablianBlohmann:2026}

\end{itemize}
It is straightforward to come up with more examples by generalizing and combining these cases.

\subsection{Affine schemes}
\label{sec:AffSchemes}

In the category of affine schemes $\Aff_k \simeq \CAlg_k^\op$, differential cubical forms are generally different from Lie-Rinehart forms. The tangent structure of affine schemes was introduced and studied in \cite[\S 4.1]{CruttwellLemay:2023}. The tangent functor is given by
\begin{equation*}
  T\Spec(A)
  \coloneqq \Spec\bigl( \Sym_A(\Omega^1_{A/k}) \bigr)
  \,,
\end{equation*}
where $\Omega^1_{A/k}$ is the $A$-module of K\"ahler differentials of the commutative $k$-algebra $A$ and $\Sym_A: \Mod_A \to \CAlg_A$ is the left adjoint to the forgetful functor $\CAlg_A \to \Mod_A$. Let $\Spec(A) \to \Spec(B)$ be a morphism of affine schemes given by a morphism of algebras $\phi: B \to A$. Then the tangent morphism $T\Spec(A) \to T\Spec(B)$ is given by
\begin{align*}
  \Sym_B(\Omega^1_{B/k})
  &\longrightarrow
  \Sym_A(\Omega^1_{A/k})
  \\
  b &\longmapsto \phi(b)
  \\
  db &\longmapsto d\bigl( \phi(b) \bigr)
  \,.
\end{align*}
The fiberwise addition $T_2 \Spec(A) \to T\Spec(A)$ is given by the morphism of algebras
\begin{align*}
  +_A:
  \Sym_A(\Omega^1_{A/k})
  &\longrightarrow
  \Sym_A(\Omega^1_{A/k}) \otimes_A \Sym_A(\Omega^1_{A/k})
  \\
  da &\longmapsto da \otimes 1 + 1 \otimes da
  \,.
\end{align*}
The unitality, $1 \mapsto 1 \otimes_A 1$, implies that $a \mapsto a (1 \otimes_A 1) = a \otimes_A 1$.

In \cite[\S 7.7]{AintablianBlohmann:2026}, the scalar multiplication $\bbA_k \times T\Spec(A) \to T\Spec(A)$ by the affine line $\bbA_k = \Spec(k[x])$ was added to the tangent structure, which is given by the morphism of algebras
\begin{align*}
  \kappa_A: \Sym_A(\Omega^1_{A/k})
  &\longrightarrow k[x] \otimes \Sym_A(\Omega^1_{A/k})
  \\
  a
  &\longmapsto 1 \otimes a
  \\
  da
  &\longmapsto x \otimes da
  \,.
\end{align*}

We will now describe cubical $1$-forms. A morphism $\omega$ in
\begin{equation*}
  \Hom\bigl( T\Spec(A), \bbA_k \bigr)
  \cong \CAlg_k\bigl( k[x], \Sym_A(\Omega^1_{A/k})\bigr)
\end{equation*}
is uniquely determined by the value $\omega(x) \in \Sym_A(\Omega^1_{A/k})$ of the generator. The multiplication of $\bbA_k$ is given by the comultiplication $\Delta_{\Rmult}: k[x] \to k[x] \otimes k[x]$, $x \mapsto x \otimes x$. It follows that a morphism $T\Spec(A) \to \bbA_k$ given by $\omega(x) \in \Sym_A(\Omega^1_{A/k})$ is $\bbA_k$-equivariant if and only if
\begin{equation*}
  \kappa_A\bigl(\omega(x) \bigr)
  = x \otimes \omega(x)
  \,,
\end{equation*}
which is the case if and only if $\omega(x)$ is linear in the symbols $da$, that is, $\omega(x) \in \Omega^1_{A/k} \subset \Sym_A(\Omega^1_{A/k})$. Similarly, the addition of $\bbA_k$ is given by the coaddition $\Delta_{\Radd}: k[x] \to k[x] \otimes k[x]$, $x \mapsto x \otimes 1 + 1 \otimes x$. It follows that the morphism $\omega$ is additive if and only if
\begin{equation*}
  +_A\bigl(\omega(x)\bigr) = \omega(x) \otimes 1 + 1 \otimes \omega(x)
  \,.
\end{equation*}
Again, this is the case if and only if $\omega(x) \in \Omega^1_{A/k} \subset \Sym_A(\Omega^1_{A/k})$. We conclude that the $A$-module of cubical $1$-forms can be identified with $\Omega^1_{A/k}$.

\begin{Lemma}
\label{lem:Schemes1}
Every cubical $1$-form on $\Spec(A)$ is differential.
\end{Lemma}

\begin{proof}
Let $\Omega^{(2)}_{A/k}$ denote the $A$-module generated by symbols $a$, $d_1 a$, $d_2 a$, and $d_{12} a$ for all $a \in A$, subject to the relations
\begin{equation}
\label{eq:Schemes1}
  d_i (1) = 0 \,,\quad
  d_i(a + b) = d_i a + d_i b \,,\quad
  d_i(ab) =  d_i(a)\, b + a \, d_i(b) \,,\quad
  d_i^2(a) = 0 \,,
\end{equation}
for $i \in \{1,2\}$ and
\begin{equation*}
  d_1 (d_2 a)
  = d_{12} a
  = d_2 (d_1 a)
  \,.
\end{equation*}
Then
\begin{equation*}
  T^2 \Spec(A) =  \Spec\bigl( \Sym_A(\Omega^{(2)}_{A/k}) \bigr)
  \,.
\end{equation*}
The vertical lift $T\Spec(A) \to T^2 \Spec(A)$ is the homomorphism
\begin{equation*}
  \lambda_A: \Sym_A(\Omega^{(2)}_{A/k}) \to \Sym_A(\Omega^1_{A/k})
\end{equation*}
given by $a \mapsto a$, $d_1 a \mapsto 0$, $d_2 a \mapsto 0$, and $d_{12} a \mapsto da$. The tangent scheme $T\bbA_k$ of the affine line is given by
\begin{equation*}
\begin{split}
  T\bbA_k
  &= \Spec\bigl( \Sym_{k[x]}(\Omega^1_{k[x]/k}) \bigr)
  \cong \Spec\bigl( k[x, dx] \bigr)
  \cong \Spec\bigl( k[x] \otimes k[dx] \bigr)
  \\
  &\cong
  \bbA_k \times \bbA_k
  \,.
\end{split}
\end{equation*}
The morphism $T\bbA_k \to \bbA_k$ to the tangent fiber is given by
\begin{align*}
  \Val_{k[x]}: k[x]
  &\longrightarrow k[x, dx]
  \\
  x
  &\longmapsto dx
  \,.
\end{align*}
The tangent morphism of a cubical $1$-form given by $\omega(x) = \sum_i a_i db_i$ is
\begin{align*}
  T\omega: k[x,dx]
  &\longrightarrow \Sym_A(\Omega^{(2)}_{A/k})
  \\
  x &\longmapsto \sum_i a_i\, d_2 b_i
  \\
  dx &\longmapsto \sum_i (d_1 a_i)(d_2 b_i) + a_i\, (d_{12} b_i)
  \,.
\end{align*}
Now we can compute
\begin{equation*}
\begin{split}
  (\lambda_A \circ T\omega \circ \Val_{k[x]})(x)
  &=
  (\lambda_A \circ T\omega)(dx)
  \\
  &= \lambda_A\Bigl(
  \sum_i (d_1 a_i)(d_2 b_i) + a_i\, (d_{12} b_i) \Bigr)
  \\
  &=
  \sum_i a_i\, db_i
  \\
  &= \omega(x)
  \,,
\end{split}
\end{equation*}
which shows that $\omega$ is differential.
\end{proof}

With Lemma~\ref{lem:Schemes1} we conclude that the differential cubical 1-forms on an affine scheme are the K\"ahler differentials,
\begin{equation*}
  \Omega^1_\mathrm{cub}\bigl( \Spec(A) \bigr)
  \cong
  \Omega^1_{A/k}
  \,.
\end{equation*}
By contrast, the Lie-Rinehart $1$-forms are given by
\begin{equation*}
  \Omega^1_\mathrm{LR}\bigl( \Spec(A) \bigr)
  =
  \Hom_A\bigl( \Der_k(A), A \bigr)
  \cong
  \Hom_A\bigl( \Hom_A(\Omega^1_{A/k}, A), A \bigr)
  \,.
\end{equation*}
The natural map
\begin{equation*}
  \Omega^1_{A/k} \longrightarrow
  \Hom_A\bigl( \Hom_A(\Omega^1_{A/k}, A), A \bigr)
\end{equation*}
from the $A$-module of K\"ahler differentials to its bidual is generally not an isomorphism. (A sufficient condition is that $A$ is smooth over $k$, which implies that $\Omega^1_{A/k}$ is finitely generated projective \cite[\S 7.7]{AintablianBlohmann:2026}.) This shows that differential cubical forms and Lie-Rinehart forms on affine schemes are generally different.

\begin{Proposition}
\label{prop:SchemeCubKaehler}
Assume that $\mathrm{char}(k) \neq 2$. Then the $A$-module $\Omega^n_{A/k} \coloneqq \wedge^n_A \Omega^1_{A/k}$ of K\"ahler $n$-forms on $\Spec(A)$ is naturally isomorphic to the $A$-module of differential cubical $n$-forms,
\begin{equation*}
  \Omega_\mathrm{cub}^n \bigl( \Spec(A) \bigr)
  \cong
  \Omega^n_{A/k}
  \,,
\end{equation*}
for all $n \geq 0$.
\end{Proposition}

\begin{proof}[Sketch of proof]
Let $\Omega^{(n)}_{A/k}$ denote the $A$-module generated by the symbols $d_I a$, for all subsets $I \subset \{ 1, \ldots, n\}$, where $d_\emptyset a \coloneqq a$, subject to the relations~\eqref{eq:Schemes1} for all singletons and the relations
\begin{equation*}
  d_I (d_J a)
  =
  \begin{cases}
    d_{I \cup J}\, a &; I \cap J = \emptyset \\
    0 &; \text{otherwise.}
  \end{cases}
\end{equation*}
Then
\begin{equation*}
  T^n \Spec(A) =  \Spec\bigl( \Sym_A(\Omega^{(n)}_{A/k}) \bigr)
  \,.
\end{equation*}
The action of the symmetric group on $\Omega^{(n)}_{A/k}$ is defined by
\begin{equation*}
  \sigma (d_I a) = d_{\sigma(I)} a
  \,,
\end{equation*}
for all $\sigma \in S_n$ and $a \in A$. This extends to a unique automorphism of the algebra $\Sym_A(\Omega^{(n)}_{A/k})$. Consider an antisymmetric morphism $T^n \Spec(A) \to \bbA_k$ given by a morphism of algebras
\begin{equation*}
  \omega: k[x] \longrightarrow \Sym_A(\Omega^{(n)}_{A/k})
  \,.
\end{equation*}
For every subset $I \subset \{1, \ldots, n\}$ with $|I| > 1$ there is a transposition $\tau_{ij} \in S_n$ that swaps two indices $i,j \in I$, so that $d_{\tau_{ij}(I)} a = d_I a$. Let $\omega(x) = d_I a$. Since $\omega$ is antisymmetric, $\tau_{ij} \circ \omega = - \omega$, so $\tau_{ij} (d_I a) = d_I a = - d_I a$. Since $A$ is an algebra over a field of characteristic $\neq 2$, it follows that $d_I a = 0$. In general, $\omega(x)$ cannot contain a summand with a factor $d_I a$, $|I| > 1$.

Moreover, $\omega$ is linear at the $i$-th position if and only if every term contains exactly one factor of the form $d_i a$. We conclude that every cubical $n$-form can be written as a finite sum
\begin{equation*}
  \omega(x) = \sum_i a_i (d_1 b_{i,1})(d_2 b_{i,2}) \cdots (d_n b_{i,n})
  \,.
\end{equation*}
The symmetric group acts by
\begin{equation*}
\begin{split}
  (\sigma\circ \omega)(x)
  &= \sum_i a_i (d_{\sigma(1)} b_{i,1}) \cdots (d_{\sigma(n)} b_{i,n})
  \\
  &= \sum_i a_i (d_1 b_{i,\sigma^{-1}(1)}) \cdots
  (d_n b_{i,\sigma^{-1}(n)})
  \,,
\end{split}
\end{equation*}
for all $\sigma \in S_n$. Since $\omega$ is antisymmetric, the summands must be antisymmetric in $b_{i,1}, \ldots, b_{i,n}$. This shows that $\omega(x)$ can be identified with an element of $\Omega^n_{A/k}$. Finally, it can be shown that every cubical $n$-form is differential by generalizing the proof of Lemma~\ref{lem:Schemes1}. This finishes the sketch of the proof.
\end{proof}

It is straightforward to show that the isomorphism of Proposition~\ref{prop:SchemeCubKaehler} intertwines the differentials and inner derivatives of the Cartan calculi of cubical and K\"ahler forms. This leads to the following statement.

\begin{Proposition}
\label{prop:SchemeCubCartan}
The Cartan calculus of differential cubical forms on affine schemes over fields of characteristic $\neq 2$ is isomorphic to the Cartan calculus of the algebraic de Rham complex $(\Omega^\bullet_{A/k}, d)$.
\end{Proposition}

\subsection{Elastic diffeological spaces}
\label{sec:Elastic}

The site of euclidean spaces $\Eucl$ is the category of open subsets of all $\bbR^k$, $k \geq 1$, and smooth maps as morphisms, equipped with the Grothendieck topology of open covers. The category of diffeological spaces $\Dflg$ is the category of concrete sheaves on $\Eucl$. The category of elastic diffeological spaces \cite{Blohmann:2024a} is a full subcategory of $\Dflg$ with a tangent structure given by the pointwise left Kan extension of the tangent structure on $\Eucl$ (Section~\ref{sec:SmoothMflds}).

The tangent functor on a diffeological space $X$ is the pointwise left Kan extension of the tangent functor on euclidean spaces,
\begin{equation*}
  TX \coloneqq \Colim_{U \to X} TU
  \,,
\end{equation*}
where the colimit is over the category of plots of $X$ and taken in $\Dflg$. Similarly, we can left Kan extend the contravariant functor that maps an open subset $U \subset \bbR^k$ to the differential graded algebra $\Omega_\mathrm{dR}(U)$ of de Rham forms on $U$,
\begin{equation*}
  \Omega_\mathrm{dR}(X) \coloneqq \Lim_{U \to X} \Omega_\mathrm{dR}(U)
  \,,
\end{equation*}
which is the differential graded algebra of de Rham forms on $X$ \cite[\S 6.28]{IglesiasZemmour:2013}.

The key property of an \emph{elastic} diffeological space is that the fiber products of the tangent bundle and the powers of the tangent functor commute with the colimit of the left Kan extension,
\begin{equation*}
  T_n ( \Colim_{U \to X} U ) = \Colim_{U \to X} T_n U
  \,,\quad
  T^n ( \Colim_{U \to X} U ) = \Colim_{U \to X} T^n U
  \,,
\end{equation*}
where the colimits are taken over the category of plots of $X$. The natural transformations of the tangent structure are given by the induced maps on these colimits. For example, the fiberwise addition $+_X: T_2 X \to TX$ is induced by the addition $+_U: T_2 U \to TU$ on each plot and the universal property of the colimit,
\begin{equation*}
  +_X: T_2 X \cong \Colim_{U \to X} T_2 U
 \xrightarrow{\Colim_U +_U} \Colim_{U \to X} TU \cong TX
  \,.
\end{equation*}
The other natural transformations are induced in a similar way.

Now we can describe cubical forms. Since
\begin{equation*}
\begin{split}
  \Hom(T^n X, \bbR)
  &= \Hom\bigl( \Colim_{p:U \to X} T^n U, \bbR \bigr)
  \\
  &= \Lim_{p:U \to X} \Hom( T^n U, \bbR)
  \,,
\end{split}
\end{equation*}
by elasticity, a map $\omega: T^n X \to \bbR$ is given by a compatible family of maps $\omega_p: T^n U \to \bbR$ for each plot $p: U \to X$. The condition~\eqref{diag:CubForm1and2} that $\omega$ is additive at the first position is the commutativity of the diagram
\begin{equation*}
\begin{tikzcd}[column sep=6em]
  \mathllap{T_2 T^{n-1} X \cong} \Colim\limits_{p:U \to X} T_2 T^{n-1} U
  \ar[r, "{\Colim_p (\omega_p \times \omega_p)}"]
  \ar[d, "{\Colim\limits_{p:U \to X} (+_{T^{n-1} U})}"']
  &
  \bbR \times \bbR
  \arrow[d, "\Radd"] \\
  \mathllap{T^n X \cong} \Colim\limits_{p:U \to X} T^n U
  \ar[r, "\Colim_p \omega_p"']
  &
  \bbR
\end{tikzcd}
  \,.
\end{equation*}
It follows from the universal property of the colimit that $\omega$ is additive if and only if $\omega_p$ is additive for every plot $p$. Similarly, $\omega$ is $\bbR$-equivariant if and only if $\omega_p$ is $\bbR$-equivariant for every $p$. By an analogous argument, we can show that $\omega$ is differential if and only if $\omega_p$ is differential for every $p$. This shows that a differential cubical form on $X$ is a family of differential cubical forms $\omega_p$ representing an element of the left Kan extension
\begin{equation*}
  \Omega_\mathrm{cub}(X) \cong \Lim_{p:U \to X} \Omega_\mathrm{cub}(U)
  \,.
\end{equation*}
In Section~\ref{sec:SmoothMflds}, we have argued that on $U \subset \bbR^k$ the differential cubical forms are isomorphic to the de Rham forms, $\Omega_\mathrm{cub}(U) \cong \Omega_\mathrm{dR}(U)$. We arrive at the following result.

\begin{Proposition}
\label{prop:cubElastic}
The differential cubical forms on an elastic diffeological space $X$ are isomorphic to the de Rham forms,
\begin{equation*}
  \Omega_\mathrm{cub}(X) \cong \Omega_\mathrm{dR}(X)
  \,.
\end{equation*}
\end{Proposition}

Since the symmetric structure on $\Dflg$ is also obtained by left Kan extension, the ring structure of $\Omega_\mathrm{cub}(X)$ is the left Kan extension of the ring structure of $\Omega_\mathrm{dR}(U)$ for all plots. Similarly, the differential $d$ on cubical forms is the left Kan extension of the de Rham differential on plots. We conclude that the isomorphism of Proposition~\ref{prop:cubElastic} is an isomorphism of differential graded algebras.

The assignment $U \mapsto \calX(U)$ of the Lie algebra of vector fields is not functorial on the category of plots. Therefore, the vector fields $\calX(X)$ on a diffeological space $X$ and their inner derivatives cannot be obtained by left Kan extension. The Cartan calculus on $\Omega_\mathrm{cub}(X)$, and hence on $\Omega_\mathrm{dR}(X)$ by Proposition~\ref{prop:cubElastic}, is a genuinely new structure that arises from the theory of differential cubical forms on cartesian tangent categories with scalar multiplication.

\bibliographystyle{alpha}
\bibliography{Cartan}

\end{document}